\documentclass[letterpaper, reqno,11pt]{article}
\usepackage[margin=1.0in]{geometry}

\usepackage{color,latexsym,amsmath,amssymb, amsthm, graphicx,subfigure,revsymb, enumerate,xcolor,mathtools}
\usepackage{color,latexsym,amsmath,amssymb, amsthm, graphicx,subfigure,revsymb, enumerate,xcolor,mathtools}
\usepackage[super]{nth}

\usepackage{amsmath}\allowdisplaybreaks 
\numberwithin{equation}{section}

\renewcommand{\dim}{\mathrm{dim}\,}

\newcommand{\RR}{\mathbb{R}}
\newcommand{\CC}{\mathbb{C}}
\newcommand{\ZZ}{\mathbb{Z}}

\newcommand{\eps}{\varepsilon}
\newcommand{\itemizeEqnVSpacing}{\rule{0pt}{1pt}\vspace*{-12pt}}

\newcommand{\poly}{\operatorname{Poly}}

\newtheorem{thm}{Theorem}
\numberwithin{thm}{section}

\newtheorem{prop}[thm]{Proposition}
\newtheorem{defn}[thm]{Definition}
\newtheorem{lem}[thm]{Lemma}

\newtheorem*{maximalFnBdNonconcentratedInformal}{Proposition \ref{maximalFnBdNonconcentrated}, informal version}
\newtheorem*{RectBoundThmInformal}{Theorem \ref{RectBoundThm}, informal version}
\newtheorem*{mainThmTechnical}{Theorem \ref{mainThm}$'$}

\theoremstyle{remark}
\newtheorem{rem}[thm]{Remark}
\newtheorem{example}[thm]{Example}

\newtheorem*{example1Contd}{Example \ref{runningExampleMomentCurve} continued}

\title{On Maximal Functions Associated to Families of Curves in the Plane}

\author{Joshua Zahl}
\date{\today}

\begin{document}
\maketitle

\begin{abstract}
We consider the $L^p$ mapping properties of maximal averages associated to families of curves, and thickened curves, in the plane. These include the (planar) Kakeya maximal function, the circular maximal functions of Wolff and Bourgain, and their multi-parameter analogues. We propose a framework that allows for a unified study of such maximal functions, and prove sharp $L^p\to L^p$ operator bounds in this setting. A key ingredient is an estimate from discretized incidence geometry that controls the number of higher order approximate tangencies spanned by a collection of plane curves. We discuss applications to the F\"assler-Orponen restricted projection problem, and the dimension of Furstenberg-type sets associated to families of curves.
\end{abstract}


\section{Introduction}
In this paper, we study the $L^p$ mapping properties of maximal functions associated to families of curves in the plane. The prototypical example is the (planar) Kakeya maximal function
\begin{equation}\label{KakeyaMaximalOperator}
K_\delta f (e) = \frac{1}{\delta} \sup_{\ell |\!| e}\int_{\ell^\delta}|f|,\quad e\in S^1.
\end{equation}
In the above expression, $\delta>0$ is a small parameter; the supremum is taken over all unit line segments $\ell$ parallel to the vector $e$; and $\ell^\delta$ denotes the $\delta$ neighborhood of $\ell$. Cordoba \cite{cordoba} obtained the estimate $\Vert K_\delta f\Vert_p \leq C (\log 1/\delta)^{1/p} \Vert f\Vert_{p}$ for $p\geq 2$. The range of Lebesgue exponents is sharp, and the dependence of the operator norm on $\delta$ is also best possible (up to the choice of constant $C$). In particular, the existence of measure zero Besicovitch sets (compact sets in the plane that contain a unit line segment pointing in every direction) shows that for $p<\infty$ the operator $K_\delta$ cannot be bounded in $L^p$ with operator norm independent of $\delta$.

A second Kakeya-type maximal function was introduced by Wolff \cite{WolffKakeyacircles}. Let $C^\delta(x,y,r)$ denote the $\delta$ neighborhood of the circle centered at $(x,y)$ of radius $r$, and define
\begin{equation}\label{WolffMaximalFn}
W_\delta f (r) = \frac{1}{\delta} \sup_{(x,y)\in\RR^2}\int_{C^\delta(x,y,r)}|f|,\quad r\in [1,2].
\end{equation}
Wolff \cite{WolffKakeyacircles} obtained the estimate $\Vert W_\delta f\Vert_p \leq C_\eps\delta^{-\eps} \Vert f\Vert_{p}$ for $p\geq 3$. The range of Lebesgue exponents is sharp, and the existence of measure zero Besicovitch-Rado-Kinney sets (compact sets in the plane that contain a circle of every radius $r\in[1,2]$) shows that for $p<\infty$, the operator $W_\delta$ cannot be bounded in $L^p$ with operator norm independent of $\delta$.

A second class of maximal functions contains the Bourgain circular maximal function and its generalizations. For $(x,y)\in\RR^2$, let
\begin{equation}\label{bourgainMaxmlFn}
\begin{split}
Bf(x,y) = \sup_{1\leq r\leq 2}\int_{C(x,y,r)}|f|.
\end{split}
\end{equation}
Bourgain \cite{Bour1, Bour2} proved that $B$ is bounded from $L^p\to L^p$ for $p>2$. This is the sharp range of Lebesgue exponents for $L^p\to L^p$ bounds (the full range of exponents for which $B$ is bounded from $L^p\to L^q$ is more complicated; see \cite{Schlag97, Lee03} for details). As a consequence, if $K\subset\RR^2$ has positive measure and if $X\subset\RR^2$ contains a circle centered at every point of $K$, then $|X|>0$, i.e.~there are no analogues of measure-zero Besicovitch sets or Besicovitch-Rado-Kinney sets in this setting. 

Finally, we recall the Erdo\smash{\u{g}}an elliptic maximal function

\begin{equation}
\begin{split}
Ef(x,y) = \sup_{W}\int_{W}|f|,
\end{split}
\end{equation}
where the supremum is taken over all ellipses centered at $(x,y)$ whose semi-major and semi-minor axes have lengths in $[1/2, 2]$. 
This is a multi-parameter generalization of the Bourgain circular maximal function. Erdo\smash{\u{g}}an \cite{Erd} conjectured that $E$ should be bounded from $L^p\to L^p$ for $p>4$. Prior to this work, the best-known bound was $p>12$ by Lee, Lee, and Oh \cite{LLO}.

\subsection{The Setup}\label{setupSection}
The above maximal functions can be described as follows: We have a family of plane curves $\mathcal{C}$ (e.g.~lines, circles, ellipses) and a projection $\Phi\colon\mathcal{C}\to\RR^d$ (e.g.~the map sending a line to its slope, a circle to its radius, a circle to its center, etc.). For each $z\in\RR^d$, the maximal function $Mf(z)$ is a maximal average of $f$ taken over all (possibly thickened) curves $\gamma\in \mathcal{C}$ with $\Phi(\gamma)=z$; this is a subvariety of $\mathcal{C}$ of codimension $d$. 

The above maximal functions exhibit two phenomena. First, when $d=1$, we have examples of measure zero Besicovitch-type sets (and hence no operator norm bounds that are independent of $\delta$), while for $d>1$ we have not seen such examples. Second, the dimension of the fibers $\Phi^{-1}(z)$ determines the range of Lebesgue exponents for which $L^p\to L^p$ bounds can hold. 

Our first task is to describe the family of curves associated to our maximal function. It will be convenient to describe such curves as the graphs of functions.
\begin{defn}\label{paramFamily}
Let $\mathcal{C}$ be an $m$-dimensional manifold and let $I\subset\RR$ be an interval. Let $h\colon \mathcal{C}\times I\to\RR$ and define 
\[
F^h_t(u)=\big(h(u;t),\ \partial_t h(u;t),\ldots,\partial_t^{m-1} h(u;t)\big).
\] 
We say that $h$ \emph{parameterizes an $m$-dimensional family of cinematic curves} if $F^h_t\colon \mathcal{C}\to\RR^m$ is a local diffeomorphism for each $t\in I$.
\end{defn}
\begin{example}\label{runningExampleMomentCurve}
Let $\mathcal{C} =  \RR^m$, $I = [0,1],$ and $h(u,t) =  (1, t, t^2,\ldots, t^{m-1})\cdot u.$ Then $h$ parameterizes an $m$-dimensional family of cinematic curves.
\end{example}
Next we discuss a transversality condition that controls the behavior of the fibers $\Phi^{-1}(z)$. Continuing with our setup above, let $1\leq s < m$. For $(u,t)\in \mathcal{C}\times I$, define
\begin{equation}\label{manifoldVut}
V_{u;t} = \{u'\in \mathcal{C}\colon \partial_t^j h(u';t)=\partial_t^j h(u;t),\quad j = 0,\ldots, s\}.
\end{equation}
The restriction of $V_{u;t}$ to a small neighborhood of $u$ is a $(m-s-1)$-dimensional manifold. 
\begin{defn}\label{defnQuantTranverse}
We say a smooth submersion $\Phi\colon \mathcal{C} \to\RR^{m-s}$ is \emph{transverse to $h$} if for each $(u,t)\in \mathcal{C}\times I$, the derivative of $\Phi|_{V_{u;t}}$ has maximal rank (i.e.~rank $m-s-1$) at $u$. Note that this condition is vacuously satisfied if $s=m-1$. 
\end{defn}
\begin{example1Contd}
Continuing Example \ref{runningExampleMomentCurve}, let $\Phi\colon \RR^m \to\RR^{m-s}$ be the projection to the last $m-s$ coordinates. For $(u,t)\in \RR^m\times [0,1]$, we have
\[
V_{u;t} = \{u'\in \RR^m \colon u_j'=u_j,\ j=0,\ldots s\}=(u_0,u_1,\ldots,u_s)\times\RR^{m-s-1}.
\]
If $s=m-1,$ then $V_{u;t} = \{u\}$, and there is nothing to show. If $s < m-1$, then $\Phi|_{V_{u;t}}$ is given by $\Phi(u_0,u_1,\ldots,u_s, u_{s+1}',\ldots,u_{m-1}') = (u_{s+1}',\ldots,u_{m-1}')$, from which it is clear that the derivative of $\Phi|_{V_{u;t}}$ has maximal rank. We conclude that in this example, $\Phi$ is transverse to $h$.
\end{example1Contd}

With these definitions, we can now describe our class of maximal functions. 

\begin{defn}\label{defnClassOfMaximalOperators}
Let $1\leq s<m$, let $h\colon \mathcal{C}\times I \to\RR$ parameterize an $m$-dimensional family of cinematic curves, and let $\Phi\colon \mathcal{C} \to\RR^{m-s}$ be transverse to $h$. Fix a compact set $\mathcal{C}_0\subset \mathcal{C}$ and a compact interval $I_0\subset I$. Abusing notation, we restrict $h$ and $\Phi$ to $\mathcal{C}_0\times I_0$ and $\mathcal{C}_0$, respectively. For each $u\in\mathcal{C}_0$, define the curve
\[
\gamma_u = \big\{\big(t, h(u;t)\big)\colon t\in I_0\big\},
\] 
and let $\gamma_u^\delta$ denote the $\delta$ neighborhood of $\gamma_u$. We define the maximal functions $M_\delta$ and $M$ by 
\begin{align}
M_\delta f(v) & = \frac{1}{\delta} \sup_{u \in \Phi^{-1}(v)}\Big| \int_{\gamma_u^\delta} f\Big|,\label{defnMaximalFnDelta}\\
Mf(v) & =  \sup_{u \in \Phi^{-1}(v)}\Big|\int_{\gamma_u} f\Big|.\label{defnMaximalFnNoDelta}
\end{align}
We call these  \emph{$s$-parameter maximal functions associated to an $m$-dimensional family of cinematic curves}. These maximal functions are defined on the compact set $\Phi(\mathcal{C}_0)$.
\end{defn}

We remark that the $L^p$ mapping properties of these operators remain unchanged if we replace the integrand $f$ by $|f|$, but for technical reasons (see Section \ref{bourgainIntroSection}) we adopt the formulation above. The Kakeya, Wolff, Bourgain, and Erdo\smash{\u{g}}an maximal functions can be re-written in the above framework, with $(m,s)$ equal to $(2,1), (3,2)$, $(3,1),$ $(5,3)$, respectively. This is a straightforward computation, which is described in Appendix \ref{examplesSection}.

\subsection{Kakeya-type maximal functions}
Our main result is a sharp $L^p\to L^p$ bound for the Kakeya-type maximal function $M_\delta$.

\begin{thm}\label{mainThmMaximalKakeya}
Let $m>s\geq 1$ be integers, and let $M_\delta$ be an $s$-parameter maximal function associated to an $m$-dimensional family of cinematic curves. Let $\eps>0$. Then for all $\delta>0$ sufficiently small, we have

\begin{equation}\label{maximalFnBdMDelta}
\Vert M_\delta f\Vert_p \leq \delta^{-\eps}\Vert f\Vert_p,\quad p\geq s+1. 
\end{equation}
\end{thm}

Previous work in this setting has focused on the cases $m=2,s=1$ \cite{cordoba}; $m=3,s=1$ \cite{Stein, Bour2}; and $m=3,s=2$ \cite{KolasaWolff, PYZ, WolffKakeyacircles, Zahl2012b, Zahl2012}. The most interesting case is when $s=m-1$; the case $s<m-1$ can be reduced to $s=m-1$ by slicing (Definition \ref{defnQuantTranverse} is precisely the transversality condition needed to apply a slicing argument). The stated range of $p$ in \eqref{maximalFnBdMDelta} is sharp. This can be seen by selecting $\mathcal{C},h,$ and $\Phi$ as in Example \ref{runningExampleMomentCurve}, and selecting $f$ to be the characteristic function of the Knapp rectangle $[0, \delta^{1/s}]\times[0,\delta]$.

When $s=m-1$, the existence of measure-zero Besicovitch sets shows that for $p<\infty$ the operator $M_\delta$ cannot in general be bounded in $L^p$ with operator norm independent of $\delta$. This can be seen by choosing $\mathcal{C}$ and $h$ as in Example \ref{runningExampleMomentCurve}; $\Phi(u_0,u_1,\ldots,u_{m-1}) = u_1$; and $f$ the characteristic function of the $\delta$--thickening of a measure-zero Besicovitch set. More generally, Besicovitch and Rado \cite{BesRad} describe a procedure for constructing a measure-zero set that contains a translated copy of every algebraic curve from a one-parameter family.

\subsection{Bourgain-type maximal functions}\label{bourgainIntroSection}
In certain circumstances, Theorem \ref{mainThmMaximalKakeya} can be used to obtain sharp $L^p\to L^p$ bounds for the maximal function $Mf$ from Definition \ref{defnClassOfMaximalOperators}.
\begin{defn}\label{highFrequencyDecayCondition}
For $f\colon\RR^2\to\CC$, let $P_kf$ denote the Littlewood-Paley projection to the frequency annulus of magnitude $\sim 2^k$. We say that a sublinear operator $M$ has \emph{high frequency decay} if there exists $p<\infty$ and $C,c>0$ so that
\begin{equation}\label{highFrequencyDecayEqn}
\Vert M (P_k f)\Vert_p < C 2^{-ck}\Vert f\Vert_p,\quad f\in L^p(\RR^2).
\end{equation}
\end{defn}
Bourgain \cite{Bour2} (see also \cite{Sogge}) observed that if a maximal function $M$ has high frequency decay, then the estimate \eqref{highFrequencyDecayEqn} can be interpolated with an estimate of the form \eqref{maximalFnBdMDelta} to obtain $L^p\to L^p$ operator norm bounds for $M$, for all $p$ strictly larger than the range in \eqref{maximalFnBdMDelta}. Bourgain \cite{Bour2} followed this strategy (with slightly different notation) to obtain sharp $L^p$ bounds for his circular maximal function, and Chen, Guo, and Yang \cite{CGYv2} followed this strategy to obtain sharp $L^p$ bounds for the axis-parallel elliptic maximal function (see also \cite{LLO} for previous results on this operator).

These maximal functions are translation invariant, in the sense that for each point $(x,y)\in\RR^2$, the operator is a maximal average over a fixed family of curves that have been translated to the point $(x,y)$. We formalize this as follows:
\begin{defn}\label{defnTransInvariantMf}
Let $M$ be an $s$-parameter maximal function associated to an $(s+2)$-dimensional family of cinematic curves. Let $h\colon \mathcal{C}\times I\to\RR$ and $\Phi\colon \mathcal{C} \to\RR^{2}$ be the associated parameterization and projection functions. We say that $M$ is \emph{translation invariant} if in a neighborhood of each point of $\mathcal{C}\times I$, we can choose local coordinates $u=(x,y,w_1,\ldots,w_s)$ so that $\Phi$ has the form $\Phi(u) = (x,y)$, and $h$ has the form $h(u; t) = g(w_1,\ldots,w_s; t-x) + y$.
\end{defn}
\noindent The Bourgain circular maximal function and the elliptic maximal function are translation invariant according to this definition. 

Lee, Lee, and Oh \cite{LLO} recently proved a sharp local smoothing estimate for the elliptic and axis-parallel elliptic maximal functions, and in doing so they showed that these maximal functions have high frequency decay. Shortly thereafter, Chen, Guo, and Yang \cite{CGYv2} proved that every translation invariant maximal function (in the sense of Definition \ref{defnTransInvariantMf}) has high frequency decay (their result uses slightly different notation and applies to a slightly modified form of the maximal function \eqref{defnMaximalFnNoDelta}; see Proposition \ref{CGYProp} and the surrounding discussion for a precise statement). The Lee-Lee-Oh and Chen-Guo-Yang result has the following consequence.
\begin{thm}\label{LpMaximalFnBd}
Let $s\geq 1$ be an integer and let $M$ be an $s$-parameter translation invariant maximal function associated to a $(s+2)$-dimensional family of cinematic curves. Then

\begin{equation}\label{maximalFnBdMNoDelta}
\Vert M f\Vert_p \leq C_p\Vert f\Vert_p,\quad p> s+1. 
\end{equation}
\end{thm}
The stated range of $p$ is sharp, as can be seen by modifying an example due to Schlag \cite{Schlag}; see Appendix \ref{rangeOfPIsSharp} for details. In particular, Theorem \ref{LpMaximalFnBd} resolves Erdo\smash{\u{g}}an's conjecture by showing that the elliptic maximal operator is bounded from $L^p\to L^p$ in the sharp range $p>4$. Previously, Lee, Lee, and Oh \cite{LLO} (in the elliptic and axis-parallel elliptic case) and Chen, Guo, and Yang \cite{CGYv2} (in the translation invariant case) proved a variant of Theorem \ref{LpMaximalFnBd} for $p>s(s+1)$.

We conjecture that when $m=s+2$, every maximal function of the form \eqref{defnMaximalFnNoDelta} has high frequency decay. This was proved by Sogge \cite{Sogge} when $s=1$. If true, such a result could be combined with Theorem \ref{mainThmMaximalKakeya} and a slicing argument (see Section \ref{proofOfThmMainThmMaximalKakeyaSection}) to yield the analogue of Theorem \ref{LpMaximalFnBd} for all $m\geq s+2$ and all $s$-parameter maximal functions associated to an $m$-dimensional family of cinematic curves.

\medskip

{\bf Added Dec 2024:} Recently, Chen, Guo, and Yang \cite{CGYv6} made progress towards the above conjecture. They proved that a natural class of $s$-parameter maximal operators associated to an $(s+2)$-dimensional family of cinematic functions have high frequency decay. Chen, Guo, and Yang \cite{CGYv6} also combined their new result with Theorem \ref{mainThmMaximalKakeya} to obtain a variant of Theorem \ref{LpMaximalFnBd} in which the transversality hypotheses in Definition \ref{defnQuantTranverse} are weakened, at the cost of weakening the range of $p$ in the estimate \eqref{maximalFnBdMNoDelta} to $p>s+2$. See \cite{CGYv6} for further details.

\medskip

It is natural to ask about analogues of Theorems \ref{mainThmMaximalKakeya} and \ref{LpMaximalFnBd} for curves in $\RR^n$, in the spirit of the helical maximal function and its generalizations \cite{BGHS, KSO, KSO2}. This appears to be rather difficult at present, since our proof of Theorem \ref{LpMaximalFnBd} uses Theorem \ref{mainThmMaximalKakeya}, and the latter is at least as difficult as the Kakeya maximal function conjecture, which is open in dimension 3 and higher.

\subsection{An $L^p$ estimate for collections of plane curves}

To prove Theorem \ref{mainThmMaximalKakeya}, we begin by establishing \eqref{maximalFnBdMDelta} when $s=m-1$. This is a consequence of a slightly more general maximal function estimate associated to collections of thickened curves in the plane. The setting is as follows.

\begin{defn}\label{forbidTangencyDefn}
We say that a set $\mathcal{F}\subset C^k(I)$ \emph{forbids $k$--th order tangency} if there exists a constant $c>0$ so that for all $f,g\in \mathcal{F}$, we have
\begin{equation}\label{cinematicFunctionCondition}
\inf_{t\in I} \sum_{i=0}^k |f^{(i)}(t)-g^{(i)}(t)| \geq c \Vert f-g\Vert_{C^k(I)}.
\end{equation}
\end{defn} 
\noindent Here and throughout, we define $\Vert f\Vert_{C^k(I)}=\sum_{i=0}^k \Vert f^{(i)}\Vert_{\infty}$.
\begin{example}\label{forbidTangency}
$\phantom{1}$\smallskip
\begin{enumerate}
    \item[(i)] On a compact interval, linear functions forbid 1st order tangency. More generally, polynomials of degree at most $k$ forbid $k$--th order tangency. 
    \item[(ii)] An $m$--dimensional family $\mathcal{C}$ of cinematic curves restricted to a sufficiently small compact set forbid $(m-1)$--st order tangency.
\end{enumerate}
\end{example}

Recall that a set $\mathcal{F}\subset C^\infty(I)$ is \emph{uniformly smooth} if $\sup_{f\in \mathcal{F}}\Vert f^{(i)}\Vert_\infty<\infty$ for each $i\geq 0$. The functions in Example \ref{forbidTangency}(ii) are uniformly smooth. The functions in Example \ref{forbidTangency}(i) are uniformly smooth if we restrict the coefficients to a bounded set. With this definition, we can now state the main technical result of the paper.

\begin{thm}\label{mainThm}
Let $k\geq 1$, let $I$ be a compact interval, and let $\mathcal{F}\subset C^\infty(I)$ be uniformly smooth and forbid $k$--th order tangency. Let $\eps>0$. Then the following is true for all $\delta>0$ sufficiently small. Let $F\subset\mathcal{F}$ satisfy the non-concentration condition
\begin{equation}\label{ballCondition}
\#(F\cap B_r)\leq r/\delta\quad\textrm{for all balls}\ B_r\subset C^k(I)\ \textrm{of radius}\ r.
\end{equation}
Then 
\begin{equation}\label{LpBound}
\Big\Vert \sum_{f\in F} \chi_{f^\delta} \Big\Vert_{\frac{k+1}{k}} \leq \delta^{-\eps}(\delta\#F)^{\frac{k}{k+1}},
\end{equation}
where $f^\delta$ is the $\delta$ neighborhood of the graph of $f$.
\end{thm}
The bound \eqref{LpBound} is a Kakeya-type estimate for families of curves that forbid $k$--th order tangency. The exponent $(k+1)/k$ is best possible, and the existence of measure zero Besicovitch sets shows that the $\delta^{-\eps}$ factor (or at least some quantity that becomes unbounded as $\delta\searrow 0$) is also necessary.

\begin{rem}
In the special case $k=2$, Pramanik, Yang, and the author \cite{PYZ} proved a stronger version of Theorem \ref{mainThm} in which the hypothesis $\mathcal{F}\subset C^\infty(I)$ is replaced by the weaker hypothesis $\mathcal{F}\subset C^2(I)$. Both the present paper and \cite{PYZ} reduce the estimate \eqref{LpBound} to a statement in (discretized) incidence geometry. This reduction is similar in spirit in both \cite{PYZ} and the present paper. The main technical innovation in \cite{PYZ} is a discretized incidence bound that controls the number of approximate tangencies determined by an arrangement of $C^2$ curves. This is done using tools from topological graph theory---the idea is that if two curves are (almost) tangent, then after a small perturbation these curves define an object called a lens. A result of Marcus and Tardos \cite{MarcusTardos} bounds the number of such lenses. The present paper takes a different approach and uses tools from real algebraic geometry and ODE to control the number of approximate higher order tangencies determined by an arrangement of smooth curves; this is similar in spirit to the ideas from \cite{Zahl2020}. The approach taken in this paper has the downside that it only works for smooth curves, but it has the benefit that it applies to higher order tangencies, and thus allows us to obtain the estimate \eqref{LpBound} for all $k$.
\end{rem}

We will prove a slightly more technical generalization of Theorem \ref{mainThm}, where the ball condition \eqref{ballCondition} is replaced by a Frostman-type condition, and the sets $f^\delta$ are replaced by subsets that satisfy a similar Frostman-type condition. This more technical generalization will be called Theorem \ref{mainThm}$'$. Theorem \ref{mainThm}$'$ implies Theorem \ref{mainThmMaximalKakeya} in the special case $s=m-1$. The result is also connected to questions in geometric measure theory. We discuss some of these connections below. 

\subsection{Applications to geometric measure theory}
\noindent \textbf{Restricted projections.}\ \ 
In \cite{KOV2021}, K\"{a}enm\"{a}ki, Orponen, and Venieri discovered a connection between maximal function estimates for families of plane curves, and Marstrand-type results for projections in a restricted set of directions; the latter question was first investigated by F\"assler and Orponen in \cite{FasslerOrponen}. Accordingly, Theorem \ref{mainThm} is closely related to the following Kaufman-type estimate for the restricted projection problem. In what follows, ``$\dim\!$'' refers to Hausdorff dimension.

\begin{thm}\label{KaufmanThm}
Let $\gamma\colon[0,1]\to\RR^n$ be smooth and satisfy the non-degeneracy condition 
    \begin{equation}\label{nonDegenCurveCondition}
        \det\big(\gamma(t),\gamma'(t),\ldots,\gamma^{(n-1)}(t)\big)\neq 0,\quad t\in [0,1].
    \end{equation}
Let $E\subset\RR^n$ be Borel and let $0\leq s\leq\min(\dim E, 1)$. Then 
    \begin{equation}\label{exceptionalSetEstimate}
        \dim\big\{ t\in [0,1]\colon \dim(\gamma(t)\cdot E)< s \big\}\leq s.
    \end{equation}
\end{thm}
We comment briefly on the history of this problem. In \cite{FasslerOrponen}, F\"assler and Orponen introduced the non-degeneracy condition \eqref{nonDegenCurveCondition}, and they conjectured that if a smooth curve $\gamma\colon [0,1]\to\RR^3$ satisfies \eqref{nonDegenCurveCondition}, then $\dim(\gamma(t)\cdot E)=\min(1, \dim E)$ for a.e.~$t$; they made partial progress towards this conjecture. In \cite{KOV2021}, K\"{a}enm\"{a}ki, Orponen, and Venieri used circle tangency bounds proved by Wolff to resolve this conjecture in the special case where $\gamma(t) = (1, t, t^2)$. In \cite{PYZ}, Pramanik, Yang, and the author used a more general curve tangency bound (corresponding to $k=2$) to prove a mild generalization of Theorem \ref{KaufmanThm} when $n=3$; the result in \cite{PYZ} only requires that the curve $\gamma$ be $C^2$. In \cite{GGM}, Gan, Guth, and Maldague proved an estimate in a similar spirit to \eqref{exceptionalSetEstimate} (sometimes referred to as a ``Falconer-type'' exceptional set estimate) using techniques related to decoupling. Finally, in \cite{GGW}, Gan, Guo, and Wang proved a Falconer-type exceptional set estimate for general $n$, again using decoupling inequalities. 

\medskip

\noindent \textbf{Furstenberg sets.}\ \ 
As noted above, a consequence of Cordoba's Kakeya maximal function bound is that Besicovitch sets in the plane must have Hausdorff dimension 2. Similarly, Wolff's circular maximal function bound implies that Besicovitch-Rado-Kinney sets must have Hausdorff dimension 2. Theorem \ref{mainThm} has a similar consequence; in fact a slightly stronger statement is true in the spirit of the Furstenberg set conjecture. We first define a Furstenberg set of curves.
\begin{defn}\label{furstenbergSetCurves}
Let $\alpha,\beta\geq 0$ and let $\mathcal{F}\subset C^k(I)$. We say a set $E\subset\RR^2$ is an $(\alpha,\beta)$ \emph{Furstenberg set of curves from} $\mathcal{F}$ if there is a set $F\subset\mathcal{F}$ with $\dim(F)\geq\beta$ (here ``$\dim\!$'' refers to Hausdorff dimension in the metric space $C^k(I)$) so that $\dim(\operatorname{graph}(f) \cap E)\geq \alpha$ for each $f \in F$.
\end{defn}

\begin{thm}\label{generalizedFurstenbergSetsThm}
Let $k\geq 1$, let $I$ be a compact interval, and let $\mathcal{F}\subset C^\infty(I)$ be uniformly smooth and forbid $k$--th order tangency. Let $0\leq \beta \leq \alpha \leq 1$. Then every $(\alpha,\beta)$ Furstenberg set of curves from $\mathcal{F}$ has Hausdorff dimension at least $\alpha+\beta$. 
\end{thm}

We comment briefly on the history of this problem. In \cite{WolffKakeya}, Wolff defined a class of Besicovitch-type sets, inspired by the work of Furstenberg \cite{Fur}, which he called Furstenberg sets. In brief, for $0\leq\alpha\leq 1$, an $\alpha$-Furstenberg set is a compact set $E\subset\RR^2$ with the property that for each direction $e\in S^1$, there is a line $\ell$ parallel to $e$ with $\dim(E\cap\ell)\geq\alpha$. Wolff proved that every set of this type must have dimension at least $\max\big\{2\alpha,\alpha+\frac{1}{2}\big\}$, and he constructed examples of such sets that have dimension $\frac{3\alpha}{2} + \frac{1}{2}$. He conjectured that the latter bound is sharp. This conjecture was recently proved by Ren and Wang \cite{RW}; see also \cite{OrpSh2}. In \cite{MR}, Molter and Rela introduced the related notion of an $(\alpha,\beta)$-Furstenberg set. In the plane, their definition coincides with Definition \ref{furstenbergSetCurves}, where $\mathcal{F}$ is the set of linear functions. See \cite{OrpSh, OrpSh2} and the references therein for a survey of the Furstenberg set problem, and \cite{HKLO} for variants in higher dimensions. 

Recently, F\"assler, Liu, and Orponen \cite{FLO} considered the analogous problem where lines are replaced by circles; they formulated the analogous definition of a Furstenberg set of circles, and they proved that if $0\leq \alpha\leq\beta\leq 1$, then every $(\alpha,\beta)$  Furstenberg set of circles must have dimension at least $\alpha+\beta$. Theorem \ref{generalizedFurstenbergSetsThm} generalizes the F\"assler-Liu-Orponen result from circles to a larger class of curves. Theorem \ref{generalizedFurstenbergSetsThm} is clearly sharp in the stated range $0\leq \beta \leq \alpha \leq 1$. When $\alpha<\beta$, it is not obvious what dimension bounds should hold for $(\alpha,\beta)$ Furstenberg sets of curves.


\subsection{Curve tangencies, and tangency rectangles}\label{introTangencySection}
The main input to Theorem \ref{mainThm} is a new estimate in discretized incidence geometry that controls the number of approximate higher-order tangencies spanned by a collection of plane curves; this is Theorem \ref{RectBoundThm} below. Theorem \ref{RectBoundThm} requires several technical definitions. We will give an informal explanation of these definitions and then state an informal version of Theorem \ref{RectBoundThm}.

A $(\delta;k)$ \emph{tangency rectangle} $R$ is the $\delta$ neighborhood of the graph of a function with $C^k$ norm at most 1, above an interval $I$ of length $\delta^{1/k}$ (we are abusing terminology slightly, since the set $R$ need not be rectangle in the usual geometric sense). We say that a function $f$ is \emph{tangent} to $R$ (denoted $f\sim R$) if $I$ is contained in the domain of $f,$ and the graph of $f$, restricted to $I$, is contained in $R$. If $F$ is a set of functions and $\mu\geq 1$, we say a tangency rectangle is $\mu$-\emph{rich} with respect to $F$ if it is tangent to at least $\mu$ functions $f\in F$. We say two $(\delta;k)$ tangency rectangles $R_1,R_2$ are comparable if they are contained in a common $(2^k\delta;k)$ tangency rectangle. Otherwise they are incomparable (the factor $2^k$ simplifies certain parts of the proof, but any constant larger than 1 would suffice).

Observe that if two functions $f_1,f_2\colon [0,1]\to\RR$ with $C^k$ norm at most 1 are both tangent to a common $(\delta;k)$ tangency rectangle $R$ above the interval $[a, a+\delta^{1/k}]$, then for all $t\in[0,1]$ we have 
\begin{equation}\label{f1mF2Bd}
|f_1(t)-f_2(t)|\lesssim |t-a|^k+\delta.
\end{equation}
We say that $R$ is \emph{broad} with respect to $F$ if for most pairs of functions $f_1,f_2\in F$ that are tangent to $R$, the inequality \eqref{f1mF2Bd} is almost tight, i.e.~there is a matching lower bound $|f_1(t)-f_2(t)|\gtrsim |t-a|^k$ for all $t\in[0,1]$. The precise definition of broadness involves additional quantifiers; see Definition \ref{robustlyBroadDefn} for details. With these (informal) definitions, we can now state an informal version of Theorem \ref{RectBoundThm}.

\begin{RectBoundThmInformal}
Let $k\geq 1,\mu\geq 2$ and let $\delta>0$. Let $F$ be a set of low degree polynomials, and let $\mathcal{R}$ be a set of pairwise incomparable $(\delta;k)$ tangency rectangles, each of which are $\mu$-rich and broad with respect to $F$. Provided $\delta>0$ is sufficiently small, we have
\begin{equation}\label{informalRectBd}
\#\mathcal{R}\leq\delta^{-\eps}\Big(\frac{\#F}{\mu}\Big)^{\frac{k+1}{k}}.
\end{equation}
\end{RectBoundThmInformal}

\begin{rem}\label{RectBoundThmInformalRemarks}
$\phantom{1}$
\smallskip
\begin{enumerate}

\item[(i)] The requirement that the tangency rectangles in $\mathcal{R}$ are broad (or some analogous requirement) is necessary. Without this assumption, we could construct a counter-example to Theorem \ref{RectBoundThm} as follows. Let $F$ be a set of functions with $\#F = \mu$, each of which is an infinitesimal perturbation of the same function $f_0$, and let $\mathcal{R}$ be a set of $\delta^{-1/k}$ pairwise incomparable tangency rectangles arranged along the graph of $f_0$. 

\item[(ii)] When $k=1$, the bound \eqref{informalRectBd} follows from double-counting triples $(f_1, f_2, R)$, where $f_1,f_2$ are functions whose graphs transversely intersect inside $R$. When $k=2$ and the graphs of the functions in $F$ are (arcs of) circles, a bilinear variant of \eqref{informalRectBd} was proved by Wolff \cite{Wolff2000} using techniques from computational geometry originating from \cite{CEGSW}. This was generalized by the author in \cite{Zahl2012} for more general curves (again with $k=2$). Recently, Pramanik, Yang, and the author \cite{PYZ} proved a variant of Theorem \ref{RectBoundThm} for $k=2$ that works for $C^2$ functions.

\item[(iii)] The exponent $\frac{k+1}{k}$ follows from the numerology inherent in the polynomial method. For $k=2$, there are at least three independent proofs of this same bound, using different techniques (see Remark \ref{RectBoundThmInformalRemarks}(ii) above). However, it is not clear whether the exponent $\frac{k+1}{k}$ in \eqref{informalRectBd} is sharp. For $k=2$ the current best construction comes from a sharp Szemer\'edi-Trotter construction, and yields a lower bound with exponent $\frac{4}{3}$. 

\end{enumerate}
\end{rem}


\subsection{Main ideas, and a sketch of the proof}\label{introProofSketch}
In this section, we sketch the proofs of Theorems \ref{mainThm} and \ref{RectBoundThm}. We begin with Theorem \ref{RectBoundThm}. For simplicity during this proof sketch, we will suppose that $\mu$ has size close to 1 (for example $\mu = \delta^{-\eps}$) and $I=[0,1]$. When writing or describing inequalities, we will ignore constants that are independent of $\delta$ and $\#F$. We will prove the result by induction on the cardinality of $F$.
The induction step proceeds as follows. For each curve $f\in F$, we consider the graph of the $(k-1)$--st order ``jet lift'' 
\[
\zeta_f = \big\{\big(t, f(t), f'(t), \ldots, f^{(k-1)}(t)\big)\colon t\in [0,1]\big\}\subset\RR^{k+1}.
\] 
For each tangency rectangle $R\in\mathcal{R}$, we consider the corresponding ``tangency prism'' $\hat R\subset \RR^{k+1},$ which is a (curvilinear) prism of dimensions roughly $\delta^{1/k}\times \delta^{k/k}\times\delta^{(k-1)/k}\times\ldots\times\delta^{1/k}$. This prism will be constructed in such a way so that if $f\in F$ is tangent to a tangency rectangle $R\in \mathcal{R}$, then $\zeta_f$ intersects $\hat R$ in a curve of length roughly $\delta^{1/k}$. If this happens then we say $\zeta_f$ is incident to $\hat R$.

We have transformed the problem of estimating the number of broad tangency rectangles in the plane into a problem about incidences between curves and tangency prisms in $\RR^{k+1}$. To attack the latter problem, we use the Guth-Katz polynomial partitioning theorem. Let $E$ be a large number, and let $Q\subset\RR[t, x_0,\ldots,x_{k-1}]$ be a polynomial of degree at most $E$, so that $\RR^{k+1}\backslash\{Q=0\}$ is a union of about $E^{k+1}$ ``cells'' (open connected regions), with the property that at most $(\#\mathcal{R}) E^{-k-1}$ prisms are contained in each cell (if a prism intersects more than one cell, it is not counted here). Using a variant of B\'ezout's theorem and the assumption that $\mathcal{F}$ is uniformly smooth, we can ensure that at most $(\#F) E^{-k}$ curves intersect a typical cell. 

Since each prism $\hat R$ is connected, it is either contained inside a cell, or it must intersect the partitioning hypersurface $\{Q=0\}$. Our argument now divides into two cases: If at least half of prisms are contained inside a cell, then we are in the ``cellular case.'' If at least half of the prisms intersect the partitioning hypersurface, then we are in the ``algebraic case.''

We handle the cellular case as follows. Using our induction hypothesis, we conclude that since a typical cell $\Omega$ intersects roughly $(\#F) E^{-k}$ curves, there are at most $\big((\#F)E^{-k}\big)^{\frac{k+1}{k}}=(\#F)^{\frac{k+1}{k}}E^{-k-1}$ tangency rectangles $R\in \mathcal{R}$ with $\hat R \subset \Omega$. Thus the total contribution from all of the cells is at most $E^{k+1}\cdot (\#F)^{\frac{k+1}{k}}E^{-k-1} = (\#F)^{\frac{k+1}{k}}$. With some care (and a slight weakening of exponents, which introduces the $\delta^{-\eps}$ factor in \eqref{informalRectBd}), the induction closes. It is this argument (and the associated numerology) that determines the shape of the bound \eqref{informalRectBd}. 

The ideas described above that are used to handle the cellular case are not new; they were inspired by similar arguments in \cite{Guth}. To handle the algebraic case, however, new ideas are needed. This is the main innovation in this paper. We now sketch the proof of the algebraic case. We begin with several simplifying assumptions. \emph{Simplifying Assumption (A)}: the surface $\{Q=0\}$ can be written as a graph $\{x_{k-1} = L(t, x_0,\ldots, x_{k-2})\}$, where $L$ is $1$--Lipschitz. As a consequence of Assumption (A), if a tangency prism $\hat R$ intersects $\{Q=0\}$, then $\hat R$ is contained in a thin neighborhood of the graph of $L$, i.e.~$\hat R \subset Z^*$, where
\begin{equation}\label{informalDefnZStar}
 Z^* = \big\{(t, x_0,\ldots, x_{k-1})\in[0,1]^{k+1}\colon |x_{k-1} - L(t, x_0,\ldots,x_{k-2})|\leq \delta^{1/k}\big\}.
\end{equation}
Next we make \emph{Simplifying Assumption (B)}: each curve $\zeta_f$ is contained in $Z^*$. This means that $f$ almost satisfies the ODE $f^{(k-1)}(t) = L\big(t, f(t), f'(t), \ldots, $ $f^{(k-2)}(t)\big)$. More precisely, we have 
\begin{equation}\label{informalODE}
\big|f^{(k-1)}(t) - L\big(t, f(t), f'(t), \ldots, f^{(k-2)}(t)\big)\big|\leq \delta^{1/k},\quad t\in [0,1].
\end{equation}

If $\zeta_f$ and $\zeta_g$ are both incident to a common prism $\hat R$, then a straightforward calculus exercise shows that there must exist some $t_0$ for which the first $k-1$ derivatives of $f$ and $g$ almost agree, in the sense that
\begin{equation}\label{informalIVP}
|f^{(i)}(t_0) - g^{(i)}(t_0)|\leq\delta^{1/k},\quad i=0,\ldots,k-1.
\end{equation}
Inequality \eqref{informalODE} (and its analogue for $g$) say that $f$ and $g$ almost satisfy the same ODE, and \eqref{informalIVP} says that $f$ and $g$ almost have the same initial conditions, and hence $f$ and $g$ almost satisfy the same initial value problem. Since $L$ is $1$--Lipschitz, we can use a quantitative version of Gronwall's inequality to conclude that $|f(t)-g(t)|$ is small for all $t\in [0,1]$. We conclude that all of the curves tangent to a common tangency rectangle $R\in\mathcal{R}$ must remain close for all time $t\in [0,1]$; but this contradicts the requirement that the tangency rectangles in $\mathcal{R}$ are broad. This implies $\mathcal{R}$ must be empty. Thus we have established Theorem \ref{RectBoundThm}, except that we have not yet justified Simplifying Assumptions (A) and (B). 

First, we will explain how to remove Simplifying Assumption (B); this is mostly a technical matter. While the curves $\zeta_f$ need not be contained in $Z^*$, each curve intersects $Z^*$ in a small number of curve segments, and the curve-prism incidences occur within these segments. Thus we can find a typical length $\ell$, so that most curve-prism incidences occur within segments that have length roughly $\ell$. After partitioning space into rectangular prisms of the appropriate dimensions and rescaling, we reduce to the case where $\ell=1$. 

Next, we will explain how to remove Simplifying Assumption (A); this issue is more serious. In general, we may suppose that each tangency prism is contained in the $\delta^{1/k}$ neighborhood of the variety $\{Q=0\}$. This is a semi-algebraic set, and after restricting to $[-1,1]^{k+1}$, this set has volume roughly $\delta^{1/k}$. We prove a new structure theorem which says that any semi-algebraic set in $[0,1]^{k+1}$ with small $(k+1)$-dimensional volume can be decomposed into a union of pieces, each of which is the thin neighborhood of a Lipschitz graph (with controlled Lipschitz constant), plus a final piece whose projection to the first $k$ coordinates has small $k$-dimensional volume. If the majority of prisms and curves are contained in one of the Lipschitz graph pieces, then (a slight weakening of) Simplifying Assumption (A) holds, and we can argue as above. If instead the majority of prisms and curves are contained in the final piece, then we project from $\RR^{k+1}$ to the first $k$--coordinates. The Tarski–Seidenberg theorem says that the image under this projection is a semi-algebraic subset of $[0,1]^k$, and thus we can apply the same decomposition again. After iterating this procedure at most $k$ times, we arrive at a situation where Simplifying Assumption (A) holds, and we can apply the arguments described above.

\medskip

\noindent {\bf From Tangency Rectangles to Maximal Functions} 

We now sketch the proof of Theorem \ref{mainThm}. The proof is complicated by the fact that the collection of curves $F$ can be arranged in many different ways. To begin, we will examine three specific arrangements that will give the reader a sense of the range of possibilities. For clarity when writing inequalities, we will ignore constants that are independent of $\delta$, and we will omit terms of the form $\delta^{-\eps}$. 

\medskip

\noindent {\it Arrangement 1.} Suppose that for a typical pair of functions $f,g\in F$ for which $f^\delta\cap g^\delta$ is non-empty, we have that the graphs of $f$ and $g$ intersect transversely. This means that $|f^\delta\cap g^\delta|$ typically has size about $\delta^2$, and thus we might expect
\begin{equation}\label{ElemL2Bd}
\Big\Vert \sum_{f\in F} \chi_{f^\delta} \Big\Vert_2 \leq \Big(\sum_{f,g\in F}|f^\delta\cap g^\delta|\Big)^{1/2}\leq \delta (\#F).
\end{equation}
On the other hand, we have 
\begin{equation}\label{L1BdIntro}
\Big\Vert \sum_{f\in F} \chi_{f^\delta} \Big\Vert_1 \leq \delta(\#F).
\end{equation}
Interpolating \eqref{ElemL2Bd} and \eqref{L1BdIntro} using H\"older, we obtain
\begin{equation}\label{arrangement1Bd}
\Big\Vert \sum_{f\in F} \chi_{f^\delta} \Big\Vert_{\frac{k+1}{k}}\leq \delta(\#F).
\end{equation}
Note that this is stronger than \eqref{LpBound}, since the ball condition \eqref{ballCondition} implies that $\delta(\#F)\lesssim 1$.

\medskip

\noindent {\it Arrangement 2.} Suppose that for a typical pair of functions $f,g\in F$ for which $f^\delta \cap g^\delta$ is non-empty, we have that the graphs of $f$ and $g$ are tangent to order $k-1$. This means that $|f^\delta\cap g^\delta|$ is a curvilinear rectangle of dimensions roughly $\delta\times\delta^{1/k}$. In this situation, we can find a number $2\leq\mu\leq\#F$ and a set $\mathcal{R}$ of $\mu$-rich, broad $(\delta;k)$ tangency rectangles, so that 
\begin{equation}\label{elemRectEstimateLp}
\Big\Vert \sum_{f\in F} \chi_{f^\delta} \Big\Vert_{\frac{k+1}{k}}^{\frac{k+1}{k}} 
\leq \sum_{R\in\mathcal{R} }\int_R \Big(\sum_{\substack{f\in F\\ f\sim R}}\chi_{f^\delta}\Big)^{\frac{k+1}{k}}.
\end{equation}
By Theorem \ref{RectBoundThm}, $\#\mathcal{R}\leq  \big(\frac{\#F}{\mu}\big)^{\frac{k+1}{k}},$ and the contribution from each $R\in\mathcal{R}$ to the RHS of \eqref{elemRectEstimateLp} is at most $(\mu\delta)^{\frac{k+1}{k}}$. Thus we again have the bound
\begin{equation}\label{arrangement2Bd}
\Big\Vert \sum_{f\in F} \chi_{f^\delta} \Big\Vert_{\frac{k+1}{k}} \leq \Big(\big(\frac{\#F}{\mu}\big)^{\frac{k+1}{k}} \cdot (\mu\delta)^{\frac{k+1}{k}}\Big)^{\frac{k}{k+1}}= \delta(\#F).
\end{equation}

\medskip

\noindent {\it Arrangement 3.} Suppose that $F=\{f\}$. Then
\[
\Big\Vert \sum_{f\in F} \chi_{f^\delta} \Big\Vert_{\frac{k+1}{k}} = \delta^{\frac{k}{k+1}} = \big(\delta \#F\big)^{\frac{k}{k+1}}.
\]

\medskip

Note that our bounds \eqref{arrangement1Bd} and \eqref{arrangement2Bd} for Arrangements 1 and 2 are stronger than the corresponding estimate \eqref{LpBound} from Theorem \ref{mainThm}. In this direction, we will first prove a variant of Theorem \ref{mainThm}, where the non-concentration condition \eqref{ballCondition} is replaced by a (local) two-ends type non-concentration condition on the set of curves passing through each point. This is Proposition \ref{maximalFnBdNonconcentrated} below. Informally, the statement is as follows
\begin{maximalFnBdNonconcentratedInformal}
Let $k\geq 1$ and let $\delta>0$. Let $F$ be a set of functions that come from a uniformly smooth family of curves. Suppose that for a typical point $x\in \RR^2$, a typical pair of curves from $F$ whose $\delta$ neighborhoods contain $x$ diverge from each other at speed at least $t^k$ in a neighborhood of $x$. Then

\begin{equation}\label{NonconcentratedPropLpBoundInformal}
\Big \Vert \sum_{f\in F}\chi_{f^\delta}\Big\Vert_{\frac{k+1}{k}}\leq \delta(\#F).
\end{equation}
\end{maximalFnBdNonconcentratedInformal}
Note that if \eqref{NonconcentratedPropLpBoundInformal} is established for some value of $k$, then the analogous result immediately follows for all larger $k$ by interpolation with the trivial $L^1$ estimate \eqref{L1BdIntro}. This observation will play an important role in the proof. We prove Proposition \ref{maximalFnBdNonconcentrated} by induction on $k$. In the inequalities that follow, we will ignore all constants that are independent of $\delta$, and all factors of the form $\delta^{-\eps}$.

The base case $k=1$ is essentially the estimate \eqref{ElemL2Bd}. For the induction step, we select the smallest $\rho\in[\delta,1]$ so that the intersection of a typical pair of curves is localized to a $\rho\times\rho^{1/k}$ curvilinear rectangle. This allows us to find a set $\mathcal{R}$ of $(\rho;k)$ tangency rectangles, each of which have roughly the same richness $\mu$, so that
\begin{equation}\label{LpNormBreaksOverRectangles}
\int \Big(\sum_{f\in F} \chi_{f^\delta} \Big)^{\frac{k+1}{k}} \leq \sum_{R\in\mathcal{R}}\int_R \Big(\sum_{\substack{f\in F \\ f \sim R}} \chi_{f^\delta}\Big)^{\frac{k+1}{k}}.
\end{equation}
Furthermore, the tangency rectangles in $\mathcal{R}$ are broad, and hence by Theorem \ref{RectBoundThm} we have $
\#\mathcal{R}\leq \big(\frac{\#F}{\mu}\big)^{\frac{k+1}{k}}.$
If $\rho$ has size roughly $\delta$, then we are in the situation of Arrangement 2 and we can immediately apply \eqref{arrangement2Bd}. If instead $\rho$ is substantially larger than $\delta$ (and hence $\delta/\rho$ is substantially smaller than $1$), then our definition of $\rho$ has the following consequence: If we rescale a tangency rectangle $R\in\mathcal{R}$ to the unit square, then the images of the functions $\{f\in F\colon f \sim R\}$ under this rescaling satisfy the hypothesis of (the informal version of) Proposition \ref{maximalFnBdNonconcentrated}, with $k-1$ in place of $k$. Denote the image of $f$ under this rescaling by $\tilde f$, and let $\tilde \delta=\delta/\rho$. Then we have 
\begin{equation}\label{introLpInterpolation}
\Vert h\Vert_{\frac{k+1}{k}}^{\frac{k+1}{k}}  
\leq \Vert h\Vert_{1}^{\frac{1}{k}} \Vert h\Vert_{\frac{k}{k-1}} 
\leq \Big(\tilde\delta \mu \Big)^{\frac{1}{k}}\Big( \tilde\delta\mu \Big)
=\big(\frac{\delta}{\rho}\mu\big)^{\frac{k+1}{k}},\qquad h = \sum_{\substack{f\in F \\ f \sim R}} \chi_{\tilde f^{\tilde \delta}}.
\end{equation}
In the above inequality, we use \eqref{L1BdIntro} to obtain an $L^1$ estimate, and we use the induction hypothesis to obtain an $L^{\frac{k}{k-1}}$ estimate. Note that the rescaling from $R$ to the unit square distorts volumes by a factor of $\rho^{1+1/k}$, and thus \eqref{introLpInterpolation} says that for each $R\in\mathcal{R}$ we have
\begin{equation}\label{boundInsideR}
\int_R \Big(\sum_{\substack{f\in F \\ f \sim R}} \chi_{f^\delta}\Big)^{\frac{k+1}{k}}\leq \big(\delta\mu)^{\frac{k+1}{k}}.
\end{equation}
Inserting the estimate \eqref{boundInsideR} into \eqref{LpNormBreaksOverRectangles} and using our bound on the size of $\mathcal{R}$, we obtain 
\[
\int \Big(\sum_{f\in F} \chi_{f^\delta} \Big)^{\frac{k+1}{k}}
\leq \sum_{R\in\mathcal{R}}\big(\delta\mu)^{\frac{k+1}{k}}
\leq (\delta\#F)^{\frac{k+1}{k}},
\]
which is \eqref{NonconcentratedPropLpBoundInformal}. This closes the induction. The details of this argument are discussed in Section \ref{tangenciesToMaximalFnSection}.

Finally, we remark that Arrangements 1 and 2 satisfy the hypotheses of Proposition \ref{maximalFnBdNonconcentrated}, and thus are amenable to the above argument. Arrangement 3 does not satisfy the hypotheses of Proposition \ref{maximalFnBdNonconcentrated}, and indeed the conclusion of Proposition \ref{maximalFnBdNonconcentrated} is false for Arrangement 3. The final step in the proof of Theorem \ref{mainThm} is to reduce an arbitrary arrangement of curves that forbid $k$--th order tangency and satisfy the non-concentration condition \eqref{ballCondition} to a collection of (rescaled) non-interacting sub-arrangement, each of which satisfies the hypotheses of Proposition \ref{maximalFnBdNonconcentrated}. This is a standard ``two-ends rescaling'' type argument.


\subsection{Paper organization}
In Sections \ref{tangencyRectSection} and \ref{tangenciesSemiSmallVolSec}, we execute the proof sketch described in Section \ref{introProofSketch} to prove Theorem \ref{RectBoundThm}. In Section \ref{tangenciesToMaximalFnSection}, we continue following the proof sketch to show how Theorem  \ref{RectBoundThm} implies Proposition \ref{maximalFnBdNonconcentrated}. The remaining Sections \ref{pfOfMainThmSec}, \ref{proofOfThmMainThmMaximalKakeyaSection}, \ref{proofOfTheoremLpMaximalFnBdSec}, and \ref{reductionMainThmsSection} are devoted to the proofs of Theorems \ref{mainThm}, \ref{mainThmMaximalKakeya}, \ref{LpMaximalFnBd}, and \ref{KaufmanThm} + \ref{generalizedFurstenbergSetsThm}, respectively.

\subsection{Thanks}
The author would like to thank Young-Heon Kim for helpful conversations and discussions about Gronwall's inequality, which helped shape Section \ref{jetLiftSection}. The author would like to thank Shaoming Guo for helpful conversations and discussions about local smoothing and its implications for maximal functions over curves, which helped shape Section \ref{proofOfTheoremLpMaximalFnBdSec}. The author would like to thank Mingfeng Chen, Jonathan Hickman, Sanghyuk Lee, Chenjian Wang, Tongou Yang, and the anonymous referees for suggestions and corrections to an earlier version of this manuscript. The author was supported by an NSERC Discovery grant and by the Nankai Zhide Foundation.


\subsection{Notation}
We use $A\lesssim B$ or $A = O(B)$ or $B = \Omega(A)$ to mean $A\leq KB$, where $K$ is a quantity that may depend on the parameter $k$ from the statement of Theorem \ref{mainThm}. Some of the intermediate results used to prove Theorem \ref{mainThm} will be stated in slightly greater generality; in these cases the implicit constant $K$ will also be allowed to depend on the ambient dimension, which will be denoted by $\RR^n$. When these intermediate results are applied to prove Theorem \ref{mainThm}, we will always have $n\leq k+1$. If $K$ is allowed to depend on an additional parameter $\eps$, then we denote this by $A\lesssim_\eps B$ or $A = O_\eps(B)$. Finally, if $A\lesssim B$ and $B\lesssim A$, then we write $A\sim B$. The notation $f\sim R$ is also used to denote tangency between functions and rectangles, but it will always be apparent from context which of these two meanings is intended.

Unless otherwise specified, all intervals will be assumed to be subsets of $[0,1]$, and all functions will be assumed to have domain $[0,1]$ and co-domain $\RR$. We abbreviate $C^k([0,1])$ as $C^k$, and $\Vert f \Vert_{C^k([0,1])}$ as $\Vert f \Vert_{C^k}.$


\section{Curves and tangency rectangles}\label{tangencyRectSection}
In this section we will state the precise version of Theorem \ref{RectBoundThm}, and begin the proof. We begin with precise versions of the informal definitions from Section \ref{introTangencySection}. In what follows, the vertical $\delta$ neighborhood of a set $S\subset\RR^2$ is the union of vertical line segments of length $2\delta$ centered at the points of $S$.

\begin{defn}\label{defnTangencyRectangle}
Let $\delta,T>0$ and $k\geq 1$. A $(\delta;k;T)$ \emph{tangency rectangle} is the (closed) vertical $\delta$ neighborhood of the graph of a function with $C^k$ norm at most 1, above a closed interval of length $(T\delta)^{1/k}$. When $T=1$, we abbreviate this to $(\delta;k)$ tangency rectangle, or $(\delta;k)$  rectangle.
\end{defn}

\begin{defn}
If $R$ is a $(\delta;k;T)$ tangency rectangle above an interval $I$, and $f\colon [0,1]\to\RR$, we say $f$ is \emph{tangent} to $R$ if the graph of $f$ above $I$ is contained in $R$. We denote this by $f\sim R$.
\end{defn}

Next, we will describe what it means for two tangency rectangles to be distinct.

\begin{defn}\label{comparableDefn}
We say two $(\delta;k;T)$ tangency rectangles are \emph{comparable} if there is a $(2^k\delta;k;T)$ rectangle that contains them both. Otherwise they are \emph{incomparable}.
\end{defn}
The factor $2^k$ in the above definition was chosen to make the following true: if $R_1,R_2$ are incomparable $(\delta; k; T)$ rectangles above intervals $I_1$ and $I_2$ respectively, and if $R_1$ and $R_2$ are both tangent to a common function $f$ with $C^k$ norm at most 1, then $I_1$ and $I_2$ are disjoint.

\begin{defn}\label{robustlyBroadDefn}
Let $R$ be a $(\delta;k)$ rectangle and let $F$ be a set of functions from $[0,1]$ to $\RR$. Let $\mu\geq 1$. We say that $R$ is $\mu$--\emph{rich} and $\eps$--\emph{robustly broad with error at most} $B$ (with respect to $F$) if there is a set $F(R)\subset \{f\in F\colon f\sim R\}$ with $\#F(R)\geq\mu$ that has the following property: For every $\rho\in [\delta,1]$, every $T\in [1, \rho^{-1}]$, and every $(\rho; k; T)$ rectangle $R'$ containing $R$, we have
\begin{equation}\label{fExtendedRectangleBound}
\#\{f\in F(R)\colon f \sim R'\} \leq B T^{-\eps}\#F(R).
\end{equation}
\end{defn}
\begin{rem}\label{caveatRemarkMu}
Note that if $R$ is $\mu$-rich and $\eps$--robustly broad with error at most $B$, then $\mu\geq B^{-1}(2\delta)^{-\eps}$. This is because we can choose $\rho=2\delta$, $T = (2\delta)^{-1}$, and choose $R'$ to be the vertical $2\delta$ neighbourhood of the graph of a function $f\in F(R)$. 
\end{rem}

During informal discussion, we will say that $R$ is \emph{robustly broad} if we do not wish to emphasize the role of $\mu,$ $\eps$, or $B$.

\begin{rem}
When $k=1$, a $(\delta;1)$ rectangle $R$ is robustly broad if many of the pairs $f_1,f_2\in F(R)$ have graphs that intersect transversely. If $k>1$, then all of the functions in $F(R)$ will intersect (almost) tangentially, but if $R$ is robustly broad then many pairs of functions will diverge outside of $R$ at speed roughly $t^k$---this is the fastest possible speed of divergence that is allowed by the geometry of $R$ and the constraint that the functions have $C^k$ norm at most 1. 
\end{rem}

With these definitions, we can now precisely state our incidence bound. 

\begin{thm}\label{RectBoundThm}
Let $k,N\geq 1$ and $\eps>0$. Then there exists $\eta,\delta_0>0$ so that the following holds for all $\delta\in(0,\delta_0]$. 

Let $F$ be a set of (univariate) polynomials of degree at most $\delta^{-\eta}$, each of which has $C^{k}$ norm at most 1. Suppose that $\#F\leq\delta^{-N}$. Let $\mu\geq 1$ and let $\mathcal{R}$ be a set of pairwise incomparable $(\delta,k)$ rectangles that are $\mu$-rich and $\eps$-robustly broad with error at most $\delta^{-\eta}$ with respect to $F$. Then
\begin{equation}\label{RectBdEqn}
\# \mathcal{R} \leq \delta^{-\eps} \Big(\frac{\# F}{\mu}\Big)^{\frac{k+1}{k}}.
\end{equation}
\end{thm}

\begin{rem}
The hypotheses of Theorem \ref{RectBoundThm} state that $\mu\geq 1$. However, the theorem is only interesting when $\mu\geq \delta^{\eta}(2\delta)^{-\eps}$. Indeed, if $1\leq \mu < \delta^{\eta}(2\delta)^{-\eps}$, then by Remark \ref{caveatRemarkMu} we automatically have $\mathcal{R} = \emptyset.$
\end{rem}


\subsection{Initial reductions}
Our proof of Theorem \ref{RectBoundThm} will follow the outline discussed in Section \ref{introProofSketch}. We begin by reducing Theorem \ref{RectBoundThm} to a version that is weaker in several respects. First, the hypotheses are strengthened: we only need to consider the case where $\mu$ has size roughly $\delta^{-\eps}$. Second, the conclusion is weakened: the exponent $\frac{k+1}{k}$ is weakened to $\frac{k+1}{k}+\eps$.

\begin{prop}\label{tangencyRectanglesAlgCurvesProp}
Let $k\geq 1$, $\eps>0$. Then there exist (large) constants $B=B(k)$ and $C= C(k,\eps)$ and a small constant $\eta=\eta(k,\eps)>0$ so that the following holds. Let $F$ be a set of polynomials of degree at most $\delta^{-\eta}$, each of which has $C^{k}$ norm at most 1. Let $\mathcal{R}$ be a set of pairwise incomparable $(\delta,k)$ rectangles, each of which are $\mu$-rich for some $\mu\geq 1$ (see Remark \ref{caveatRemarkMu}), and $\eps$ robustly broad with error at most $\delta^{-\eta}$ with respect to $F$. Then
\begin{equation}\label{RectBdEqnForPolysMuSmall}
\# \mathcal{R} \leq C  \delta^{-B \eps} (\# F)^{\frac{k+1}{k}+\eps}.
\end{equation}
\end{prop}

To reduce to the case where $\mu$ is small, we will refine the set $F$ by randomly keeping each element with probability roughly $\mu^{-1}$. To ensure that the resulting refinement satisfies the hypotheses of Proposition \ref{tangencyRectanglesAlgCurvesProp}, we will use the following special case of Chernoff’s inequality \cite{Cher}.

\begin{thm}[Chernoff]\label{ChernoffThm}
Let $X_1,\ldots,X_n$ be independent random variables taking value 1 with probability $p$ and value 0 with probability $1-p$. Let $X$ denote their sum. Let $A\geq 2$. Then

\begin{equation}\label{ChernoffBd}
\mathbb{P}\big(X \leq pn/2  \big)  < e^{\frac{-pn}{8}};\qquad \mathbb{P}\big(X \geq Apn \big)  < e^{\frac{-Apn}{6}}.
\end{equation}
\end{thm}

We can now explain the reduction from Proposition \ref{tangencyRectanglesAlgCurvesProp} to Theorem \ref{RectBoundThm}.

\begin{proof}[Proof that Proposition \ref{tangencyRectanglesAlgCurvesProp} implies Theorem \ref{RectBoundThm}]
Suppose that Proposition \ref{tangencyRectanglesAlgCurvesProp} is true. Let $k\geq 1$, $\eps>0$, $\delta>0$, $\mu\geq 1$, $F$, and $\mathcal{R}$ satisfy the hypotheses of Theorem \ref{RectBoundThm}. Our goal is to show that if $\eta>0$ and $\delta_0>0$ are selected appropriately (depending on $k$ and $\eps$), then \eqref{RectBdEqn} holds.

First, we may suppose that $\mu\leq\#F$, since otherwise $\mathcal{R}=\emptyset$ and we are done. 

\medskip

\noindent\textbf {Step 1: Random sampling}.
Let $\eps_1=\eps_1(k,\eps)>0$ be a small quantity to be chosen below; we will select $\eps_1$ small compared to $\eps$. If $\mu\leq\delta^{-2\eps_1}$, define $F'=F$ and proceed to the computation \eqref{finalCompRectBd} below. Otherwise, after dyadic pigeonholing the set $\mathcal{R}$ (this decreases the size of $\mathcal{R}$ by $O(\log\# F)\lesssim N\log\delta$) and increasing $\mu$ if necessary, we can suppose that for each $R\in\mathcal{R}$, there is a set $F(R)\subset\{f\in F\colon f\sim R\}$ that satisfies \eqref{fExtendedRectangleBound}, with $\mu\leq \#F(R) < 2\mu$. 

Let $p = (\delta^{2\eps_1}\mu)^{-1}$ (since $\mu>\delta^{-2\eps_1}$, we have $0<p<1$). Let $F'\subset F$ be obtained by randomly selecting each $f\in F$ with probability $p$. $F'$ has expected cardinality $p(\#F)\geq\delta^{-2\eps_1}$. 

\medskip

\noindent\textbf {Step 2: Robust broadness with respect to $F'$}.
We claim that if $\delta_0$ is chosen sufficiently small depending on $\eps_1$, then with probability at least $1/2$ the following is true:
\begin{itemize}
    \item $\#F' \leq 2p(\#F)=2\delta^{-2\eps_1}\mu^{-1}(\#F).$
    \item Each rectangle in $\mathcal{R}$ is $\frac{1}{4}p\mu$-rich and $\eps_1$-robustly broad with error at most $O(\delta^{\eta})$ with respect to $F'$. 
\end{itemize}
The first item holds with probability at least $3/4$ (in fact much higher probability!) by Theorem \ref{ChernoffThm}.  

We will show that the second item also holds with probability at least $3/4$. Fix $R\in\mathcal{R}$ with an associated set $F(R)$. By Theorem \ref{ChernoffThm}, we have
\begin{align*}
&\mathbb{P}\Big[\#(F(R)\cap F')\leq \frac{1}{2}p(\#F(R)) \Big]\leq e^{\frac{-p(\#F)}{8}},\\
&\mathbb{P}\Big[\#(F(R)\cap F')\geq 2p(\#F(R)) \Big]\leq e^{\frac{-p(\#F)}{3}},
\end{align*}
and hence the probability that at least one of these events occurs is at most $e^{-\delta^{-\eps_1}}$. Suppose that neither of these events occur, and hence
\[
p\mu/4 \leq \#(F(R)\cap F') \leq 4p\mu. 
\]
Let $\rho\in[\delta,1]$, let $T\in[1,1/\rho]$, and let $R'\supset R$ be a $(\rho; k; T)$ rectangle. We would like to show that with high probability,
\begin{equation}\label{probabilityEventGood}
\# \{f\in (F(R)\cap F')\colon f\sim R'\}= O\big(\delta^{-\eta} T^{-\eps_1} \mu p\big).
\end{equation}
We will estimate the probability that \eqref{probabilityEventGood} fails. First, we will estimate the probability that
\begin{equation}\label{probabilityEvent}
\# \{f\in (F(R)\cap F')\colon f\sim R'\} > 4\delta^{-\eta} T^{-\eps_1} \mu p.
\end{equation}
Define $n= \#\{f\in F(R)\colon f\sim R'\}$. By hypothesis, the rectangles in $\mathcal{R}$ are $\mu$ rich and $\eps$ robustly broad with error at most $\delta^{-\eta}$ with respect to $F$, and hence $n\leq \delta^{-\eta}T^{-\eps}(\# F(R))\leq 2\delta^{-\eta} T^{-\eps_1}\mu$. Write $4\delta^{-\eta}T^{-\eps_1}\mu p = Apn$, i.e.~$A=\frac{4\delta^{-\eta}T^{-\eps_1}\mu}{n}\geq 2$. Applying Theorem \ref{ChernoffThm} with $n$, $p$, and $A$ as above and using the fact that $T\leq \rho^{-1}\leq\delta^{-1}$, we conclude that the probability that \eqref{probabilityEvent} holds is at most 

\begin{equation}\label{prob26HoldsOneRect}
e^{\frac{-Apn}{6}}= e^{-2\delta^{-\eta}T^{-\eps_1}(\mu p)/6 } \leq e^{-2\delta^{\eps_1-\eta}(\delta^{-2\eps_1})/6 }\leq e^{-\delta^{-\eps_1}}.
\end{equation}
Our goal is to show that with high probability, \eqref{probabilityEventGood} holds for all $\rho\in [\delta,1]$; all $T\in [1,1/\rho]$, and all $(\rho;k;T)$ rectangles $R'$. We claim that it suffices to show that with high probability, \eqref{probabilityEvent} fails when we consider rectangles with the following three properties:
\begin{itemize}
\item[(i)] $\rho$ is of the form $\delta 2^j$ for $j\geq 0$ an integer.
\item[(ii)] $T$ is of the form $2^\ell$ for $\ell\geq 0$ an integer.
\item[(iii)] $R'$ is the vertical neighborhood of the graph of a function from $F$.
\end{itemize}
To verify this claim, we argue as follows. By the triangle inequality, if $C_0 = O(1)$ is chosen appropriately and if there is a rectangle $R'$ for which \eqref{probabilityEventGood} fails with (implicit) constant $C_0$, then there is a rectangle $R''$ satisfying Properties (i), (ii), and (iii), for which \eqref{probabilityEvent} holds. Taking the contrapositive, if \eqref{probabilityEvent} fails with high probability for every rectangle $R'\supset R$ satisfying Properties (i), (ii), and (iii), then \eqref{probabilityEventGood} holds with high probability for every rectangle $R'\supset R$, provided the implicit constant has been chosen appropriately.

We have shown that the probability that \eqref{probabilityEvent} holds for a fixed choice of rectangle $R\in\mathcal{R}$ and rectangle $R'$ is bounded above by \eqref{prob26HoldsOneRect}. Since there are $\delta^{-O(1)}$ rectangles that satisfy Properties (i), (ii), and (iii), we use the union bound to conclude that for a fixed choice of $R\in\mathcal{R}$, the probability that \eqref{probabilityEvent} holds for at least one rectangle $R'$ satisfying Properties (i), (ii), and (iii) is at most $e^{-\delta^{-\eps_1}}\delta^{-O(1)}$. This means that the probability that \eqref{probabilityEventGood} holds for a fixed rectangle $R\in\mathcal{R}$ is at least $1-e^{-\delta^{-\eps_1}}\delta^{-O(1)}$. Since the rectangles in $\mathcal{R}$ are pairwise incomparable, we have $\#\mathcal{R}=O(\delta^{-k-1})$, and hence the probability that \eqref{probabilityEventGood} holds for every rectangle in $\mathcal{R}$ is at least $1-e^{-\delta^{-\eps_1}}\delta^{-O(1)}$. If $\delta_0$ (and hence $\delta$) is selected sufficiently small depending on $k$ and $\eps_1$ (recall that $\eps_1$ in turn depends on $k$ and $\eps$), then this quantity is at least $3/4$. This completes the proof of our claim.

\medskip

\noindent\textbf {Step 3: Applying Proposition \ref{tangencyRectanglesAlgCurvesProp}}.

Next, let $\eps_2=\eps_2(k,\eps)<\eps_1$ be a quantity to be determined below, and let $\eta_1 = \eta_1(k, \eps_2)$ be the quantity from the statement of Proposition \ref{tangencyRectanglesAlgCurvesProp} (with $k$ as above and $\eps_2$ in place of $\eps$). If $\eta\leq\eta_1/2$ and $\delta_0$ is sufficiently small, then the rectangles in $\mathcal{R}$ are $\eps_2$ robustly broad with error at most $\delta^{-\eta_1}$ with respect to $F'$. Thus we can apply Proposition \ref{tangencyRectanglesAlgCurvesProp} (with $\eps_2$ in place of $\eps$ and $\eta_1$ in place of $\eta$) to conclude that 
\begin{equation}\label{finalCompRectBd}
\begin{split}
\#\mathcal{R} & \leq C \delta^{-B \eps_2} (\#F')^{\eps_2} (\# F')^{\frac{k+1}{k}} 
\leq C  \delta^{-B \eps_2} \delta^{-k\eps_2} \Big(\delta^{-2\eps_1}\frac{\#F}{\mu}\Big)^{\frac{k+1}{k}}.
\end{split}
\end{equation}
The result now follows by selecting $\eps_1<\eps/10$; $\eps_2<\eps/(10(B+k))$ (recall that $B$ depends only on $k$); and $\delta_0$ sufficiently small. 
\end{proof}


\subsection{Tangency Rectangles and Tangency prisms}
We now turn to the proof of Proposition \ref{tangencyRectanglesAlgCurvesProp}. We begin by analyzing the structure of tangency rectangles. Recall that a $(\delta;k)$ rectangle is the vertical $\delta$ neighborhood of a function $f$ with $C^k$ norm at most 1, above an interval $I$ of length $\delta^{1/k}$. For notational convenience, we will write this as $R^f(I)$ or $R(I)$, and we will write $I(R)$ to denote the interval $I$ associated to $R$. The next result says that the tangency rectangle $R^f(I)$ is accurately modeled by the $(k-1)$--st order Taylor expansion of $f$.

\begin{lem}[Structure of tangency rectangles]\label{tangencyRectStructure}
Let $R=R^f(I)$ be a $(\delta,k)$ tangency rectangle, with $I = [a, a+\delta^{1/k}]$. Let 
\[
g(t) = f(a)+ \sum_{j=1}^{k-1}\frac{f^{(j)}(a)}{j!}(t-a)^j
\] 
be the $(k-1)$-st order Taylor expansion of $f$ around $a$. Then $R$ is contained in the vertical $2\delta$ neighborhood of the graph of $g$ above $I$.
\end{lem}

\noindent This is a consequence of Taylor's theorem. We now define the ``tangency prisms'' introduced in Section \ref{introProofSketch}.

\begin{defn}\label{defnTangencyPrism}
A $(\delta;k)$ \emph{tangency prism} is a set $P$ of the form
\begin{equation}\label{tangencyPrismEqn}
\begin{split}
\Big\{(t, y_0,\ldots&,y_{k-1})\in\RR^{k+1}\colon  t\in [a, a+\delta^{1/k}],\\
&\Big| y_j - \sum_{i=j}^{k-1}\frac{(t-a)^{i-j}}{(i-j)!}b_i\Big|\leq K\delta^{1-j/k},\ j=0,\ldots,k-1\Big\}.
\end{split}
\end{equation}
In the above expression, $a\in[0,1]$ and $b_0,\ldots,b_{k-1}\in [-1,1]$ are parameters that define the tangency prism, and $K$ is a constant depending on $k$; the specific choice of $K$ will be fixed in Lemma \ref{tangencyPrismVsRect} below. We call $I = [a, a+\delta^{1/k}]$ the \emph{interval associated to} $P$.
\end{defn}

\begin{defn}
Let $P\subset\RR^{k+1}$ be a $(\delta;k)$ tangency prism with associated interval $I\subset[0,1]$, and let $h\colon [0,1]\to\RR^{k}$. We say $h\sim P$ if the graph of $h$ above $I$ is contained in $P$.
\end{defn}

\begin{defn}\label{jetLiftDefn}
Let $f\in C^k$ and let $0\leq j\leq k$. We define the $j$--th order jet lift of $f$, denoted $\mathcal{J}_jf,$ to be the function $\mathcal{J}_jf(t) = \big(f(t),f'(t),\ldots,f^{(j)}(t)\big)$. When $j=0$, we have $\mathcal{J}_0f(t)=f(t)$.
\end{defn}

\begin{defn}
Let $R=R^f(I)$ be a $(\delta,k)$ tangency rectangle. We define the tangency prism $\hat R$ to be a set of the form \eqref{tangencyPrismEqn}, where $a$ is the left endpoint of $I$, and $b_i = f^{(i)}(a)$ for $i=0,\ldots,k-1$.
\end{defn}

\begin{lem}\label{tangencyPrismVsRect}
If the quantity $K=O(1)$ from Definition \ref{defnTangencyPrism} is chosen appropriately (depending on $k$), then the following is true. Let $R$ be a $(\delta,k)$ tangency rectangle, let $f$ be a function with $C^k$ norm at most 1, and suppose that $f\sim R$. Then $\mathcal{J}_{k-1}f\sim \hat R$.
\end{lem}
\begin{proof}
Write $R=R^g(I)$, with $I = [a, a+\delta^{1/k}]$. Since $f\sim R$, we have $|f(t)-g(t)|\leq\delta$ on $I$, and thus by Lemma \ref{smallImpliesSmallCkOnJ} there exists a constant $K_1=K_1(k)\geq 1$ so that $|f^{(j)}(t)-g^{(j)}(t)|\leq K_1 \delta^{1-j/k}$ for $j=0,\ldots,k-1$. On the other hand, by Taylor's theorem, for each index $j< k$ and each $t\in I$, there exists $t_1$ between $a$ and $t$ so that
\[
g^{(j)}(t) = \sum_{i=j}^{k-1} \frac{(t-a)^{i-j}}{(i-j)!} g^{(i)}(a) + \frac{(t-a)^{k-j}}{(k-j)!}g^{(k)}(t_1).
\]
Define $b_i= g^{(i)}(a)$ for $j=0,\ldots,k-1$. Then for each $j=0,\ldots,k-1$ and each $t\in I$, we have
\begin{equation}
\begin{split}
\Big| f^{(j)}(t) - \sum_{i=j}^{k-1}\frac{(t-a)^{j-i}}{(j-i)!}b_i\Big| & \leq |f^{(j)}(t)-g^{(j)}(t)| + \Big| g^{(j)}(t) - \sum_{i-j}^{k-1}\frac{(t-a)^{i-j}}{(i-j)!} b_i\Big|\\
& \leq |f^{(j)}(t)-g^{(j)}(t)| + \Big| \frac{(t-a)^{k-j}}{(k-j)!}g^{(k)}(t_1)\Big|\\
&\leq K_1 \delta^{1-j/k} + \delta^{1-j/k},
\end{split}
\end{equation}
where the final line used the assumption that $\Vert g \Vert_{C^k}\leq 1$. Thus the lemma holds with $K=K_1+1$.
\end{proof}


\subsection{Tangency and rescaling}
In this section, we will explore how rescaling a tangency rectangle $R$ induces a rescaling of functions tangent to $R$, and also induces a rescaling of (smaller) tangency rectangles contained in $R$.

\begin{defn}\label{rectangleCoveringDefn}
Let $0<\delta<\rho\leq 1$. Let $R$ be a $(\delta;k)$ tangency rectangle, and let $S$ be a $(\rho;k)$ tangency rectangle. We say $S$ \emph{covers} $R$, denoted $S\succ R$ or $R\prec S$, if $\hat{R}\subset \hat S$.
\end{defn}

\begin{defn}\label{pullbackRect}
Let $\rho>0$ and let $S=S^g(I)$ be a $(\rho;k)$ rectangle; here $\Vert g \Vert_{C^k}\leq 1$ and $I=[a, a+\rho^{1/k}]$. Let $K=K(k)\geq 1$ be the constant from Definition \ref{defnTangencyPrism}, and let $c=\big((k+1)K\big)^{-1}$. 

\smallskip

\noindent (A): For $x\in I$, define
\[
\phi^S(x,y) = \big( \rho^{-1/k}(x-a),\ c\rho^{-1}(y-g(x))\big).
\]
With this definition, we have $\phi^S(S)=[0,1]\times [-c,c]$.

\smallskip

\noindent (B): For $f\colon [0,1]\to\RR$, define $f_S$ to be the function whose graph is $\phi^S(\operatorname{graph}f|_I)$. 

\smallskip

\noindent (C): For $x\in I$ and $y_0,\ldots,y_{k-1}\in\RR^k$, define 
\begin{equation*}
\begin{split}
\psi^S(x,&y_0,\ldots,y_{k-1}) \\
= \Big(&\rho^{-1/k}(x-a),\ c\rho^{-1}(y_0-g(x)),\ c\rho^{-1+1/k}(y_1-g'(x)),\\
& \qquad c\rho^{-1+2/k}(y_2-g''(x)),\ldots, c\rho^{-1/k}(y_{k-1}-g^{(k-1)}(x))\Big).
\end{split}
\end{equation*}
With this definition, we have $\psi^S(\hat S)=[0,1]\times [-Kc,Kc]^k = [0,1]\times [\frac{-1}{k+1},\frac{1}{k+1}]^k$.
\end{defn}

\begin{lem}\label{CkNormRectangleRescaling}
Let $S$ be a $(\rho;k)$ rectangle, let $\Vert f\Vert_{C^k}\leq 1$, and suppose $f\sim S$. Then $\Vert f_S \Vert_{C^k}\leq 1$. 
\end{lem}
\begin{proof}
By the chain rule, 
\begin{equation}\label{intertwinePullback}
\operatorname{graph}\big(\mathcal{J}_{k-1}(f_S)\big) = \psi^{S}(\operatorname{graph}\mathcal{J}_{k-1} f|_I).
\end{equation}
As a consequence, if $f\sim S$ and $\Vert f \Vert_{C^k}\leq 1$, then by Lemma \ref{tangencyPrismVsRect} we have $\mathcal{J}_{k-1}f\sim \hat S$, i.e.~$\operatorname{graph}\mathcal{J}_{k-1} f|_I\subset \hat S$, and thus
\[
\psi^{S}(\operatorname{graph}\mathcal{J}_{k-1} f|_I) \subset \psi^{S}(\hat S) = [0,1]\times [\frac{-1}{k+1},\frac{1}{k+1}]^k.
\]
In particular, we have 
\begin{equation}\label{supBdPhif}
\sup_{x\in[0,1]} |f_S^{(j)}(x)|\leq (k+1)^{-1},\quad j=0,\ldots,k-1.
\end{equation}
Write $S = S^g(I)$. Then for $x\in[0,1]$ we have
\[
f_S(x) = c\rho^{-1}\big( f(a+\rho^{1/k}x) - g(a+\rho^{1/k}x)\big).
\]
Using the chain rule and the triangle inequality, we compute 
\begin{equation}\label{finalDerivative}
\sup_{x\in[0,1]} |f_S^{(k)}(x)|\leq c\big( \sup_{z\in I}|f^{(k)}(z)| + \sup_{z\in I}|g^{(k)}(z)| \big)\leq \frac{1}{k+1},
\end{equation}
where the final inequality used the assumptions that $\Vert f\Vert_{C^k}\leq 1$ and $\Vert g\Vert_{C^k}\leq 1$, and the fact that $K\geq 2$. Combining \eqref{supBdPhif} and \eqref{finalDerivative}, we conclude that $\Vert f_S \Vert_{C^k}\leq 1$. 

\end{proof}

Motivated by the above computation, we introduce the following definition.

\begin{defn}\label{defnCapPhiRect}
Let $R$ be a $(\delta;k)$ tangency rectangle and let $S$ be a $(\rho;k)$ tangency rectangle. Suppose $R\prec S$. Then $\phi^S(R)$ is the vertical $c \delta/\rho$ neighborhood of a function $h$ (which has $C^k$ norm at most 1) above an interval $J$ of length $(\delta/\rho)^{1/k}$. Define $R_S$ to be the $(\delta/\rho; k)$  tangency rectangle given by the vertical $\delta/\rho$ neighborhood of $h$ above $J$.
\end{defn}

The next lemma says that our definitions of $f_S$ and $R_S$ preserve broadness. 
\begin{lem}
Let $R\prec S$ be tangency rectangles. Let $F$ be a set of functions with $C^k$ norm at most 1, all of which are tangent to $S$. Let $F(R)\subset\{f\in F\colon f\sim R\}$ satisfy \eqref{fExtendedRectangleBound} for some $B\geq 1$ and some $\eps>0$. Then the functions $\{f_S\colon f\in F(R)\}$ are tangent to $R_S,$ and satisfy the analogue of \eqref{fExtendedRectangleBound} with $B$ replaced by $O(B)$ and $\eps$ as above.
\end{lem}
\begin{proof}
Let $M\geq 1$ and suppose there exists a $(\tau;k;T)$ rectangle $R'=R^h(J) \supset R_S$ that is tangent to $M$ functions from $\{f_S \colon f\in F(R)\}$; denote this set of functions by $F_1$. Our goal is to show that
\begin{equation}\label{boundOnM}
M \lesssim B T^{-\eps}\#F(R).
\end{equation}

Fix a function $g_S\in F_1$. By the triangle inequality, the graph of each $f_S\in F_1$ above $J$ is contained in the vertical $2\tau$ neighborhood of $g_S$ above $J$; denote this latter set by $R''$ (note that $R''\supset R_S$ ). 

We have that $(\phi^S)^{-1}(R'')$ is the vertical $(2\tau)(c^{-1}\rho)$ neighborhood of $g$ (here $g$ is a function of $C^k$ norm at most 1), above an interval of length $(T\rho\tau)^{1/k}$, and this set contains $R$. In summary, we have constructed a $\big(\frac{2}{c}\tau\rho; k; \frac{cT}{2}\big)$ tangency rectangle that is tangent to at least $M$ functions from $F(R)$. Comparing with \eqref{fExtendedRectangleBound}, we conclude that
\[
M \leq B (\frac{cT}{2})^{\eps} \#F(R) \leq (2B/c) T^\eps\#F(R).
\]
Since $c>0$ depends only on $k$, this establishes \eqref{boundOnM}.
\end{proof}

We close this section with a final observation about tangency rectangles and covering.
\begin{lem}\label{tangencyImpliesCover}
Let $0<\delta<\rho\leq 1$ with $\rho\geq 2^k \delta$. Let $S=S^g(I)$ be a $(\rho;k)$ tangency rectangle, let $R=R^h(I')$ be a $(\delta;k)$ tangency rectangle, and let $f\colon I\to\RR$. Suppose that $I'\subset I$, that 
\begin{equation}\label{fVeryTangentG}
\sup_{t\in I}|f(t)-g(t)|\leq \frac{1}{2}\rho
\end{equation}
(this in particular implies $f\sim S$), and suppose that $f\sim R$. Then $R\prec S$.
\end{lem}
The proof of Lemma \ref{tangencyImpliesCover} is similar to that of Lemma \ref{tangencyPrismVsRect}, so we will just provide a brief sketch. The hypothesis \eqref{fVeryTangentG} implies that for all $t\in I$ we have
\[
|f^{(j)}(t)-g^{(j)}(t)|\leq \frac{1}{2} K \rho^{1-j/k},\quad j = 0,\ldots,k-1,
\]
while the hypothesis $f\sim R$ implies that for all $t\in I'$ we have
\[
|f^{(j)}(t)-g^{(j)}(t)|\leq  K \delta^{1-j/k}\leq \frac{1}{2}K\rho^{1-j/k},\quad j = 0,\ldots,k-1.
\]
The inclusion $\hat R \subset \hat S$ follows from the triangle inequality.


\subsection{Proof of Proposition \ref{tangencyRectanglesAlgCurvesProp} Part 1: Space curves, partitioning, and induction}\label{proofOfProptangencyRectanglesAlgCurvesPropPt1Sec}
We are now ready to begin the proof of Proposition \ref{tangencyRectanglesAlgCurvesProp}. Our basic strategy is as follows. We lift each function $f\in F$ to its $(k-1)$-st order jet $\mathcal{J}_{k-1}f$, and we lift each rectangle $R\in\mathcal{R}$ to its corresponding tangency prism $\hat R$. Proposition \ref{tangencyRectanglesAlgCurvesProp} then becomes an incidence theorem between (polynomial) curves and prisms in $\RR^{k+1}$. Roughly speaking, the statement is as follows: given a set of $N$ polynomial curves in $\RR^{k+1}$ that come from the jet lifts of plane curves, there can be at most $N^{\frac{k+1}{k}+\eps}$ prisms that are (broadly) incident to these curves. We prove this statement by induction on $N$. For the induction step, we use the Guth-Katz polynomial partitioning theorem to divide $\RR^{k+1}$ into cells, most of which interact with only a small fraction of the (lifted) curves from $F$. The precise statement is a consequence of the following two theorems. The first is the celebrated Guth-Katz polynomial partitioning theorem \cite{GuthKatz}.

\begin{thm}\label{polyPartitioningTheorem}
Let $\mathcal{P}\subset\RR^n$ be a finite set of points. Then for each $E\geq 1$, there is a polynomial $Q\in\RR[x_1,\ldots,x_n]$ so that $\RR^n\backslash \{Q=0\}$ is a union of $O(E^n)$ open connected sets, and each of these sets intersects $O(E^{-n}\#\mathcal{P})$ points from $\mathcal{P}$.
\end{thm}

The second is a variant of B\'ezout's theorem for real varieties. This is a special case of the main result from \cite{BaroneBasu}.
\begin{prop}\label{BBProp}
Let $\zeta\subset\RR^n$ be a one-dimensional real variety defined by polynomials of degree at most $D$. Let $Q\in\RR[x_1,\ldots,x_n]$ be a polynomial of degree $E\geq D$. Then $\zeta$ intersects $O(D^{n-1}E)$ connected components of $\RR^n\backslash\{Q=0\}$. 
\end{prop}

We apply the induction hypothesis inside each cell, and sum the resulting contributions. The exponent $\frac{k+1}{k}+\eps$ was chosen so that the induction closes. There is also a contribution from the boundary of the partition. This will be described in greater detail (and dealt with) later. We now turn to the details.

\begin{proof}[Proof of Proposition \ref{tangencyRectanglesAlgCurvesProp}]
Fix $k$ and $\eps$.  We will prove the result by induction on $\# F$. The induction will close, provided $B,C,$ and $\eta$ have been chosen appropriately. When $F=\emptyset,$ there is nothing to prove. 

\medskip

\noindent {\bf Step 1. Polynomial partitioning.} 
Suppose that $\#F = N$, and that the result has been proved for all sets of curves $F'$ of cardinality less than $N$. To each $(\delta;k)$ tangency rectangle $R^f(I)\in\mathcal{R}$, associate the point $p_R = (a, f(a), f'(a),\ldots, f^{(k-1)}(a))\in\RR^{k+1}$, where $a$ is the left endpoint of $I$. Observe that $p_R\in\hat R$. It is easy to verify that distinct (and hence incomparable) rectangles in $\mathcal{R}$ give rise to distinct (in fact $\gtrsim\delta$ separated) points. Let $\mathcal{P}=\{p_R\colon R\in\mathcal{R}\}$.

Let $E\geq 1$ be a number to be specified below. Use Theorem \ref{polyPartitioningTheorem} to select a polynomial $Q\in\RR[t, y_0,\ldots,y_{k-1}]$ of degree at most $E$, so that $\RR^{k+1}\backslash\{Q=0\}$ is a union of $O(E^{k+1})$ open connected components, each of which contain $O(E^{-k-1}\#\mathcal{P})$ points from $\mathcal{P}$. Let $\mathcal{O}$ denote the set of connected components.

Define $Z=\{Q=0\}$, and define $Z^*$ to be the union of all $(\delta;k)$ tangency prisms that intersect $Z$. We claim that for each $R\in\mathcal{R}$, at least one of the following must hold:
\begin{itemize}
\item There is a cell $\Omega\in\mathcal{O}$ so that $\hat R\subset \Omega$.

\item  $\hat R\subset Z^*$.
\end{itemize}
Indeed, if the second item does not hold then $\hat R$ is disjoint from $Z$. Since $\hat R$ is connected, we must have $\hat R\subset \Omega$ for some $\Omega\in \mathcal{O}$. 

For each $\Omega\in \mathcal{O}$, define 
\[
\mathcal{R}_\Omega=\{R\in\mathcal{R}\colon \hat R \subset \Omega\}.
\]
We have $\#\mathcal{R}_{\Omega} \leq \#(\mathcal{P}\cap \Omega) =O(E^{-k-1}\#\mathcal{R})$. If $R\in\mathcal{R}_{\Omega}$ and $f\in F$ with $f\sim R$, then $\operatorname{graph}(\mathcal{J}_{k-1}f)\cap \hat R\neq\emptyset,$ and hence $\operatorname{graph}(\mathcal{J}_{k-1}f)\cap \Omega\neq\emptyset$.

Define $\mathcal{R}_Z=\mathcal{R}\backslash\bigcup_{\Omega\in \mathcal{O} }\mathcal{R}_{\Omega}$. We say we are in the \emph{cellular case} if 
\[
\#\bigcup_{\Omega\in \mathcal{O}}\mathcal{R}_\Omega \geq \frac{1}{2}\#\mathcal{R}.
\] 
Otherwise we are in the \emph{algebraic case}. We remark that if $E^{k+1}$ is substantially larger than $\#\mathcal{R}$, then the bound $O(E^{-k-1}\#\mathcal{P})$ from the application of Theorem \ref{polyPartitioningTheorem} might be smaller than 1, i.e.~each cell contains zero points from $\mathcal{P}$. If this happens, then $\mathcal{P}\subset \{Q=0\}$, and we are most certainly in the algebraic case. 

\medskip

\noindent {\bf Step 2. The cellular case.} 
Suppose we are in the cellular case. Then we may select a set $\mathcal{O}'\subset \mathcal{O}$ so that $\sum_{\Omega\in \mathcal{O}'}\#\mathcal{R}_{\Omega} \geq \frac{1}{4}\#\mathcal{R}$, and 
\begin{equation}\label{eachCellManyRectangles}
\#\mathcal{R}_{\Omega}\geq c_1(k) E^{-k-1}(\#\mathcal{R})\quad\textrm{for each}\ \Omega\in \mathcal{O}',
\end{equation} 
where $c_1(k)>0$ is a quantity depending only on $k$. To simplify notation, write $\zeta_f$ for $\operatorname{graph}(\mathcal{J}_{k-1}f)$. Note that if $f$ is a polynomial of degree $D$, then $\zeta_f$ is a one-dimensional real variety defined by polynomials of degree at most $D$.

By Proposition \ref{BBProp}, since each polynomial in $F$ has degree at most $\delta^{-\eta}$, then provided $E\geq\delta^{-\eta}$ (our choice of $E$ will satisfy this hypotheses; see below), there are at most $K_1(k) \delta^{-k\eta} E(\# F)$ pairs $(\Omega, f)\in\mathcal{O}'\times F$ with $\zeta_f\cap \Omega\neq\emptyset$ (here $K_1(k)\geq 1$ is a constant depending only on $k$). Thus there is a cell $\Omega \in\mathcal{O}'$ with 
\begin{equation}\label{numberOfGInteractingF}
\#\{f\in F \colon \zeta_f\cap \Omega \neq\emptyset\}\leq K_1(k) \delta^{-k\eta} E^{-k}(\#F).
\end{equation}
Denote the above set by $F_\Omega$. If we choose $E$ sufficiently large ($E \geq K_1(k)\delta^{-\eta}$ will suffice), then  $\#F_\Omega<N$. Since the $(\delta,k)$ rectangles in $\mathcal{R}_\Omega$ are $\mu$ rich (for some $\mu\geq 1$) and $\eps$ robustly broad with error at most $\delta^{-\eta}$ with respect to $F_{\Omega}$, we may apply the induction hypothesis to conclude that
\begin{equation}\label{inductionHypothesisOnRO}
\#\mathcal{R}_\Omega \leq C  \delta^{-B\eps}(\# F_\Omega)^{\frac{k+1}{k}+\eps}.
\end{equation}

Combining \eqref{eachCellManyRectangles}, \eqref{inductionHypothesisOnRO}, and then \eqref{numberOfGInteractingF}, we conclude that
\begin{equation}\label{boundCardR}
\begin{split}
\#\mathcal{R}&\leq \Big(c_1(k)^{-1} E^{k+1}\Big)\Big( C  \delta^{-B\eps}(\# F_\Omega)^{\frac{k+1}{k}+\eps}\Big)\\
&\leq \Big(c_1(k)^{-1} E^{k+1}\Big)\Big(C  \delta^{-B\eps}\big(K_1(k)\delta^{-k\eta} E^{-k}\#F\big)^{\frac{k+1}{k}+\eps}\Big)\\
&\leq \Big(c_1(k)^{-1} K_1(k)^{\frac{k+1}{k}+\eps}\delta^{-(k+1)\eta-k\eps\eta}  E^{-k\eps} \Big) \Big( C  \delta^{-B\eps} (\#F)^{\frac{k+1}{k}+\eps}\Big).
\end{split}
\end{equation}
At this stage we fix a choice of $E$ of the form $E = C_1 \delta^{-3\eta/\eps}$. If we select $C_1$ sufficiently large (depending on $k$ and $\eps$), then the first bracketed term on the final line of \eqref{boundCardR} is at most $1$, and thus
\[
\#\mathcal{R}\leq C  \delta^{-B\eps} (\#F)^{\frac{k+1}{k}+\eps},
\]
and the induction closes. This completes the proof of Proposition \ref{tangencyRectanglesAlgCurvesProp} when we are in the cellular case.


\medskip
\noindent {\bf Step 3. The algebraic case.} 
Next we consider the algebraic case. Observe that the tangency prisms associated to rectangles in $\mathcal{R}_Z$ are contained in a thin neighborhood of the variety $Z$. The following theorem of Wongkew \cite{Wongkew} controls the volume of the thin neighborhood of a variety. 
\begin{thm}\label{WonkewThm}
Let $Z=\{Q=0\}\subset\RR^n$, where $Q$ is a non-zero polynomial. Let $B\subset\RR^n$ be a ball of radius $r$. Then for all $\rho>0$, we have 
\begin{equation}\label{WonkewBd}
|B \cap N_\rho(Z)|\lesssim (\operatorname{deg} Q)^{n}\rho r^{n-1},
\end{equation}
where $N_\rho(Z)$ denotes the $\rho$-neighbourhood of $Z$.
\end{thm}

The set on the LHS of \eqref{WonkewBd} is described by a Boolean combination of polynomial (in)equalities. Sets of this form are called semi-algebraic; we give a precise definition below. 
\begin{defn}\label{complexitySemiAlgSet}
Let $M\geq 1$. A set $W \subset\RR^n$ is called a \emph{semi-algebraic set} of complexity at most $M$ if there exists an integer $N\leq M$; polynomials $P_1,\ldots,P_N$, each of degree at most $M$; and a Boolean formula $\Phi\colon \{0,1\}^N\to\{0,1\}$ such that
\[
W = \big\{x\in\RR^n \colon \Phi\big( P_1(x)\geq 0, \ldots, P_N(x)\geq 0 \big)=1\big\}.
\]
\end{defn}

The next result describes the structure of arrangements of tangency rectangles whose corresponding tangency prisms are contained in a semi-algebraic set of small volume. 
\begin{prop}\label{tangencyRectanglesInsideSemiAlgSetProp}
Let $k\geq 1$, $\eps>0$. Then there exist positive numbers $c=c(k)$, $\eta=\eta(k,\eps),$ and $\delta_0=\delta_0(k,\eps)$ so that the following holds for all $\delta\in(0,\delta_0]$. 
Let $F$ be a set of polynomials of degree at most $\delta^{-\eta}$, each of which has $C^k$ norm at most 1. Let $\mathcal{R}$ be a set of pairwise incomparable $(\delta;k)$ rectangles. For each $R\in\mathcal{R}$, let $F(R)\subset\{f\in F\colon f\sim R\}$. Define the dual relation $\mathcal{R}(f) = \{R\in\mathcal{R}\colon f\in F(R)\}$. Suppose that for each $f\in F$, the rectangles in $\mathcal{R}(f)$ satisfy the following ``two-ends'' type non-concentration condition: for each interval $J\subset[0,1]$, we have

\begin{equation}\label{twoEndsAlongEachCurveInProp}
    \# \{R\in\mathcal{R}(f) \colon I(R)\subset J\}\leq \delta^{-\eta} |J|^{\eps}\#\mathcal{R}(f).
\end{equation}
Let $W\subset [0,1]^{k+1}$ be a semi-algebraic set of complexity at most $\delta^{-\eta}$ and volume $|W|\leq \delta^\eps$. Suppose that $\hat R\subset W$ for each $R\in\mathcal{R}$. 

Then there exist $R\in\mathcal{R}$, $\tau\in [\delta,\delta^{c}]$, and a $(\tau; k+1; c)$ tangency rectangle $R_1\supset R$ with
\begin{equation}
\#\{f\in F(R)\colon f\sim R_1\} \gtrsim \#F(R).
\end{equation}
\end{prop}

We defer the proof of Proposition \ref{tangencyRectanglesInsideSemiAlgSetProp} to the next section. Using Proposition \ref{tangencyRectanglesInsideSemiAlgSetProp}, we will handle the algebraic case.  

\medskip
\noindent {\bf Step 3.1 A two-ends reduction.} 
Recall that for each $R\in\mathcal{R}$, there is a set $F(R)\subset F$ that satisfies the non-concentration condition \eqref{fExtendedRectangleBound} from Definition \ref{robustlyBroadDefn} with $\delta^{-\eta}$ in place of $B$, and $\eps$ as above. 
Since every set of pairwise incomparable $(\delta;k)$ rectangles has cardinality $O(\delta^{-k-1})$, we may suppose that $\#F\leq\delta^{-k}$, or else \eqref{RectBdEqnForPolysMuSmall} is immediately satisfied and we are done. Thus we have that $1\leq \# F(R)\leq \delta^{-k}$ for each $R\in\mathcal{R}$. After dyadic pigeonholing, we can find a set $\mathcal{R}_1\subset\mathcal{R}$ with $\#\mathcal{R}_1\gtrsim(\log1/\delta)^{-1}\#\mathcal{R}$ and a number $\mu$ so that $\mu\leq \#F(R)<2\mu$ for each $R\in\mathcal{R}_1$.  Define 
\[
\mathcal{I}_1 = \{ (f,R)\colon R\in\mathcal{R}_1,\ f\in F(R)\}.
\] 
We have
\begin{equation}\label{sizeOfI1}
\mu(\#\mathcal{R}_1)\leq \#\mathcal{I}_1 <2\mu(\#\mathcal{R}_1).
\end{equation}
For each $f\in F$, the curve $\zeta_f$ intersects $Z^*$ in a union of $O( (\delta^{-\eta}E)^{O(1)}) =O_\eps(\delta^{-O(\eta/\eps)})$ connected components, each of which is the graph of $\mathcal{J}_{k-1}f$ above an interval. Recall that if $(f,R)\in \mathcal{I}_1$, then the graph of $\mathcal{J}_{k-1}f$ above $I(R)$ is contained in $Z^*$, so $I(R)$ is contained in one of these intervals. 

Applying pigeonholing to each $f\in F$ in turn, we select a set $\mathcal{I}_2\subset\mathcal{I}_1$ with 
\[
\#\mathcal{I}_2\gtrsim_\eps (\log 1/\delta)^{-1}\delta^{O(\eta/\eps)}(\#\mathcal{I}_1)-\log(1/\delta)(\#F)
\]
so that the following holds: for each $f\in F$, there is an interval $I_f^\dag\subset[0,1]$ with the following properties:
\begin{itemize}
	\item $I_f^\dag$ is dyadic, i.e.~it is of the form $[2^{-j}a, 2^{-j}(a+1)]$ for some $j,a\in\ZZ$ with $\delta^{1/k}\leq 2^{-j} \leq 1$.
	\item The restriction of $\zeta_f$ to the graph of $\mathcal{J}_{j-1}f$ above $I_f^\dag$ is contained in $Z^*$.
	\item For every $R\in\mathcal{R}$ with $(f,R)\in \mathcal{I}_2$, we have $I(R)\subset I_f^\dag.$
\end{itemize}
For each $f\in F$, we select the interval $I_f^\dag$ as follows. First, choose an interval $J=J(f)\subset[0,1]$ so that the graph of $\mathcal{J}_{k-1}f$ above $J$ is contained in $Z^*$, and 
\begin{equation}\label{mostIncidenceRectanglesKept}
\#\{R\colon (f,R)\in \mathcal{I}_1,\ I(R)\subset J\}\gtrsim_\eps \delta^{O(\eta/\eps)}\#\{R\colon (f,R)\in \mathcal{I}_1\}.
\end{equation}
Divide the interval $J$ into $O(\log 1/\delta)$ interior-disjoint dyadic intervals $\{J_i\}$, each of length $\geq \delta^{1/k}$, plus two (possibly non-dyadic) intervals of length $\leq \delta^{1/k}$. Recall that the intervals $I(R)$ corresponding to the incidence rectangles on the LHS of \eqref{mostIncidenceRectanglesKept} are disjoint and have length $\delta^{1/k}$, and thus all but $O(\log 1/\delta)$ of these intervals are contained in a single interval from the set $\{J_i\}$ described above. By pigeonholing, we can select a dyadic interval $I_f^\dag \in \{J_i\}$ so that
\begin{equation*}
\begin{split}
\#\{R\colon (f,R)\in \mathcal{I}_1,\ & I(R)\subset I_f^\dag \}\\
&\gtrsim_\eps (\log 1/\delta)^{-1}\delta^{O(\eta/\eps)}\#\{R\colon (f,R)\in \mathcal{I}_1\}-\log 1/\delta.
\end{split}
\end{equation*}
The set $\mathcal{I}_2$ constructed in this way satisfies 
\[
\#\mathcal{I}_2\gtrsim_\eps (\log 1/\delta)^{-1}\delta^{O(\eta/\eps)}(\#\mathcal{I}_1) - (\log 1/\delta)(\# F).
\]
We may suppose that the first term dominates, i.e.~$\#\mathcal{I}_2\gtrsim_\eps (\log 1/\delta)^{-1}\delta^{O(\eta/\eps)}(\#\mathcal{I}_1)$, since otherwise we would have $\mathcal{I}_1\lesssim_\eps \log(1/\delta)(\#F)$, and hence \eqref{RectBdEqnForPolysMuSmall} holds and we are done.

Let $\eps_1>0$ be a small quantity to be chosen below; we will select $\eps_1$ small compared to $\eps$. For each $f\in F$, let $I_f\subset I_f^\dag$ be a dyadic interval that maximizes the quantity 
\[
|I_f|^{-\eps_1} \#\{R \colon (f,R)\in \mathcal{I}_2, I(R) \subset I_f \}.
\]

Define 
\[
\mathcal{I}_3 = \{(f,R)\in\mathcal{I}_2\colon  I(R)\subset I_f\}.
\]

By the maximality of $I_f$, for each dyadic interval $J\subset I_f$ we have
\begin{equation}\label{twoEndsAlongEachCurveInsideIf}
\begin{split}
    \# \{R & \colon (f,R)\in \mathcal{I}_2,\  I(R)\subset J\} \\
    & \leq (|J|/|I_f|)^{\eps_1} \# \{R \colon (f,R)\in \mathcal{I}_2,\ I(R)\subset I_f\} \\
    & = (|J|/|I_f|)^{\eps_1} \# \{R\colon (f,R)\in \mathcal{I}_3\}.
\end{split}
\end{equation}
On the one hand, since $I_f$ is a dyadic interval, each (not necessarily dyadic) interval $J\subset I_f$ is contained in a union of at most 2 dyadic intervals of length at most $2|J|$, and thus \eqref{twoEndsAlongEachCurveInsideIf} continues to hold, with the RHS weakened by a factor of 4, for all intervals $J\subset I_f$.
Since $|I_f^\dag|\leq 1$, we have
\begin{equation}\label{mostOfI2Captured}
\begin{split}
|I_f|^{-\eps_1}  \#\{R\colon (f,R)& \in \mathcal{I}_2,\ I(R) \subset I_f \} \\
& \geq \#\{R \colon (f,R)\in \mathcal{I}_2,\ I(R) \subset I_f^\dag \}.
\end{split}
\end{equation}
On the other hand, if the set on the RHS of \eqref{mostOfI2Captured} is non-empty, then we may suppose that $I_f$ contains at least one interval of length $\delta^{1/k}$, and in particular $|I_f|^{-\eps_1}\leq \delta^{-\eps_1}$. We conclude that $\#\mathcal{I}_3\geq\delta^{\eps_1}(\#\mathcal{I}_2)$. Observe that \eqref{twoEndsAlongEachCurveInsideIf} is a rescaled analogue of \eqref{twoEndsAlongEachCurveInProp} inside $I_f$. 

After further dyadic pigeonholing, we can select a set $F_4\subset F,$ a multiplicity $\nu$, and a dyadic length $\ell=2^{-j}\in[\delta^{1/k},1]$ so that the following conditions hold:
\begin{enumerate}
\item[(a)] 
For each $f\in F_4$, $I_f$ is a dyadic interval of length $\ell$.
\item[(b)] Each $f\in F_4$ satisfies 
\begin{equation}\label{multBdFForI3}
\nu\leq\# \{R\colon (f,R)\in \mathcal{I}_3\}< 2\nu.
\end{equation}
\item[(c)] The set $\mathcal{I}_4=\mathcal{I}_3\cap (F_4\times\mathcal{R})$ satisfies $\#\mathcal{I}_4\gtrsim (\log 1/\delta)^{-2}(\#\mathcal{I}_3)$.
\end{enumerate}
We have the following bounds on the size of $\mathcal{I}_4$:
\begin{equation}\label{sizeOfI3}
\begin{split}
(\log 1/\delta)^{-3}\delta^{\eps_1+O(\eta/\eps)}\mu(\#\mathcal{R}_1)\lesssim_\eps\phantom{.} & \#\mathcal{I}_4 <2\mu(\#\mathcal{R}_1),\\
 \nu(\#F_4) \leq\phantom{_\eps .} & \#\mathcal{I}_4 <2\nu(\#F_4).
\end{split}
\end{equation}
Note that \eqref{multBdFForI3} continues to hold with $\mathcal{I}_4$ in place of $\mathcal{I}_3$.

\medskip
\noindent {\bf Step 3.2 Graph refinement.} 
At this point, the functions $f\in F_4$ satisfy a rescaled analogue of \eqref{twoEndsAlongEachCurveInProp}. Unfortunately, while all of the rectangles $R\in\mathcal{R}_1$ are incident to at least $\mu$ functions from $F$ under the incidence relation $\mathcal{I}_1$, the same need not be true under the incidence relation $\mathcal{I}_4$. We can fix this by applying the following graph refinement lemma from \cite{dg}.

\begin{lem}[Graph refinement]\label{graphRefinementLemma}
Let $G = (A\sqcup B, E)$ be a bipartite graph. Then there is an induced sub-graph $G'=(A'\sqcup B', E')$ so that $\#E'\geq \#E/2$; each vertex in $A'$ has degree at least $\frac{\#E}{4\#A}$; and each vertex in $B'$ has degree at least $\frac{\#E}{4\#B}$.
\end{lem}

Apply Lemma \ref{graphRefinementLemma} to the bipartite graph $(F_4\sqcup \mathcal{R}_1, \mathcal{I}_4)$. We obtain an induced sub-graph $(F_5\sqcup \mathcal{R}_5, \mathcal{I}_5)$ with the following properties:
\begin{itemize}
\item $\#\mathcal{I}_5\geq \frac{1}{2}\#\mathcal{I}_4$, and hence \eqref{sizeOfI3} continues to hold with $\mathcal{I}_5$ in place of $\mathcal{I}_4$, with the LHS weakened by a factor of $1/2$. 
\item Each $f\in F_5$ satisfies an analogue of \eqref{multBdFForI3} with $\mathcal{I}_5$ in place of $\mathcal{I}_3$, except that the LHS is weakened to $\nu/4$.
\item Each $R\in \mathcal{R}_5$ is incident (under the incidence relation $\mathcal{I}_5$) to at least $(\#\mathcal{I}_4)/(4\# \mathcal{R}_1)\gtrsim_\eps (\log1/\delta)^{-3}\delta^{\eps_1+O(\eta/\eps)}\mu$ functions $f\in F_5$. 
\end{itemize}

\medskip
\noindent {\bf Step 3.3 Rescaling.} 
If $\ell\leq\delta^{1/k-\eps}$, then for each $f\in F_5$ there are at most $\delta^{-\eps}$ rectangles $R\in \mathcal{R}$ with $(f,R)\in \mathcal{I}_5$ (see the comment after Definition \ref{comparableDefn}). We conclude that 
\[
\#\mathcal{R}\lesssim (\log 1/\delta)\# \mathcal{R}_1  \lesssim (\log1/\delta)^{4}\delta^{-\eps_1-O(\eta/\eps)}\#\mathcal{I}_5\leq (\log1/\delta)^{4}\delta^{-\eps-\eps_1-O(\eta/\eps)}\#F,
\]
and hence \eqref{RectBdEqnForPolysMuSmall} holds and we are done (provided we select $\eps_1\leq\eps$, $\eta\leq c\eps^2$ for a sufficiently small constant $c\sim 1$, $B\geq 3$, and $C$ sufficiently large).

Next, suppose that
\begin{equation}\label{ellLong}
\ell\geq\delta^{1/k-\eps}.
\end{equation}
Our goal is to obtain a contradiction, and thereby finish the proof. 

Let $\rho=\ell^k\geq\delta^{1-k\eps}$. Let $\mathcal{S}$ be a set of $(\rho;k)$ rectangles constructed as follows. Let $\mathcal{J}$ be the set of dyadic intervals of length $\ell$ in $[0,1]$ (there are $\ell^{-1}$ such intervals). For each interval $J=[a, a+\ell]\in\mathcal{J}$, $\mathcal{S}$ contains all rectangles of the form $S^g(J)$, where $g$ is a function satisfying $\Vert g\Vert_{C^k(J)}\leq 1$ that is of the form $g(t) = \sum_{i=0}^{k-1}(t-a)^i b_i$ with $b_i\in \frac{1}{100 2^k} \rho^{1-i/k}\ZZ$. 

For every $f\in F_5$, there exists a function $g\colon I_f\to\RR$ of the above form so that
\begin{equation}\label{fCloseToGRho2}
\sup_{t\in I_f}|f(t)-g(t)|\leq \rho/2.
\end{equation}
 For each $f\in F_5$, let $\mathcal{S}(f)$ be the set of rectangles $S = S^g(I_f)$ for which \eqref{fCloseToGRho2} holds. In particular,  $\mathcal{S}(f)$ is non-empty. By Lemma \ref{tangencyImpliesCover} (here we require $\delta>0$ to be sufficiently small so that $\delta^{-k\eps}\geq 2^k$, and thus  $\rho\geq 2^k\delta$; if $\delta$ is not sufficiently small then \eqref{RectBdEqnForPolysMuSmall} is immediate), we have that if $(f,R)\in\mathcal{I}_5$ and $S\in \mathcal{S}(f)$, then $R\prec S$. 

Finally, each point in $\RR^{k+1}$ is contained in $O(1)$ sets $\{\hat S\colon S\in\mathcal{S}\}$, which means that for each $R\in\mathcal{R}$, there are 
$O(1)$ rectangles $S\in\mathcal{S}$ with $R\prec S$.

For each $S\in\mathcal{S}$, define 
\begin{equation}
\begin{split}
F_S & =\{f\in F_5\colon S\in\mathcal{S}(f)\},\\
\mathcal{R}_S & = \{R\in\mathcal{R}_5\colon R\prec S\},\\
\mathcal{I}_S & = \mathcal{I}_5\cap (F_S\times \mathcal{R}_S).
\end{split}
\end{equation}
Then 
\begin{equation}\label{FSVsIS}
\frac{1}{4}\nu(\#F_S) \leq \#\mathcal{I}_S \leq 2\nu(\# F_S),
\end{equation}
and
\begin{equation}\label{notTooMuchOverlap}
\mathcal{I}_5 = \bigcup_{S\in\mathcal{S}}\mathcal{I}_S,\qquad \sum_{S\in\mathcal{S}}\#\mathcal{R}_S \lesssim \#\mathcal{R}_5.
\end{equation}

Since each $R\in\mathcal{R}_5$ is incident to $\gtrsim_\eps (\log1/\delta)^{-3}\delta^{\eps_1+O(\eta/\eps)}\mu$ functions $f\in F_5$ under the incidence relation $\mathcal{I}_5$, by pigeonholing and \eqref{notTooMuchOverlap} we can select a rectangle $S\in\mathcal{S}$ so that
\[
\#\mathcal{I}_S \gtrsim_\eps (\log1/\delta)^{-3}\delta^{\eps_1+O(\eta/\eps)}\mu(\#\mathcal{R}_S).
\]

Apply Lemma \ref{graphRefinementLemma} to the graph $(F_S\sqcup \mathcal{R}_S, \mathcal{I}_S)$, and denote the output by $(F_S'\sqcup \mathcal{R}_S',\ \mathcal{I}_S')$. Then under the incidence relation $\mathcal{I}_S'$, each $f\in F_S'$ is incident to between $\frac{1}{16}\nu$ and $2\nu$ rectangles $R\in\mathcal{R}_S'$, and each $R\in\mathcal{R}_S'$ is incident to $\gtrsim_\eps (\log1/\delta)^{-3}\delta^{\eps_1+O(\eta/\eps)}\mu$ functions $f\in F_S'$. 

Furthermore, for each $f\in F_S'$ and each interval $J\subset I_f$, we have
\begin{equation}\label{stillObeyTwoEnds}
\begin{split}
\#\{R\colon (f,R)\in\mathcal{I}_S',\ I(R)\subset J\} &\leq  \# \{R  \colon (f,R)\in \mathcal{I}_2,\  I(R)\subset J\}\\
& \leq (|J|/|I_f|)^{\eps_1} \# \{R\colon (f,R)\in \mathcal{I}_3\}\\
& \lesssim (|J|/|I_f|)^{\eps_1} \nu\\
& \lesssim (|J|/|I_f|)^{\eps_1} \#\{R\colon (f,R)\in\mathcal{I}_S'\}.
\end{split}
\end{equation}

Apply the rescaling $f\mapsto f_S$ and $R\mapsto R_S$ from Definitions \ref{pullbackRect} and \ref{defnCapPhiRect} to the sets $F_S'$ and $\mathcal{R}_S'$. This gives sets $\tilde F_S$ and $\tilde{\mathcal{R}}_S,$ and an incidence relation $\tilde{\mathcal{I}}_S$.

For each $\tilde R\in \tilde{\mathcal{R}}_S,$ define 
\[
\tilde F_S(\tilde R)=\{\tilde f\in \tilde F_S\colon (\tilde f,\tilde R)\in \tilde{\mathcal{I}}_S \}.
\]
Then by \eqref{stillObeyTwoEnds}, the sets $\tilde F_S$ and $\tilde{\mathcal{R}}_S,$ and the sets $\{\tilde F_S(\tilde R)\}$ obey the two-ends non-concentration condition \eqref{twoEndsAlongEachCurveInProp} from Proposition \ref{tangencyRectanglesInsideSemiAlgSetProp} at scale $\tilde\delta=\delta/\rho$ in place of $\delta$, with $\eps_1$ in place of $\eps$ and a number $O(1)$ in place of $\delta^{-\eta}$ in Inequality \eqref{twoEndsAlongEachCurveInProp}. Note that each function in $\tilde F_S$ has degree at most $\delta^{-\eta}\leq \tilde\delta^{-\eta/k\eps}$.

Before we can apply Proposition \ref{tangencyRectanglesInsideSemiAlgSetProp}, however, we must show that the prisms $\{\hat{\tilde R}\colon \tilde R\in\tilde{\mathcal{R}}_S\}$  are contained in a semi-algebraic set $W$ of controlled complexity and small volume. 
First, observe that every such prism $\hat{\tilde R}$ is contained in $\psi^S(S\cap Z^*)$ (recall that $\psi^S$ is defined in Definition \ref{pullbackRect}), which in turn is contained in the union of all $(\tilde\delta;k)$ tangency prisms that intersect $\phi^S(S\cap Z)\subset ([0,1]\times[-1,1]^k) \cap \psi^S(Z)$. This in turn is contained in the set 
\[
W=\big([0,1]\times[-1,1]^k\big) \cap N_{\tilde\delta^{1/k}}(\psi^S(Z)).
\] 
$\psi^S(Z)$ is an algebraic variety of degree at most $\deg Q\leq E$, so by Theorem \ref{WonkewThm} we have
\[
|W| \lesssim E^{k+1}\tilde\delta^{1/k} \lesssim_\eps  \delta^{-O(\eta/\eps)}\tilde\delta^{1/k} \leq \tilde\delta^{1/k-O(\eta/\eps^2)},
\]
where we use the bound $\rho\geq\delta^{1-k\eps}$ (and thus $\tilde\delta\leq \delta^{k\eps}$) to replace $\delta^{-O(\eta/\eps)}$ with $\tilde\delta^{-O(\eta/\eps^2)}$.

It is straightforward to show that $W$ has complexity at most $E^{O(1)}\lesssim_\eps \delta^{O(\eta/\eps)}\lesssim \tilde\delta^{O(\eta/\eps^2)}$. We wish to apply Proposition \ref{tangencyRectanglesInsideSemiAlgSetProp} with $\tilde\delta$ in place of $\delta$ and $\eps_1$ in place of $\eps$. Let $c=c(k)>0$ and $\eta_1$ be the corresponding quantities from  Proposition \ref{tangencyRectanglesInsideSemiAlgSetProp}. If $\eta>0$ is selected sufficiently small depending on $\eta_1,k$, and $\eps$ (recall that $\eta_1$ depends on $k$ and $\eps_1$, and $\eps_1$ in turn depends on $k$ and $\eps$), then the hypotheses of Proposition \ref{tangencyRectanglesInsideSemiAlgSetProp} are satisfied. We conclude that there is a rectangle $\tilde R \in \tilde{\mathcal{R}}_S$; a number $c\gtrsim 1$; a scale $\tau\in [\tilde\delta,\tilde\delta^{c}]$; and a $(\tau;k+1; c)$ rectangle $\tilde R_1\supset\tilde R$ with
\begin{equation}\label{cardinalityBdTildeFPrimeS}
\# \{ \tilde f\in \tilde F_S(\tilde R)\colon \tilde f\sim \tilde R_1\} \gtrsim \# \tilde F_S(\tilde R) \gtrsim_\eps (\log1/\delta)^{-3}\delta^{\eps_1+O(\eta/\eps)}\mu.
\end{equation}

Undoing the rescaling, we have a curvilinear rectangle of dimensions $\tau\rho\times c^{\frac{1}{k+1}}\tau^{\frac{1}{k+1}}\rho^{1/k}=\tau\rho\times c^{\frac{1}{k+1}}\tau^{\frac{-1}{k(k+1)}}(\tau\rho)^{1/k}$; i.e., we have a $(\tau\rho; k; c^{\frac{k}{k+1}}\tau^{\frac{-1}{k+1}})$ tangency rectangle $R_1\supset R$, with 

\begin{equation}\label{concentratedInRect}
\# \{ f\in  F_S( R)\colon f\sim  R_1\}\gtrsim_\eps (\log1/\delta)^{-3}\delta^{\eps_1+O(\eta/\eps)}\mu.
\end{equation}

Finally, define $\rho_1 =\tau\rho$ and define $T=c^{\frac{k}{k+1}}\tau^{\frac{-1}{k+1}}\gtrsim \tilde\delta^{\frac{-c}{k+1}} \geq \delta^{\frac{-c \eps}{k+1}}$. Since the rectangles in $\mathcal{R}$ are $\mu$-rich and $\eps$-robustly broad with error $\delta^{-\eta}$, by \eqref{fExtendedRectangleBound} we have
\begin{equation}\label{upperBdCurvesInR1}
\# \{ f\in  F_S( R)\colon f\sim  R_1\} \leq \delta^{-\eta}T^{-\eps}(\#F(R)) \lesssim \delta^{-\eta+\frac{c \eps^2}{k+1}}\mu.
\end{equation}
Comparing \eqref{concentratedInRect} and \eqref{upperBdCurvesInR1}, we obtain a contradiction provided we select $\eps_1$ and $\eta$ sufficiently small depending on $\eps$ and $c$ (recall that $c$ in turn depends on $k$), and provided $\delta>0$ is sufficiently small. 

This contradiction shows that \eqref{ellLong} cannot hold. This completes the proof of Proposition \ref{tangencyRectanglesAlgCurvesProp}, except that we still need to prove Proposition \ref{tangencyRectanglesInsideSemiAlgSetProp}. This will be done in the next section. 
\end{proof}


\section{Tangencies inside a semi-algebraic set of small volume}\label{tangenciesSemiSmallVolSec}

In this section we will prove Proposition \ref{tangencyRectanglesInsideSemiAlgSetProp}. We begin by establishing a decomposition theorem for semi-algebraic sets with small volume.


\subsection{Covering semi-algebraic sets with thin neighborhoods of Lipschitz graphs}

In this section, we will show that a semi-algebraic set $W\subset[-1,1]^{n+1}$ with small volume can be covered by a controlled number of thin neighborhoods of Lipschitz graphs, plus a set that has small projection to $[-1,1]^n$. The precise statement is as follows. Throughout this section, all implicit constants may depend only on the dimension $n$. We write $A = \poly(B)$ to mean $A\leq C B^C$, where the constant $C$ may depend on the ambient dimension $n$.

\begin{prop}\label{coverByLipschitzGraph}
Let $M\geq 1$ and let $W\subset [-1,1]^{n+1}$ be a semi-algebraic set of complexity at most $M$. Let $0<u\leq 1$ and $L\geq 1$. Then we can cover $W$ by a collection of sets,
\begin{equation}\label{SCoveringEqn}
W \subset \bigcup_{i=0}^N W_i,
\end{equation}
with the following properties:
\begin{itemize}

\item[(i)]  $N = \poly(M).$

\item[(ii)] $W_0 = T_0\times [-1,1]$, where $T_0\subset [-1,1]^n$ is semi-algebraic with complexity $\poly(M)$, and 
\begin{equation}\label{boundOnSizeT0}
|T_0|\leq \poly(M)\big(L^{-1/n} + |W|/u\big).
\end{equation}

\item[(iii)] For each index $i=1,\ldots,N$, $W_i$ is of the form
\begin{equation}\label{WiOfTheForm}
W_i = \{(\underline x, x_{n+1})\colon \underline x \in T_i,\ f_i(\underline x)< x_{n+1} < f_i(\underline x)+u\},
\end{equation}
where $T_i\subset [-1,1]^n$ is semi-algebraic with complexity $\poly(M)$, and $f_i\colon [-1,1]^n\to[-1,1]$ is $L$-Lipschitz.
\end{itemize}
\end{prop}

One of the main tools that we will use to prove Proposition \ref{coverByLipschitzGraph} is the cylindrical algebraic decomposition. This is a technique from real algebraic geometry that was originally developed in the context of quantifier elimination.  The cylindrical algebraic decomposition decomposes an arbitrary semi-algebraic set into simpler sets, which are called cells\footnote{these are not to be confused with the connected components of $\RR^{k+1}\backslash \{Q=0\}$ from Section \ref{proofOfProptangencyRectanglesAlgCurvesPropPt1Sec}, which are also called cells.}. The standard references for this material are the textbooks by Bochnak, Coste, and Roy \cite{BCR} and Basu, Pollack, and Roy \cite{BPR}.

We begin with an informal definition of a semi-algebraic cylindrical decomposition of $\RR^n$; see \cite[Definition 5.1]{BPR} for a precise definition. In what follows, if $W\subset\RR^{n}$ is a semi-algebraic set, we say that $f\colon W\to\RR$ is a \emph{semi-algebraic function} if $\operatorname{graph}(f)$ is a semi-algebraic set. A \emph{cylindrical decomposition} of $\RR^n$ is a sequence $\mathcal{W}_1,\ldots,\mathcal{W}_n$, where for each index $j$, $\mathcal{W}_j$ is a partition of $\RR^j$ into semi-algebraic sets---these are called the cells of level $j$. Each cell $W\in\mathcal{W}_1$ is either a point or an open interval. For $1<j\leq n$, each cell of level $j$ is of the form 
\begin{equation}\label{UDecompType1}
U = \big\{(\underline x, x_j)\colon \underline x \in T,\ f(\underline x) < x_j < g(\underline x)\big\},
\end{equation}
or
\begin{equation}\label{UDecompType2}
U = \big\{(\underline x, x_{j})\colon \underline x \in T,\ x_j=f(\underline x)\big\}.
\end{equation}
In the above, $T$ is a cell from $\mathcal{W}_{j-1}$, and either $f$ is continuous and semi-algebraic, or else $f=-\infty$; similarly either $g$ is continuous and semi-algebraic, or else $g = \infty.$

We say that a cylindrical algebraic decomposition $\mathcal{W}_1,\ldots,\mathcal{W}_n$ of $\RR^n$ is \emph{adapted} to a semi-algebraic set $W$ if $W$ is a union of cells. Theorem 5.6 from \cite{BPR} says that for every semi-algebraic set $W\subset\RR^n$, there exists a cylindrical decomposition of $\RR^n$ adapted to $W$. In Section 11 (see specifically Algorithm 11.2) from \cite{BPR}, the authors describe an algorithm for computing a cylindrical decomposition adapted to a semi-algebraic set $W$, and they analyze this algorithm to show that the sum of the complexities\footnote{in \cite{BPR}, the term ``complexity'' often refers to the ``time complexity,'' (i.e.~the number of computational steps required to compute the decomposition $\mathcal{W}_1,\ldots,\mathcal{W}_n$) rather than complexity in the sense of Definition \ref{complexitySemiAlgSet}. However, the authors in \cite{BPR} also bound the number and degrees of the polynomials needed to describe the sets in $\mathcal{W}_1,\ldots,\mathcal{W}_n$; these latter quantities give an upper bound for the complexity in the sense of Definition \ref{complexitySemiAlgSet}.} of the cells in $\mathcal{W}_1,\ldots,\mathcal{W}_n$ is bounded by $\poly(M)$, where $M$ is the complexity of $W$. We remark that the complexity of the cylindrical algebraic decomposition is often referred to as ``doubly exponential.'' However, this is doubly exponential in $n$ (the dimension of the ambient space, or equivalently the number of variables); for $n$ fixed, the complexity is polynomial in the complexity of $W$.

Given a cylindrical decomposition $\mathcal{W}_1,\ldots,\mathcal{W}_n$, it is straightforward to find a new decomposition $\mathcal{V}_1,\ldots,\mathcal{V}_n$ such that the functions $f$ and $g$ from \eqref{UDecompType1} (resp.~the function $f$ from \eqref{UDecompType2}) satisfy $F(\underline x, f(\underline x))=0$ and $G(\underline x, g(\underline x)) = 0$ (resp.~$F(\underline x, f(\underline x))=0$) for each $x\in T$, where $F$ and $G$ (resp.~$F$) are polynomials whose degrees are bounded by $\poly(M)$ (here as above, $M$ is the complexity of $W$). The key step is that $T$ can be partitioned into a controlled number of semi-algebraic sets (each of controlled complexity), so that $f$ and $g$ (resp.~$f$) are of the claimed form on each of these sets. See e.g.~Lemma 2.6.3 from \cite{BCR}. In summary, we have the following


\begin{thm}[Effective cylindrical algebraic decomposition]\label{CADThm}
Let $W\subset\RR^{n+1}$ be a bounded, semi-algebraic set of complexity at most $M$ (see Definition \ref{complexitySemiAlgSet}). Then there exists a decomposition $W = \bigsqcup_{i=0}^N W_i$, with $N=\poly(M)$, where the sets $W_i$ have the following properties:
\begin{itemize}
\item[(i)] Each $W_i$ is semi-algebraic of complexity $\poly(M)$.
\item[(ii)] The projection of $W_0$ to the first $n$ coordinates is a semi-algebraic set of measure 0 and complexity  $\poly(M)$.
\item[(iii)] For each $i=1,\ldots,N$, the set $W_i$ is of one of the following two forms:
\begin{equation}\label{WiType1}
W_i = \big\{(\underline x, x_{n+1})\colon \underline x \in T_i,\ f_i(\underline x) < x_{n+1} < g_i(\underline x)\big\},
\end{equation}
or
\begin{equation}\label{WiType2}
W_i = \big\{(\underline x, x_{n+1})\colon \underline x \in T_i,\ x_{n+1}=f_i(\underline x)\big\}.
\end{equation}
In the representations \eqref{WiType1} and \eqref{WiType2} above, $T_i\subset\RR^n$ is an open semi-algebraic set of complexity $\poly(M)$; the function $f_i\colon T_i\to\RR$ is smooth; and there is a nonzero polynomial $F_i\colon\RR^{n+1}\to\RR$ of degree $\poly(M)$ so that
\[
F_i(\underline x, f_i(\underline x))=0\quad \textrm{and}\quad \partial_{x_{n+1}}F_i(\underline x, f_i(\underline x))\neq 0\qquad\textrm{for all}\ \underline x\in T_i.
\] 
The function $g_i\colon T_i\to\RR$ satisfies the analogous conditions. 
\end{itemize}
\end{thm}

Next, we will recall several definitions and results from \cite{Pawlucki}. Note that \cite{Pawlucki} works in the slightly more general context of an $o$-minimal structure on a real closed field; we specialize these results to the special case of semi-algebraic sets over the field $\RR$. The following are Definitions 1--4 from \cite{Pawlucki}.
\begin{defn}\label{defnLFunction}
Let $L\geq 0$, let $T\subset\RR^{n}$ be open, and let $f\colon T\to\RR$ be continuous and semi-algebraic. We say that $f$ is an $L$-\emph{function} if 
\[
|\partial_{x_j} f(x)|\leq L,\quad j=1,\ldots,n,
\]
for all points $x\in T$ for which $f$ continuously differential in a neighborhood of $x$. 
\end{defn}

\begin{defn}\label{defnLCell}
A bounded set $S\subset\RR^{n+1}$ is called an $L$-\emph{cell} if it is of the form
\begin{equation}\label{SMCell}
S = \{(\underline x, x_{n+1})\colon \underline x \in T,\ f(\underline x)< x_{n+1}< g(\underline x) \},
\end{equation}
where $T\subset\RR^n$ is open and semi-algebraic, and $f,g$ are $L$-functions. 
\end{defn}

\begin{defn}\label{defnRegularLCell}
A bounded set $S\subset\RR^{n+1}$ is called a \emph{regular} $L$-\emph{cell} if:
\begin{itemize}
	\item It is an open interval in the case $n=0$.
	\item For $n\geq 1$, $S$ is an $L$-cell, and the projection of $S$ to the first $n$ coordinates is a regular $L$-cell (i.e.~the set $T$ from \eqref{SMCell} is a regular $L$-cell).
\end{itemize}
\end{defn}

With these definitions, we can now state Proposition 1 from \cite{Pawlucki}. The following result says that every $C^1$ $L$-function whose domain is a regular $L$-cell is Lipschitz.
\begin{prop}\label{LFunctionsAreAlmostLipschitz}
Let $T$ be a regular $L$-cell in $\RR^{n}$, and let $f\colon T\to\RR$ be a continuously differentiable $L$-function. Then for all $x,y\in T$, we have
\begin{equation}
|f(x) - f(y)| \leq n! L^n |x-y|.
\end{equation}
\end{prop}
We now begin the process of proving Proposition \ref{coverByLipschitzGraph}. To start, we will study structural properties of the cells arising from the cylindrical algebraic decomposition.

\begin{lem}\label{smallSetBigDerivative}
Let $M\geq 1$ and let $T\subset[-1,1]^n$ be an open semi-algebraic set of complexity at most $M$. Let $f\colon T\to[-1,1]$ be differentiable, and suppose there is a nonzero polynomial $F\colon\RR^{n+1}\to\RR$ of degree at most $M$ so that $F(\underline x, f(\underline x))=0$ and $\partial_{x_{n+1}}F(\underline x, f(\underline x))\neq 0$ for all $\underline x \in T$. Finally, let $L>0$ and suppose that for each point $\underline x \in T$, there is an index $1\leq i\leq n$ so that $|\partial_{x_i}f(\underline x)|\geq L$. 

Then
\begin{equation}\label{boundOnSizeT}
|T|\lesssim  M^2/L.
\end{equation}
\end{lem}
\begin{proof}

For $i=1,\ldots n$ define
\begin{align*}
T_i &= \Big\{\underline x\in T\colon \ \Big|\frac{\partial_{x_i} F(\underline x)}{\partial_{x_{n+1}} F(\underline x)}\Big|\geq L\Big\}\\
&=\Big\{\underline x\in T\colon\ \big(\partial_{x_i} F(\underline x)\big)^2 \geq L^2\big(\partial_{x_{n+1}} F(\underline x)\big)^2\Big\}.
\end{align*}
We have
\begin{equation}\label{mostOfTCoveredByTi}
T =  \bigcup_{i=1}^n T_i,
\end{equation}
and each set $T_i$ has complexity at most $2M.$ 

Fix an index $i$. By Fubini's theorem, we can select a line $\ell\subset\RR^n$ pointing in the $e_i$ direction with $|\ell\cap T_i|\geq |T_i|/2$ (here we use the fact that $T_i\subset[-1,1]^n$; the $|\cdot|$ on the LHS denotes one-dimensional Lebesgue measure, while the $|\cdot|$ on the RHS denotes $n$-dimensional Lebesgue measure). Since $T_i$ has complexity at most $2M$ (i.e.~it is described by a Boolean formula involving at most $2M$ polynomials, each of degree at most $2M$), $\ell\cap T_i$ contains at most $4M^2$ connected components. Thus we can select an interval $\ell'\subset \ell\cap T_i$ that has length at least $|T_i|/(8M^2)$. But since $|\partial_{x_i}f|\geq L$ on $T_i$, we have $|f(a)-f(b)|\geq L|T_i|/(8M^2)$, where $a$ and $b$ are the endpoints of $\ell'$. On the other hand, $f(a),f(b)\in[-1,1]$. We conclude that
\[
\frac{L|T_i|}{8M^2}\leq |f(a)-f(b)|\leq 2.
\]
Re-arranging we have $|T_i|\leq 16M^2/L$. Summing over $i$ we obtain \eqref{boundOnSizeT}. 
\end{proof}

\begin{lem}\label{bigSmallNablaDecompLem}
Let $M\geq 1$ and let $T\subset[-1,1]^n$ be an open semi-algebraic set of complexity at most $M$. Let $f\colon T\to[-1,1]$ be differentiable, and suppose there is a nonzero polynomial $F$ of degree at most $M$ so that $F(\underline x, f(\underline x))=0$ and $\partial_{x_{n+1}}F(\underline x, f(\underline x))\neq 0$ for all $\underline x \in T$.

Let $L>0$. Then we can write $T = T'\cup T''$, where
\begin{itemize}
\item[(i)] $T'$ and $T''$ are semi-algebraic of complexity $O(M)$. 
\item[(ii)] $T'$ is open, and the restriction $f\colon T'\to\RR$  is an $L$-function.
\item[(iii)] $|T''|\lesssim M^2/L$.
\end{itemize}
\end{lem}
\begin{proof}
Let
\begin{equation*}
\begin{split}
T_2 &= \bigcup_{i=1}^n \big\{ \underline x\in T\colon (\partial_{x_i} F(\underline x))^2\geq L^2(\partial_{x_{n+1}} F (\underline x))^2  \big\},\\
T_1 & =T\backslash T_2.
\end{split}
\end{equation*}
By the implicit function theorem, there is a neighbourhood $U$ of $T_1$ and an extension of $f$ to $U$ that satisfies $F(\underline x, f(\underline x))=0$ for all $\underline x\in U$. By the Leibniz rule, we have that $f$ is differentiable and satisfies $|\partial_{x_i}f(\underline x)|\leq L$ for each $i=1,\ldots n$ on $T_1$. Finally, by Lemma \ref{smallSetBigDerivative} we have $|T_2|\lesssim M^2/L$. To finish the proof, let $T'=\operatorname{int}(T_1)$ and $T'' = T\backslash T'$.
\end{proof}

Combining Theorem \ref{CADThm} and Lemma \ref{bigSmallNablaDecompLem}, we have the following.
\begin{lem}\label{oneStageDecompSets}
Let $M\geq 1$ and let $W\subset [-1,1]^{n+1}$ be a semi-algebraic set of the form
\begin{equation}\label{WOfTheForm1}
W = \{(\underline x, x_{n+1})\colon \underline x \in T,\ f(\underline x)< x_{n+1} < g(\underline x)\},
\end{equation}
or
\begin{equation}\label{WOfTheForm2}
W = \{(\underline x, x_{n+1})\colon \underline x \in T,\ x_{n+1} = f(\underline x)\},
\end{equation}

where $T\subset [-1,1]^n$ is an open semi-algebraic set of complexity at most $M$; $f\colon T\to[-1,1]$ is smooth; and there is a nonzero polynomial $F\colon\RR^{n+1}\to\RR$ of degree at most $M$ so that
\[
F(\underline x, f(\underline x))=0\quad \textrm{and}\quad \partial_{x_{n+1}}F(\underline x, f(\underline x))\neq 0\qquad\textrm{for all}\ \underline x\in T.
\] 
Suppose $g$ satisfies the analogous condition (if $W$ is of the form \eqref{WOfTheForm1}).

Then we can cover $W$ by a collection of sets
\[
W \subset\bigcup_{i=0}^N W_i,
\]
where 
\begin{itemize}
\item[(i)] $N = \poly(M)$.
\item[(ii)] The projection of $W_0$ to the first $n$ coordinates is semi-algebraic of complexity $\poly(M)$ and has measure at most $\poly(M)/L$.
\item[(iii)] For each $i=1,\ldots,N$, $W_i$ is of the form
\begin{equation}\label{WiType1InLem}
W_i = \big\{(\underline x, x_{n+1})\colon \underline x \in T_i,\ f_i(\underline x) < x_{n+1} < g_i(\underline x)\big\},
\end{equation}
or
\begin{equation}\label{WiType2InLem}
W_i = \big\{(\underline x, x_{n+1})\colon \underline x \in T_i,\ x_{n+1}=f_i(\underline x)\big\}.
\end{equation}
In the above, $T_i\subset\RR^n$ is an open semi-algebraic set of complexity $\poly(M)$; the function $f_i\colon T_i\to[-1,1]$ is smooth; and there is a nonzero polynomial $F_i\colon\RR^{n+1}\to\RR$ of degree $\poly(M)$ so that
\[
F_i(\underline x, f_i(\underline x))=0\quad \textrm{and}\quad \partial_{x_{n+1}}F_i(\underline x, f_i(\underline x))\neq 0\qquad\textrm{for all}\ \underline x\in T_i.
\] 
The function $g_i\colon T_i\to\RR$ satisfies the analogous conditions. 
\item[(iv)] For each $i=1,\ldots,N$, the set $T_i$ from \eqref{WiType1InLem} (resp.~\eqref{WiType2InLem}) is open, and the restriction of $f$ (if $W$ is of the form \eqref{WOfTheForm2}) or $f$ and $g$ (if $W$ is of the form \eqref{WOfTheForm1}) to $T_i$ is an $L$-function. 
\end{itemize}
\end{lem}

Iterating Lemma \ref{oneStageDecompSets}, we have the following
\begin{lem}\label{iteratedOneStageDecompSets}
Let $W\subset[-1,1]^{n+1}$ be a semi-algebraic set of complexity at most $M$. Then we can cover $W$ by a collection of sets
\[
W \subset\bigcup_{i=0}^N W_i,
\]
where the sets $\{W_i\}$ satisfy Items (i), (ii), and (iii) from the conclusion of Lemma \ref{oneStageDecompSets}. In place of Item (iv), we have the following:

\begin{itemize}
\item[(iv$\:'$)] For each $i=1,\ldots,N$, the set $T_i$ from \eqref{WiType1InLem} (resp.~\eqref{WiType2InLem}) is a regular $L$-cell, in the sense of Definition \ref{defnRegularLCell}, and $f_i\colon T_i\to[-1,1]$ is an $L$-function, in the sense of Definition \ref{defnLFunction}. If $W_i$ is of the form \eqref{WiType1InLem}, then the function $g_i\colon T_i\to\RR$ is also an $L$-function.
\end{itemize}
\end{lem}

\noindent We are now ready to prove Proposition \ref{coverByLipschitzGraph}. 
\begin{proof}[Proof of Proposition \ref{coverByLipschitzGraph}]
Apply Lemma \ref{iteratedOneStageDecompSets} to $W$, with $(L/n!)^{1/n}$ in place of $L$. Let $W \subset\bigcup_{i=0}^N A_i$ be the resulting decomposition (i.e.~we will label the sets $A_0,\ldots,A_N$ rather than $W_0,\ldots,W_N$).

The projection of $A_0$ to the first $n$ coordinates has measure $\poly(M) L^{-1/n}$; denote this set by $C_0'$.
 
For each $i=1,\ldots,N$, the cell $A_i$ is of the form
\begin{equation}\label{AiType1}
A_i = \{(\underline x, x_{n+1})\colon \underline x \in B_i,\ f_i(\underline x) < x_{n+1} < g_i(\underline x)\},
\end{equation}
or
\begin{equation}\label{AiType2}
A_i = \{(\underline x, x_{n+1})\colon \underline x \in B_i,\ x_{n+1} = f_i(\underline x)\},
\end{equation}
where $B_i$ is a regular $(L/n!)^{1/n}$ cell and $f_i$ is an $(L/n!)^{1/n}$-function. By Proposition \ref{LFunctionsAreAlmostLipschitz}, we have that $f_i\colon B_i\to[-1,1]$ is $L$-Lipschitz.  By the Kirszbraun-Valentine Lipschitz extension theorem \cite{Kirs} (see \cite[Theorem 1.31]{Schw} for a proof in English), we can extend each $f_{i}$ to an $L$-Lipschitz function with domain $[-1,1]^n$.

For each cell $A_i$ of the form \eqref{AiType1}, write $B_i = C_i \sqcup C_i'$, where
\[
C_i = \{\underline x \in B_i\colon |g(\underline x) - f(\underline x)| \leq u \}.
\]
We have $u|C_i'|\leq |A_i|.$ Define
\[
W_i = \{(\underline x, x_{n+1})\colon \underline x \in C_i,\ f_i(\underline x)\leq x_{n+1}\leq f_i(\underline x)+u\}.
\] 
Then
\[
A_i \subset W_i \cup (C_i'\times[-1,1]).
\]

For each cell $A_i$ of the form \eqref{AiType2}, define $C_i=B_i$ and $C_i'=\emptyset.$ Define
\[
W_i = \{(\underline x, x_{n+1})\colon \underline x \in C_i,\ f_i(\underline x)-u/2\leq x_{n+1}\leq f_i(\underline x)+u/2\}.
\]

We have 
\begin{equation}\label{boundDiPrime}
\sum_{i=1}^N |C_i'|\leq u^{-1}\sum_{i=1}^N |A_i| = |W|/u.
\end{equation}

To conclude the proof, we define 
\[
T_0 = \bigcup_{i=0}^N C_i',
\]
and $W_0 = T_0\times[-1,1]$. Then $T_0$ is semi-algebraic of complexity $\poly(M)$, and has measure $\poly(M)(L^{-1/n}+|W|/u)$.
\end{proof}


\subsection{Jet lifts in thin neighborhoods of Lipschitz graphs}\label{jetLiftSection}
Our goal in this section is to prove the following result. In what follows, recall that the jet lift $\mathcal{J}_jf$ is given in Definition \ref{jetLiftDefn}.

\begin{prop}\label{cutCurvesAdaptedSASet}
For each $n\geq 0$ and $\kappa,\rho>0$, there are constants $A = A(n)$ and $B=B(n,\kappa),$ and a set $\mathcal{X}\subset[\rho, \rho^{1/B}]$ (depending on $n,\kappa,$ and $\rho$) of cardinality $n+1$ so that the following holds. Let $D,M\geq 1$. Let $W\subset[-1,1]^{n+2}$ be a semi-algebraic set of complexity at most $M$ and volume at most $\rho$. 

Then for each polynomial $f$ of degree at most $D$, there is a ``bad'' set $B_f\subset[0,1]$, which is a union of at most $A(DM)^A$ intervals and has measure at most $A (DM)^A \rho^{1/B}$, so that the following holds. 

Let $f,g$ be polynomials of degree at most $D$. Suppose there is a point $t_0\in [0,1]\backslash(B_f \cup B_g)$ that satisfies
\begin{equation}
\begin{split}
&(t_0, \mathcal{J}_{n}f(t_0))\in W,\quad (t_0, \mathcal{J}_{n}g(t_0))\in W, \label{fCloseToGAndliftOfFInS}\\
&|\mathcal{J}_n f(t_0)-\mathcal{J}_n g(t_0)|\leq \rho.
\end{split}
\end{equation}
Then there is a number $\tau\in \mathcal{X}$ so that
\begin{equation}\label{fRhoCloseToG}
|f(t)-g(t)|\leq  A \tau,\qquad t \in [t_0-\tau^\kappa,\ t_0 + \tau^\kappa].
\end{equation}
\end{prop}

Gronwall's inequality will play an important role in the proof of Proposition \ref{cutCurvesAdaptedSASet}. We will use the following formulation. See, e.g., \cite{HowardGron} for a discussion and proof of this version.

\begin{thm}[Gronwall's inequality]\label{gronwallThm}
Let $n\geq 1$, let $I$ be an interval, let $F,G\colon I\times \RR^n\to\RR$, let $t_0\in I$, let $\tilde{\underline x}, \tilde{\underline y} \in \RR^n$, and let $f,g\colon I\to\RR$ satisfy the initial value problems
\begin{equation}
\begin{split}
f^{(n)}(t) &=  F\big(t, \mathcal{J}_{n-1}f(t)\big), \quad \mathcal{J}_{n-1}f(t_0) = \tilde{\underline x},\\
g^{(n)}(t) &=  G\big(t, \mathcal{J}_{n-1}g(t)\big), \quad \mathcal{J}_{n-1}g(t_0) = \tilde{\underline y}.
\end{split}
\end{equation}
Suppose that for $t$ fixed, $F$ is $L$-Lipschitz in $\underline x$, i.e.
\[
|F(t, \underline x)-F(t, \underline x')|\leq L|\underline x-\underline x'|,\qquad t\in I,\ \underline x, \underline x'\in \RR^n.
\]
Let $\rho>0$. Suppose that $|\tilde{\underline x}-\tilde{\underline y}|\leq\rho$, and 
\begin{equation}\label{FGClose}
|F(t, \underline x) - G(t, \underline x)|\leq \rho,\qquad t\in I,\ \underline x\in \RR^n.
\end{equation}

Then
\[
|f(t)-g(t)|\lesssim e^{L|I|}\rho,\qquad t\in I.
\]
\end{thm}

The following result is a variant of Theorem \ref{gronwallThm}. Instead of requiring that $f$ and $g$ satisfy ``nearby'' initial value problems, in the sense of \eqref{FGClose}, we require that $f$ satisfies the initial value problem $f^{(n)}=F(t,\mathcal{J}_{n-1}f)$, and $g$ almost satisfies this same initial value problem, in the sense that $|g^{(n)}-F(t,\mathcal{J}_{n-1}g)|$ is small. The precise statement is as follows. 

\begin{lem}\label{gronwallCloseToF}
$\phantom{1}$
\begin{itemize}
    \item Let $n\geq 1$, let $I$ be an interval, and let $g\in C^n(I)$.
    \item Let $F\colon I\times\RR^n \to\RR$ be $L$-Lipschitz, let $\rho>0$, and suppose that
    \begin{equation}\label{jetGinS}
    \big|g^{(n)}(t) - F\big(t, \mathcal{J}_{n-1}g(t)\big)\big| \leq \rho,\qquad t\in I.
    \end{equation}

    \item Let $t_0\in I$, let $\tilde{\underline y}=\mathcal{J}_{n-1}g(t_0)$, and let $\tilde{\underline x}\in \RR^n$, with $|\tilde{\underline x}-\tilde{\underline y} |\leq\rho$. 

    \item Let $f\colon I\to\RR$ be a solution to the initial value problem
        \begin{equation}
        f^{(n)}(t) =  F\big(t, \mathcal{J}_{n-1}f(t)\big),\qquad \mathcal{J}_{n-1}f(t_0)=\tilde{\underline x}.
        \end{equation}
\end{itemize}
Then
\begin{equation}\label{gMinusFBound}
|g(t)-f(t)|\lesssim e^{L|I|}\rho,\quad t\in I.
\end{equation}
\end{lem}

\begin{proof}
Define 
\begin{equation*}
\begin{split}
G(t, \underline x)&= F(t, \underline x)+e(t),\\
e(t) & = g^{(n)}(t) - F\big(t, \mathcal{J}_{n-1}g(t)\big).
\end{split}
\end{equation*}
The quantity $e(t)$ is intended to measure the error between the initial value problems $F$ and $G$. Inequality \eqref{jetGinS} says that $|e(t)|\leq\rho$ for $t\in I$, and thus
\begin{equation}
\big|G(t, \underline x)-F(t, \underline x)\big|\leq\rho, \qquad t\in I,\ \underline x \in \RR^n.
\end{equation}
But $g\colon I\to\RR$ is the solution to the initial value problem
       \begin{equation}
       g^{(n)}(t) =  G\big(t, \mathcal{J}_{n-1}g(t)\big),\qquad \mathcal{J}_{n-1}g(t_0)=\tilde{\underline y}.
        \end{equation}
Thus by Theorem \ref{gronwallThm}, we have
\begin{equation}
|f(t)-g(t)|\lesssim e^{L|I|}\rho\quad\textrm{for all}\ t\in I.
\end{equation} 
\end{proof}


\begin{lem}\label{polyCurvesCloseForever}
$\phantom{1}$
\begin{itemize}
    \item Let $n,L\geq 1$, let $I$ be an interval, and let $f,g \in C^n(I)$.

    \item Let $F\colon I\times\RR^n \to\RR$ be L-Lipschitz, let $\rho>0$, and suppose that
    \begin{equation}\label{jetGHinS}
    \begin{split}
    & \big|f^{(n)}(t) - F\big(t, \mathcal{J}_{n-1}f(t)\big)\big| \leq \rho,\qquad t\in I,\\
    & \big|g^{(n)}(t) - F\big(t, \mathcal{J}_{n-1}g(t)\big)\big| \leq \rho,\qquad t\in I.
    \end{split}
    \end{equation} 

    \item Suppose there exists $t_0\in I$ so that
    \begin{equation}\label{gAndHCloseAtOnePoint}
        |\mathcal{J}_n f(t_0) - \mathcal{J}_n g(t_0)|\leq\rho.
    \end{equation}
\end{itemize}

Then
\begin{equation}\label{gCloseToHGlobally}
|f(t)-g(t)|\lesssim e^{L|I|}\rho,\qquad t\in I.
\end{equation}
\end{lem}
\begin{rem}
In practice, we will use this lemma with intervals $I$ of length $O(L^{-1})$, and thus the RHS of \eqref{gCloseToHGlobally} is $O(\rho)$. 
\end{rem}
\begin{proof} 
Define
\[
\underline z = \frac{1}{2}\Big[ \mathcal{J}_{n-1}f(t_0) + \mathcal{J}_{n-1}g(t_0) \Big].
\]
By \eqref{gAndHCloseAtOnePoint}, we have 
\[
\big|\underline z - \mathcal{J}_{n-1}f(t_0)\big|\leq \rho/2.
\]

We now apply Lemma \ref{gronwallCloseToF}. Let $h\colon I \to\RR$ be the solution to the initial value problem
\[
h^{(n)}(t) =  F\big(t, \mathcal{J}_{n-1} h(t)\big),\qquad \mathcal{J}_{n-1} h( t_0)=\underline z.
\]
By Lemma \ref{gronwallCloseToF}, we have $|h(t) - f(t)|\lesssim e^{L|I|}\rho$ for $t\in I$. But note that the construction of $h$ is symmetric in the functions $f$ and $g$, and thus we also have $|h(t) - g(t)|\lesssim e^{L|I|}\rho$ for $t\in I$. The conclusion \eqref{gCloseToHGlobally} now follows from the triangle inequality. 
\end{proof}

With these tools, we are now ready to prove the main result in this section.

\begin{proof}[Proof of Proposition \ref{cutCurvesAdaptedSASet}]
Without loss of generality, we may suppose that $\rho,\kappa\leq 1$; otherwise we can replace $\kappa$ and/or $\rho$ by $1$ and the conclusion remains valid. For $i=0,\ldots,n$, define $L_i = \rho^{-(\kappa/2)^{i+1}}$ and define $\rho_i = L_i^{-1/\kappa n}$. We select the quantity $B(n,\kappa)$ sufficiently large so that $\rho_n\leq \rho^{1/B}$ and $\rho^{n^{-1} (\kappa/2)^{n+1}}\leq\rho^{1/B}$. With $B$ selected in this way, we have 
\begin{equation}\label{rhoIGeqRho}
\rho_i\in [\rho, \rho^{1/B}],\quad 0\leq i \leq n.
\end{equation}

Define 
\[
\mathcal{X} = \{\rho_0,\ldots\rho_n\}.
\]

We will define an iterative decomposition of $W$, which will take $n+1$ steps. We will call these steps ``step zero'' through ``step $n$.'' $L_i$ will be the allowable Lipschitz constant and $\rho_i$ will be the allowable thickness at step $i$.

For the zeroth step, apply Proposition \ref{coverByLipschitzGraph} to $W$ with $L_0$ in place of $L$, and $u=\rho_0$. We obtain sets $W^0_0, W^{0}_1,\ldots, W^{0}_{N_0}$, with $N_0 = \poly(M)$. For each index $1\leq j\leq N_0$, the set $W^0_j$ is semi-algebraic of complexity at most $\poly(M)$, and is contained in the $\rho_0$ neighborhood of an $L_0$-Lipschitz graph. In addition, $W^0_0 \subset T^0_0\times[-1,1]$, where $T_0^0\subset[-1,1]^{n+1}$ is semi-algebraic of complexity at most $\poly(M)$, and 
\[
|T^0_0|\leq \poly(M)(L_0^{-1/n}+|W|\rho_0^{-1}).
\] 
Since $|W|\leq\rho$ and $\kappa\leq 1$, our choice of $L_0$ and $\rho_0$ ensures that $L_0^{-1/n}\geq |W|\rho_0^{-1}$, and thus $|T^0_0|\leq \poly(M) L_0^{-1/n}$.

We now describe the $i$--th step of our decomposition, $i=1,\ldots,n$. Suppose that there exists a constant $C_{i-1}$ (which may depend on $n$) so that the following holds:
\begin{itemize}
	\item[(i)] $T^{i-1}_0\subset[-1,1]^{n+2-i}$ is a semi-algebraic set of complexity at most $C_{i-1}M^{C_{i-1}}$.
	\item[(ii)] $|T^{i-1}_0| \leq C_{i-1}M^{C_{i-1}}L_{i-1}^{-1/n}$.
\end{itemize}
Apply Proposition \ref{coverByLipschitzGraph} to $T^{i-1}_0$ with $L_i$ in place of $L$, and $u=\rho_i$. We obtain sets $W^i_0, W^{i}_1,\ldots, W^{i}_{N_i}$, where $N_i = \poly(C_{i-1}M^{C_{i-1}})=\poly(M)$. For each index $1\leq j\leq N_i$, the set $W^i_j$ is semi-algebraic of complexity at most $\poly(M)$, and is contained in the $\rho_i$ neighborhood of an $L_i$-Lipschitz graph. In addition, $W^i_0 \subset T^i_0\times[-1,1]$, where $T^i_0\subset[-1,1]^{n+1-i}$ is semi-algebraic of complexity $\poly(C_{i-1}M^{C_{i-1}})$, and  
\begin{equation}\label{Ti0Bd}
\begin{split}
|T^i_0| & = \poly(C_{i-1}M^{C_{i-1}})(L_i^{-1/n}+|T^{i-1}_0| \rho_i^{-1})\\
& =  \poly(C_{i-1}M^{C_{i-1}})\big(L_i^{-1/n}+C_{i-1}M^{C_{i-1}}L_{i-1}^{-1/n}\rho_i^{-1}\big)\\
&=\poly(C_{i-1}M^{C_{i-1}})(L_i^{-1/n}+L_i^{-2/\kappa n}L_i^{1/\kappa n}) \\
& =\poly(C_{i-1}M^{C_{i-1}}) L_i^{-1/n}.
\end{split}
\end{equation}
The first equality is the conclusion of Proposition \ref{coverByLipschitzGraph}. The second equality used Assumption (ii) above. The third equality used the definition of $L_{i-1}$. The fourth equality used the assumption $\kappa\leq 1$ and $\rho\leq 1$, which in turn implies $L_i\geq 1$. 

We will select $C_{i}$ sufficiently large (depending on $n$ and $C_{i-1}$, which in turn depends on $n$) so that the RHS of \eqref{Ti0Bd} is at most $C_i M^{C_i}$, and the complexity of $T^i_0$ is at most $C_i M^{C_i}$.

After the process described above is complete, we have a covering of $W$ of the form
\begin{equation}\label{SDecompose}
W \subset \bigcup_{i=0}^{n+1} \bigcup_{j=1}^{N_i} \big( W^i_j \times [-1,1]^{i}\big),
\end{equation}
where for $i=0,\ldots,n$, each set $W^i_j\subset [-1,1]^{n+2-i}$ is the vertical $\rho_i$ neighborhood of an $L_i$-Lipschitz graph (we denote the associated Lipschitz function by $G^i_j\colon[-1,1]^{n+1-i}\to[-1,1]$) above a set $T^i_j\subset[-1,1]^{n+1-i}$, and for $i=n+1$,  each set $W_j^{n+1}$ is an interval (these are precisely the intervals in the set $T_0^n\subset[-1,1]$). Furthermore, we have the following bound on the sum of the lengths of these intervals:
\begin{equation}
\sum_{j=1}^{N_{n+1}}|W^{n+1}_j| = |T_0^n| =\poly(M)L_n^{-1/n} = \poly(M)\rho^{n^{-1}(\kappa/2)^{n+1}}=\poly(M)\rho^{1/B}.
\end{equation}

We would like to claim that if $f$ and $g$ are two polynomials that satisfy \eqref{fCloseToGAndliftOfFInS} at some point $t_0\in[0,1]$, then the corresponding points $\big(t_0, \mathcal{J}_n f(t_0)\big)$ and $\big(t_0, \mathcal{J}_n g(t_0)\big)$ must be contained in a common set of the form $W^i_j \times [-1,1]^{i}$ from the decomposition \eqref{SDecompose}. Unfortunately this need not be true, since even though \eqref{fCloseToGAndliftOfFInS} guarantees that the points $\mathcal{J}_n f(t_0)$ and $\mathcal{J}_n g(t_0)$ are nearby, they might nonetheless be contained in different sets from \eqref{SDecompose}.

To handle this annoyance, we expand each set $W^i_j$ slightly. For $i=0,\ldots,n$ and $j=1,\ldots,N_i$, define $(W^i_j)^*$ to be the vertical $\rho_i+\rho$ neighborhood of the graph of the Lipschitz function $G^i_j$ above the $\rho$-neighborhood of $T^i_j$. For $i=n+1$ and $j=1,\ldots,N_{n+1}$, define $(W^{n+1}_j)^*$ to be the Cartesian product of the $\rho$-neighborhood of $T^{n+1}_j$ with $[-1,1]^{n+1}$, where $T^{n+1}_j\subset[-1,1]$ is the interval associated to $W^{n+1}_j$, i.e. $W^{n+1}_j = T^{n+1}_j\times[-1,1]^{n+1}$.

With the sets $(W^i_j)^*$ defined in this way, if $f$ and $g$ satisfy \eqref{fCloseToGAndliftOfFInS} at $t_0$, and if $\big(t_0, \mathcal{J}_n f(t_0)\big)\in W^i_j \times [-1,1]^{i}$, then $\big(t_0, \mathcal{J}_n g(t_0)\big)\in (W^i_j)^* \times [-1,1]^{i}$. We will see below why this is useful.

Our next task is to define the ``bad'' set $B_f$ from the statement of Proposition \ref{cutCurvesAdaptedSASet}. Let $f$ be a polynomial of degree at most $D$. Let 
\[
J_f=\big\{t\in [0,1]\colon \big(t, \mathcal{J}_n f(t)\big)\in W\big\}.
\]
$J_f$ is semi-algebraic of complexity $\poly(DM)$, and hence is a union of $\poly(DM)$ intervals. We further sub-divide these intervals into a collection of intervals (here we consider a point to be a closed interval of length 0) that we will denote by $\mathcal{I}_f$. This collection of intervals will have the following properties:
\begin{itemize}
	\item[(a)] $J_f = \bigcup_{J\in \mathcal{I}_f}J$. 
	\item[(b)] For each $J\in \mathcal{I}_f$ and each set of the form $X = W^i_j \times [-1,1]^{i}$ from the decomposition \eqref{SDecompose}, we either have $\big(t, \mathcal{J}_n f(t)\big)\in X$ for all $t\in J$, or $\big(t, \mathcal{J}_n f(t)\big)\not\in X$ for all $t\in J$. 
	\item[(c)] The analogue of Item (b) holds for each set of the form $X = (W^i_j)^* \times [-1,1]^{i}$, where $(W^i_j)^*$ is the expansion of the set $W^i_j$ described in the previous paragraph.
	\item[(d)]  $\#\mathcal{I}_f=\poly(DM)$.
\end{itemize}
To construct the set $\mathcal{I}_f$, we consider each interval $J\subset J_f$ in turn (here $J$ is a connected component of $J_f$). For each set $X$ of the type described in Items (b) and (c) above, we have that the sets 
\[
\big\{t\in J\colon \big(t, \mathcal{J}_n f(t)\big)\in X\big\},\quad \textrm{and}\quad \big\{t\in J\colon \big(t, \mathcal{J}_n f(t)\big)\not\in X\big\} 
\]
are semi-algebraic of complexity $\poly(DM)$, and hence each of these sets is a union of $\poly(DM)$ intervals. We add each of these intervals to $\mathcal{I}_f$.

For each closed interval $J=[a,b]\in \mathcal{I}_f$, if $|J|\leq L_n^{-1}$ then define $\operatorname{Ends}(J)=J$. If $|J|> L_n^{-1}$, then define $\operatorname{Ends}(J)=[a, a+L_n^{-1}] \cup [b-L_n^{-1}, b]$. Define $\operatorname{Ends}(J)$ analogously for intervals of the form $(a, b]$, $[a, b)$ and $(a,b)$. 

Define 
\[
B_f=\bigcup_{J\in\mathcal{J}_f}\operatorname{Ends}(J)\ \cup\ \bigcup_{j=1}^{N_{n+1}}W_j^{n+1}.
\]
If we define the quantity $A=A(n)$ and $B = B(n,\kappa)$ sufficiently large, then $B_f$ is a union of at most $A(DM)^A$ intervals, and has measure at most $\frac{1}{2} A(DM)^AL_n^{-1}+\poly(M)\rho^{1/B} \leq A(DM)^A\rho^{1/B}$.

Our final task is to show that $B_f$ satisfies the conclusion of Proposition \ref{cutCurvesAdaptedSASet}. Let $f,g$ be polynomials of degree at most $D$, and suppose there exists $t_0\in [0,1]\backslash(B_f \cup B_g)$ that satisfies \eqref{fCloseToGAndliftOfFInS}. By \eqref{SDecompose} and our definition of $B_f$, there is an index $0\leq i\leq n$ and $1\leq j\leq N_i$ so that the set $W^i_j\times[-1,1]^i$ contains $\big(t_0, \mathcal{J}_n f(t_0)\big)$. But \eqref{fCloseToGAndliftOfFInS} implies that the expanded set $(W^i_j)^*\times[-1,1]^i$ must contain both $\big(t_0, \mathcal{J}_nf(t_0)\big)$ and $\big(t_0, \mathcal{J}_ng(t_0)\big)$. Furthermore, since $t_0\not\in B_f\cup B_g$, if we define $I = [t_0 - L_i^{-1}, t_0+L_i^{-1}]\subset  [t_0 - L_n^{-1}, t_0+L_n^{-1}]$, then 
\begin{equation}\label{trajectoriesInSij}
\big(t, \mathcal{J}_nf(t)\big)\in (W^i_j)^*\times[-1,1]^i\quad\textrm{and}\quad \big(t, \mathcal{J}_ng(t)\big)\in (W^i_j)^*\times[-1,1]^i, \quad t\in I.
\end{equation}
Since $(W^i_j)^*$ is contained in the vertical $\rho_i+\rho\leq 2\rho_i$ (recall \eqref{rhoIGeqRho}) neighborhood of the $L_i$-Lipschitz function $G^i_j$, \eqref{trajectoriesInSij} implies that
\begin{equation}\label{trajectoriesNearODE}
\begin{split}
\big|f^{(n-i+1)}(t) - G^i_j\big(t, \mathcal{J}_{n-i}f(t)\big)\big|\leq 2\rho_i,\quad t\in I,\\
|g^{(n-i+1)}(t) - G^i_j\big(t, \mathcal{J}_{n-i}g(t)\big)|\leq 2\rho_i,\quad t\in I.
\end{split}
\end{equation}
The function $G^i_j$ from \eqref{trajectoriesNearODE} has Lipschitz constant at most $L_i$, and the interval $I$ has length $2L_i^{-1}$. Thus we can apply Lemma \ref{polyCurvesCloseForever} with $2\rho_i$ in place of $\rho$ and $L_i$ in place of $L$. The conclusion \eqref{gCloseToHGlobally} of Lemma \ref{polyCurvesCloseForever} says that
\[
|f(t)-g(t)|\lesssim e^{L_i|I|}\rho_i\lesssim \rho_i,\quad t\in I.
\]
This is exactly conclusion \eqref{fRhoCloseToG}, provided we select $A = A(n)$ sufficiently large.  
\end{proof}


\subsection{Proof of Proposition \ref{tangencyRectanglesInsideSemiAlgSetProp}}
Apply Proposition \ref{cutCurvesAdaptedSASet} to $W$, with $n = k-1$; $\kappa=1/(k+1)$; $\rho=\delta^{\kappa}$; and $D=M=\delta^{-\eta}$. Note that the quantities $A=A(n)$ and $B=B(n,\kappa)$ coming from Proposition \ref{cutCurvesAdaptedSASet} satisfy the bounds $A = O(1)$ and $B=O(1)$. Define
\begin{equation}\label{defnCPropTangencyRectanglesInsideSemiAlgSetProp}
c = \min\big(A^{-1}, \big(2(k+1)B\big)^{-1}\big);
\end{equation}
this will be our choice of $c = c(k)$ for the conclusion of Proposition \ref{tangencyRectanglesInsideSemiAlgSetProp}.

Since $\rho = \delta^{\frac{-1}{k(k+1)}}\delta^{\frac{1}{k}}$, if we choose the constant $\delta_0$ sufficiently small depending on $k$, then by Lemma \ref{tangencyPrismVsRect} we can ensure that whenever $f,g\in F$ are tangent to a common rectangle $R=R(I)$ from $\mathcal{R}$, there is a point $t_0\in I$ with 
\begin{equation}\label{JKm1FvsG}
|\mathcal{J}_{k-1}f(t_0) -  \mathcal{J}_{k-1}g(t_0)| \leq \rho.
\end{equation} 

By Proposition \ref{cutCurvesAdaptedSASet}, for each $f\in F$ we obtain a bad set $B_f$, which is a union of at most $A(DW)^A=\delta^{-O(\eta)}$ intervals and has measure at most $A (DW)^A \rho^{1/B}\lesssim \delta^{-O(\eta)} \delta^{1/O(1)}$. Thus if $\eta=\eta(k,\eps)$ and $\delta_0(k,\eps)$ are selected sufficiently small, then by \eqref{twoEndsAlongEachCurveInProp} we have
\begin{equation}\label{mostRectanglesGood}
 \# \{R\in\mathcal{R}(f) \colon I(R) \cap B_f\neq\emptyset \}\leq \frac{1}{2} \#\mathcal{R}(f).
\end{equation}

By pigeonholing, there exists a rectangle $R\in\mathcal{R}$ so that
\begin{equation}\label{popularR}
\#\{f\in F(R)\colon I(R) \cap B_f = \emptyset\} \geq \frac{1}{2}\#F(R).
\end{equation}
Fix such a rectangle $R$, and let $F'(R)$ denote the set on the LHS of \eqref{popularR}. For each $f,g\in F'(R)$, there is a point $t_0\in I$ that satisfies \eqref{fCloseToGAndliftOfFInS}; the first item in \eqref{fCloseToGAndliftOfFInS} follows from Lemma \ref{tangencyPrismVsRect} and the hypothesis that $f\sim R$ and $g\sim R$, while the second item follows from \eqref{JKm1FvsG}. By the definition of $F'(R)$, this point $t_0$ is not contained in $B_f$ or $B_g$. Thus by Proposition \ref{cutCurvesAdaptedSASet}, there is a scale $\tau=\tau(f,g)\in \mathcal{X}$ (recall that $\mathcal{X}\subset[\rho, \rho^{1/B}]$) so that
\begin{equation}\label{fCloseToG}
|f(t)-g(t)|\leq A \tau,\quad t\in [t_0-\tau^{1/(k+1)}, t_0+\tau^{1/(k+1)}]\supset I_\tau,
\end{equation}
where $I_\tau$ is the interval of length $(cA\tau)^{1/(k+1)}\leq \tau^{1/(k+1)}$ (recall that $c$ is defined in \eqref{defnCPropTangencyRectanglesInsideSemiAlgSetProp}) with the same midpoint as $I(R)$.

Since $\#\mathcal{X}\leq k$, by pigeonholing we can select a choice of $\tau\in \mathcal{X}$; a choice of $f\in F'(R)$; and a set $F''(R)\subset F'(R)$ with $\#F''(R)\geq k^{-1}(\#F'(R))$, so that \eqref{fCloseToG} holds for this choice of $\tau$ for all $g\in F''(R)$. 

To finish the proof, define $\tilde\tau = A\tau$. Define $R_1$ to be the vertical $\tilde\tau$ neighborhood of the graph of $f$ restricted to $I_\tau$; this is an interval of length $(c\tilde\tau)^{1/(k+1)}$. Thus $R_1$ is a $(\tilde\tau; k+1; c)$ tangency rectangle.  Since $f\sim R,$ $I(R)\subset I_\tau$, and $A\tau\geq 2\delta$, by the triangle inequality we have $R_1\supset R$. By \eqref{fCloseToG}, we have $g\sim R_1$ for all $g\in F''(R)$; the latter set satisfies $\#F''(R)\gtrsim\#F(R)$. 

All that remains is to ensure that $\tilde\tau \in [\delta, \delta^c]$. Recall that $A,B = O(1)$; $\rho = \delta^{\frac{1}{k+1}}$, $\tau \in [\rho, \rho^{1/B}]$, and $\tilde\tau = A\tau\leq A\rho^{1/B} \leq A \delta^{\frac{1}{(k+1)B}}$.  If $\delta_0$ is selected sufficiently small, then this latter quantity is at most $\delta^{\frac{1}{2(k+1)B}}\leq \delta^c,$ as desired. 


\section{From rectangle tangencies to maximal functions}\label{tangenciesToMaximalFnSection}
In this section we will use Theorem  \ref{RectBoundThm} to prove Proposition \ref{maximalFnBdNonconcentrated}. Our proof will follow the outline sketched in Section \ref{introProofSketch}. The next result helps us find the scale ``$\rho$'' discussed in the proof sketch (in what follows, this quantity will be called $\delta'$). Our proofs will involve repeated use of dyadic pigeonholing, which will induce refinements by factors of $(\log 1/\delta)^{O(1)}$ (recall that $O(1)$ denotes a quantity that may depend on $k$). To simplify notation, we write $A\lessapprox B$ if $A\lesssim (\log 1/\delta)^{O(1)}B$, and we write $A\lessapprox_\eps B$ if $A\lesssim_\eps (\log 1/\delta)^{O_\eps(1)}B$ 

In the result that follows, recall that if $f \colon[0,1]\to\RR$ and $\delta>0$, then $f^\delta$ is the vertical $\delta$ neighborhood of the graph of $f$. We will often consider subsets $Y(f)\subset f^\delta$; we will call such sets \emph{shadings} of $f^\delta$. A subset of $Y(f)$ will sometimes be called a \emph{sub-shading}.

\begin{prop}\label{rectangleDecompProp}
Let $k,N\geq 2$ and let $0<\eps,\eta\leq 1$. Then there exists $\delta_0>0$ such that the following is true for all $\delta\in (0,\delta_0]$. Let $F$ be a set of functions, each of which has $C^k$ norm at most $1$. Suppose furthermore that $\#F \leq \delta^{-N}$. For each $f\in F$, let $Y(f)\subset f^\delta$ be a shading of $f$.

Suppose that the functions in $F$ are broad with respect to $(\delta,k)$ tangency rectangles, in the following sense: for every $x\in[0,1]\times[-1,1]$, every $\rho\geq\delta$, every $T\in [1, 1/\rho],$ and every $(\rho; k; T)$ rectangle $R$ containing $x$, we have
\begin{equation}\label{inductionPropHypothesis}
\#\{f\in F\colon x\in Y(f),\ f\sim R\}\leq\delta^{-\eta}T^{-\eps}\#\{f\in F\colon x\in Y(f)\}. 
\end{equation}

Then there exists the following:
\begin{itemize}
	\item A sub-shading $Y'(f)\subset Y(f)$ for each $f\in F$.
	\item A scale $\delta'\in [\delta,1]$.
	\item A set $\mathcal{R}$ of pairwise incomparable $(\delta'; k)$ rectangles.
	\item A number $\mu$, and for each $R\in\mathcal{R}$, a set $F(R)\subset \{f\in F\colon f\sim R\}$ of size $\mu$.
\end{itemize}
These objects have the following properties:
\begin{itemize}
\item[(A)] \itemizeEqnVSpacing
\begin{equation}\label{LkBdControlledBySumOverRect}
\int_{[0,1]\times[-1,1]}\Big(\sum_{f\in F}\chi_{Y(f)}\Big)^{\frac{k+1}{k}}\leq \delta^{-O(\sqrt\eta)}\sum_{R\in\mathcal{R}} \int_{R}\Big(\sum_{f\in F(R)}\chi_{Y'(f)}\Big)^{\frac{k+1}{k}}.
\end{equation}

\item[(B)]
Either $\#\mathcal{R}=1$, or for every $R\in\mathcal{R}$, every $\rho\in[\delta',1]$, every $T\in [1,1/\rho]$, and every $(\rho; k; T)$ rectangle $R'\supset R$, we have
\begin{equation}\label{conclusionB}
\#\{f\in F(R)\colon f\sim R'\} \leq (\delta')^{-2\sqrt\eta}T^{-\eps} \mu.
\end{equation}
\item[(C)]
For every $R\in\mathcal{R}$, let $\tilde F(R)=\{f_R\colon f\in F(R)\}$, and for every $\tilde f\in \tilde F(R)$, let $\tilde Y'(\tilde f)=\phi^R(Y'(f)\cap R)$ (recall Definition \ref{pullbackRect}). Then for every point $x\in[0,1]\times[-1,1]$, every $\rho\geq\delta/\delta'$, every $T\in[1,1/\rho],$ and every $(\rho; k-1; T)$ rectangle $R^\dag$ containing $x$, we have
\begin{equation}\label{inductionPropHypothesisAfterIteration}
\#\{\tilde f\in \tilde F(R)\colon x\in \tilde Y'(\tilde f),\  \tilde f\sim R^\dag\}\lessapprox_\eps  T^{-\eta/2} \#\{\tilde f\in \tilde F(R)\colon x\in \tilde Y'(\tilde f)\}. 
\end{equation}
\end{itemize}
\end{prop}

We will briefly explain the main hypotheses and conclusions of Proposition \ref{rectangleDecompProp}. The hypothesis that $\#F \leq \delta^{-N}$ is a minor technical convenience, since it allows us to bound $(\log \#F) \lesssim \log(1/\delta)$. Without this hypothesis the conclusions of Proposition \ref{rectangleDecompProp} would have (harmless) additional terms of the form $\log \#F$. 

The hypothesis \eqref{inductionPropHypothesis} says that for each point $x=(x_1,x_2)\in[0,1]\times[-1,1]$, the graphs of the space curves $\{\mathcal{J}_{k-1}(f)\}$ with $x\in Y(f)$ are broad, in the sense that not too many of these curves can all intersect tangentially at a common point. 

Conclusion (A) says that the $L^{\frac{k+1}{k}}$ norm of $\sum \chi_{Y(f)}$ can be broken into pieces localized to the rectangles in $\mathcal{R}$; c.f.~\eqref{LpNormBreaksOverRectangles}. Conclusion (B) says that the rectangles in $\mathcal{R}$ are $\mu$-rich and $\eps$-robustly broad with error at most $(\delta')^{-2\eta}$. Conclusion (C) says that for each $R\in\mathcal{R}$, the functions in $F(R)$ satisfy a (rescaled) version of the hypotheses \eqref{inductionPropHypothesis} from Proposition \ref{rectangleDecompProp}, except that $\eps$ has been replaced by $\eta/2$, and $\delta^{-\eta}$ has been replaced by $O_\eps(\log 1/\delta)^{O_\eps(1)}$.

\begin{proof}$\phantom{1}$\\
\noindent {\bf Step 1: A two-ends reduction.} Let $1\leq A\lesssim 1$ be a constant to be specified below (the reason for introducing the constant $A$ will be explained at the beginning of Step 2) and let $B=O(1)$ be a constant chosen after $A$. For each $x\in[0,1]\times[-1,1]$, let $t=t(x)$ be the infimum of all numbers $\delta\leq w\leq A^{-1}$ with the property that there exists a $(w;k; A)$ tangency rectangle $R\subset[0,1]\times[-1,1]$ containing $x$ with
\begin{equation}\label{numberOfFSimR}
\#\{f\in F\colon x\in Y(f), f\sim R\} \geq B^{-1} w^{\eta/2}\#\{f\in F\colon x\in Y(f)\}.
\end{equation}
We claim that if $B=O(1)$ is chosen sufficiently large, then the set of numbers $w$ satisfying \eqref{numberOfFSimR} is non-empty, since it contains $w=A^{-1}$. We justify this claim as follows. If $B=O(1)$ is selected sufficiently large (here $B$ is chosen after $A$), then we can find a set of $(A^{-1}; k; A)$ tangency rectangles of cardinality $B^{1/2}$ with the property that each of these rectangles is contained in $[0,1]\times[-1,1]$, and every curve in $F$ is tangent to at least one of these rectangles. By pigeonholing, we can select an $(A^{-1}; k; A)$ rectangle $R$ so that the LHS of \eqref{numberOfFSimR} has size at least $B^{-1/2}\#\{f\in F\colon x\in Y(f)\}$. Such a choice of $R$ satisfies \eqref{numberOfFSimR}, provided we select $B\geq A$.

Since the set of numbers $w$ satisfying\eqref{numberOfFSimR} is non-empty and bounded below by $\delta$, the infimum $t(x)$ described above exists. Since the graph of each curve is compact (recall that these graphs are contained in $[0,1]\times[-1,1]$) and tangency rectangles are closed (recall Definition \ref{defnTangencyRectangle}), for each $x\in[0,1]\times[-1,1]$ there exists a $(t(x); k; A)$ rectangle $R$ containing $x$ that satisfies \eqref{numberOfFSimR}. Denote this rectangle by $R(x)$. If more than one choice of $R(x)$ exists, then choose a rectangle that maximizes the LHS of \eqref{numberOfFSimR}.

For each $f\in F$, define the shading
\[
Y_1(f) = \{x\in Y(f)\colon f\sim R(x)\}.
\]
Then for $x\in[0,1]\times [-1,1]$ we have the pointwise bound
\begin{equation}\label{LInftyNormNotMuchSmaller}
\sum_{f\in F}\chi_{Y_1(f)}(x) \geq B^{-1} t(x)^{\eta/2} \sum_{f\in F}\chi_{Y(f)}(x) \gtrsim \delta^{\eta/2} \sum_{f\in F}\chi_{Y(f)}(x).
\end{equation}
and for all $f\in F$ and $x\in Y_1(f)$, we have $f\sim R(x)$. 

For each $x\in [0,1]\times [-1,1]$, each $\rho\in [\delta,t(x)]$, and each $(\rho;k;A)$ rectangle $R$ containing $x$, we have
\begin{equation}\label{twoEndsForRectangles}
\#\{f\in F\colon x\in Y_1(f),\ f\sim R\} \leq \big(\frac{\rho}{t(x)}\big)^{\eta/2}\#\{f\in F\colon x\in Y_1(f)\}.
\end{equation}
Inequality \eqref{twoEndsForRectangles} is true since we selected $R(x)$ to be a $(t(x); k; A)$ rectangle that maximizes the LHS of \eqref{numberOfFSimR}. This means that
\begin{equation*}
\begin{split}
\#\{f\in F &\colon x\in Y_1(f),\ f\sim R\} \\
& \leq \#\{f\in F\colon x\in Y(f),\ f\sim R\}\\
&< B^{-1} \rho^{\eta/2}\#\{f\in F\colon x\in Y(f)\}\\
&\leq B^{-1} \rho^{\eta/2} \big(B^{-1}t(x)^{\eta/2}\big)^{-1} \#\{f\in F\colon x\in Y(f),\ f\sim R(x)\}\\
& \leq \big(\frac{\rho}{t(x)}\big)^{\eta/2}\#\{f\in F\colon x\in Y_1(f)\}.
\end{split}
\end{equation*}

We finish this step with a final round of pigeonholing. For each $x\in[0,1]\times [-1,1]$ we have a number $\delta\leq t(x)\leq A^{-1}$ and a non-negative integer $\#\{f\in F\colon x\in Y_1(f)\}\leq \#F \leq \delta^{-N}$. Thus by dyadic pigeonholing, we can select a number $t\in[\delta,1]$ and an integer $\nu\geq 1$ so that if we define
\begin{equation}\label{defnOfSetX}
X = \big\{x\in[0,1]\times [-1,1]\colon t\leq t(x)<2t,\ \nu\leq \#\{f\in F\colon x\in Y_1(f)\}< 2\nu\big\},
\end{equation}
then
\begin{equation}\label{nuXapproxSize}
\nu^{\frac{k+1}{k}} |X|\approx\int_{x\in[0,1]\times[-1,1]}\Big(\sum_{f\in F}\chi_{Y_1(f)}(x)\Big)^{\frac{k+1}{k}},
\end{equation}
where the implicit constant depends on $N$.

For each $f\in F$, define $Y_2(f) = Y_1(f)\cap X$. Then
\begin{equation}\label{chain1}
\begin{split}
\int_{[0,1]\times[-1,1]}\Big(\sum_{f\in F}\chi_{Y(f)}\Big)^{\frac{k+1}{k}} & \lesssim \delta^{-\frac{\eta}{2}(\frac{k+1}{k})} \int_{[0,1]\times[-1,1]}\Big(\sum_{f\in F}\chi_{Y_1(f)}\Big)^{\frac{k+1}{k}}\\
& \lessapprox \delta^{-\frac{\eta}{2}(\frac{k+1}{k})}\int_{[0,1]\times[-1,1]}\Big(\sum_{f\in F}\chi_{Y_2(f)}\Big)^{\frac{k+1}{k}}.
\end{split}
\end{equation}
Furthermore, \eqref{LInftyNormNotMuchSmaller} and \eqref{twoEndsForRectangles} continue to hold with the shading $Y_2$ in place of $Y_1$.

Finally, by \eqref{inductionPropHypothesis} and \eqref{LInftyNormNotMuchSmaller}, for every $x\in \bigcup_{f\in F} Y_2(f)$, every $\rho\geq\delta$, every $T\in [1, 1/\rho]$, and every $(\rho; k; T)$ rectangle $R$ containing $x$, we have
\begin{equation}\label{notTooManyKp1RectanglesThruX}
\begin{split}
\#\{f\in F\colon x\in Y_2(f), f\sim R\} & \leq \#\{f\in F\colon \ x\in Y(f),\ f\sim R\}\\
&\leq \delta^{-\eta}T^{-\eps}\#\{f\in F\colon x\in Y(f)\}\\
&\leq 2\delta^{-\frac{3}{2}\eta}T^{-\eps}\nu.
\end{split}
\end{equation}
This bound will help us establish Conclusion (B).

\medskip

\noindent {\bf Step 2: Clustering into rectangles.}
Our goal in this step is to find a set of rectangles $\mathcal{R}_0$ so that Item (A) is satisfied. The idea is as follows: We choose $\delta'\sim t$ and select a (maximal) set of pairwise incomparable $(\delta';k)$ rectangles $\mathcal{R}_0$. For each point $x\in\RR^2$, the rectangle $R(x)$ from Step 1 will be comparable to some rectangle $R\in \mathcal{R}_0$. If $f\in F$ and $x\in Y_2(f)$, then $f\sim R(x)$ and thus (one might hope!) $f\sim R$. Thus we would have the pointwise inequality
\begin{equation}\label{onlyComparableRectanglesContribute}
\Big(\sum_{f\in F}\chi_{Y_2(f)}(x)\Big)^{\frac{k+1}{k}}\lesssim \Big(\sum_{\substack{f\in F\\ f\sim R}}\chi_{Y_2(f)}(x)\Big)^{\frac{k+1}{k}},
\end{equation}
and \eqref{LkBdControlledBySumOverRect} would follow. The only problem with the above argument is that if $f\sim R(x)$ and $R(x)\sim R$, then it is almost true that $f\sim R$, but not quite---we only have that $f$ is tangent to a slight thickening of $R$. It is to deal with this technical annoyance that we introduced the number $A = O(1)$ in Step 1.

Let $\delta'=At$, i.e.~a $(\delta';k)$ rectangle can be thought of as a $(t;k;A)$ rectangle that has been thickened by a (multiplicative) factor of $A$ in the vertical direction. If the constant $A = O(1)$ is selected appropriately, then we can find a set $\mathcal{R}_0$ of $(\delta';k)$ rectangles with the following properties: 
\begin{itemize}
	\item[(i)] For each $R\in\mathcal{R}_0$, at most $O(1)$ rectangles from $\mathcal{R}_0$ are comparable to $R$. 
	\item[(ii)] Let $x\in X$ (this is the set from \eqref{defnOfSetX}), and let $R(x)$ be a $(t;k;A)$ rectangle from Step 1. Write $R(x) = R^g(I)$, and let $R'(x)\supset R(x)$ be the $(\delta';k)$ rectangle obtained by taking the vertical $\delta'$ neighborhood of the graph of $g$ above $I$ (i.e.~$R'(x)$ is the rectangle obtained by thickening $R(x)$ in the vertical direction). Let $f\in F$ with $x\in Y_2(f)$, and hence $f\sim R(x)$. Then there exists a rectangle $R_0\in \mathcal{R}_0$ that is comparable to $R'(x)$, and satisfies $f\sim R_0$. 
\end{itemize}

For each $x\in X$, define $R_0(x)\in\mathcal{R}_0$ to be a rectangle that maximizes the quantity
\begin{equation}\label{maximizeR0}
\#\{f\in F\colon x\in Y_2(f),\ f\sim R_0\}.
\end{equation}
If more than one choice of $R_0(x)$ exists, choose one arbitrarily. At the moment, we do not know that $x\in R_0(x)$; we only know that $x\in Y_2(f)$ and $f\sim R_0$ ensures that $x$ is contained in the vertical $\delta$ neighbourhood of $R_0(x)$. We fix this (with slight abuse of notation) by replacing each rectangle in $\mathcal{R}_0$ with its vertical $\delta$ neighbourhood.

By the discussion above, there always exists a rectangle so that $\eqref{maximizeR0}\sim \nu$. For each $R\in\mathcal{R}_0$, define
\[
Y(R) = \{x\in R\colon R_0(x)=R\}.
\]
By construction, we have 
\begin{equation}\label{XPartition}
X = \bigsqcup_{R\in\mathcal{R}_0}Y(R),
\end{equation}
and we have the pointwise quasi-equality
\begin{equation}\label{quasiEqualityNu}
\sum_{R\in\mathcal{R}_0}\sum_{\substack{f\in F \\ f\sim R}}\chi_{Y(R)\cap Y_2(f)}(x) \sim \nu,\quad x\in X.
\end{equation}

By dyadic pigeonholing, for each $R\in\mathcal{R}_0$ we can choose a number $\lambda_R>0$ so that 
\[
\sum_{\substack{f\in F,\ f \sim R \\ \lambda_R \leq |Y(R)\cap Y_2(f)|<2\lambda_R}} |Y(R)\cap Y_2(f)| \gtrsim (\log 1/\delta)^{-1} 	\sum_{\substack{f\in F,\ f \sim R}}|Y(R)\cap Y_2(f)|.
\]

After further dyadic pigeonholing, we can select numbers $\lambda,$ $\mu$, so that if we define $\mathcal{R}_1\subset\mathcal{R}_0$ to be the set of those rectangles $R\in\mathcal{R}_0$ for which $\lambda\leq \lambda_R<2\lambda$ and
\begin{equation}\label{numberOfFSimMu}
 \#\{f\in F\colon f\sim R,\ \lambda_R \leq |Y(R)\cap Y_2(f)|<2\lambda_R\} \in [\mu,2\mu),
\end{equation}
then
\begin{equation}\label{sumOverR1PreservesMass}
\begin{split}
(\#\mathcal{R}_1)(\lambda \mu) & \sim \sum_{R\in\mathcal{R}_1} \sum_{\substack{f\in F,\ f \sim R \\ \lambda_R \leq |Y(R)\cap Y_2(f)|<2\lambda_R}} |Y(R)\cap Y_2(f)| \\
&\gtrapprox \sum_{R\in\mathcal{R}_0} \sum_{\substack{f\in F,\ f \sim R \\ \lambda_R \leq |Y(R)\cap Y_2(f)|<2\lambda_R}} |Y(R)\cap Y_2(f)|\\
&\gtrapprox \nu|X|.
\end{split}
\end{equation}
After a final round of refinement (this time by a factor of $O(1)$), we can select a set $\mathcal{R}_2\subset\mathcal{R}_1$ so that \eqref{sumOverR1PreservesMass} remains true with $\mathcal{R}_2$ in place of $\mathcal{R}_1$ on the LHS, and furthermore for each $f\in F$ the sets $\{R\cap \operatorname{graph}(f)\colon R\in\mathcal{R}_2,\ f\sim R\}$ are disjoint. 

Define 
\[
\mathcal{R}_3 = \{R\in\mathcal{R}_2\colon |Y(R)|\leq 2 |X|(\#\mathcal{R}_2)^{-1}\}.
\] 
Since the sets $\{Y(R)\colon R\in\mathcal{R}_2\}$ are disjoint subsets of $X$, we have $2|X|(\#\mathcal{R}_2)^{-1} (\#(\mathcal{R}_2\backslash\mathcal{R}_3)) \leq |X|$, and thus $\#\mathcal{R}_3\geq\frac{1}{2}\#\mathcal{R}_2$. In particular, \eqref{sumOverR1PreservesMass} remains true with $\mathcal{R}_3$ in place of $\mathcal{R}_1$ on the LHS. 

For each $f\in F$, define
\[
Y_3(f) = \bigcup_{\substack{R\in \mathcal{R}_3, f\sim R \\ \lambda_R \leq |Y(R)\cap Y_2(f)|< 2\lambda_R }} Y(R)\cap Y_2(f).
\]
Thus for each $R\in\mathcal{R}_3$, we have
\begin{equation}\label{suppBd}
\Big|\operatorname{supp}\Big(\sum_{\substack{f\in F \\ f\sim R}}\chi_{R\cap Y_3(f)}\Big)\Big| \leq |Y(R)|\leq 2 |X|(\#\mathcal{R}_2)^{-1} \lessapprox \frac{\lambda\mu}{\nu},
\end{equation}
where the final inequality used \eqref{sumOverR1PreservesMass}.

Note that for each $R\in\mathcal{R}_3$, \eqref{numberOfFSimMu} remains true with $Y_3$ in place of $Y_2$. Thus for each $R\in\mathcal{R}_3$ we can select a set $F(R)\subset F$ of size $\mu$, of the form 
\begin{equation}\label{defnFR}
F(R) \subset \{f\in F\colon f\sim R,\ \lambda_R \leq |Y_3(f)\cap R| < 2\lambda_R\}.
\end{equation}

\medskip

\noindent We have
\begin{equation}\label{Y3VsY2}
\sum_{R\in\mathcal{R}_3} \sum_{\substack{f\in F(R)}}|R\cap Y_3(f)| \geq \frac{1}{4}  \sum_{f\in F}|Y_3(f)| \gtrapprox \sum_{f\in F}|Y_2(f)| \sim 
\sum_{R\in\mathcal{R}_0}\sum_{\substack{f\in F \\ f\sim R}}|Y(R)\cap Y_2(f)|.
\end{equation}
The first inequality is \eqref{numberOfFSimMu} (i.e. $\#F(R)\geq \frac{1}{2}\#\{f\in F\colon f\sim R,\ Y_3(f)\cap R\neq\emptyset\}$, and as $R$ ranges over the elements of $\mathcal{R}_3$,  each of these sets contributes about the same amount to the first sum), while the second follows from the fact that the sequence of pigeonholing refinements that we used to obtain the shading $Y_3(f)\subset Y_2(f)$ only lost a $\sim \big(\log(1/\delta)\big)^{-3}$ fraction of the overall mass.

We now compute
\begin{equation}\label{Y3VsOriginalY}
\begin{split}
\sum_{R\in\mathcal{R}_3}& \int_{R}\Big(\sum_{\substack{f\in F(R)}}\chi_{Y_3(f)}\Big)^{\frac{k+1}{k}}\\
& = \int_{[0,1]\times[-1,1]}\Big(\sum_{R\in\mathcal{R}_3}\sum_{\substack{f\in F(R)}}\chi_{Y_3(f)}\Big)^{\frac{k+1}{k}}\\
& \geq |X|^{\frac{-1}{k}}\Big(\int_{[0,1]\times[-1,1]}\sum_{R\in\mathcal{R}_3}\sum_{\substack{f\in F(R)}}\chi_{Y_3(f)}\Big)^{\frac{k+1}{k}}\\
& \gtrapprox |X|^{\frac{-1}{k}}\Big(\int_{[0,1]\times[-1,1]}\sum_{R\in\mathcal{R}_0}\sum_{\substack{f\in F \\ f\sim R}}\chi_{Y(R)\cap Y_2(f)}\Big)^{\frac{k+1}{k}}\\
& \gtrapprox |X|^{\frac{-1}{k}}(|X| \nu) ^{\frac{k+1}{k}} = \nu^{\frac{k+1}{k}}|X|\\
& \gtrapprox \delta^{\frac{\eta}{2}(\frac{k+1}{k})}\int_{[0,1]\times[-1,1]}\Big(\sum_{f\in F}\chi_{Y(f)}\Big)^{\frac{k+1}{k}}.
\end{split}
\end{equation}
The first equality follows from the fact that the sets $\{R\cap \bigcup_{f\in F(R)}Y_3(f):\ R\in\mathcal{R}_3\}$ are disjoint. The second inequality is H\"older (recall that each shading $Y_3(f)$ is contained in $X$), and the third inequality is \eqref{Y3VsY2}. The fourth inequality is \eqref{quasiEqualityNu}, and the final inequality is \eqref{nuXapproxSize} followed by \eqref{chain1}.

\medskip

We now divide into cases.\\
Case 1: If $\delta'\leq\delta^{\sqrt\eta}$ then define $\mathcal{R}=\mathcal{R}_3$. This is the main case.\\
Case 2: If $\delta'>\delta^{\sqrt\eta}$, then since $\#\mathcal{R}_3\leq \#\mathcal{R}_0\lesssim (\delta')^{-O(1)}\lesssim \delta^{-O(\sqrt\eta)}$, we can select $R_0\in\mathcal{R}_3$ with
\begin{equation}\label{chain3}
\sum_{R\in\mathcal{R}_3}\int_R \Big(\sum_{\substack{f\in F(R)}}\chi_{Y_3(f)} \Big)^{\frac{k+1}{k}}\leq \delta^{-O(\sqrt\eta)}\int_{R_0}\Big(\sum_{\substack{f\in F(R)}}\chi_{Y_3(f)} \Big)^{\frac{k+1}{k}}.
\end{equation}
Define $\mathcal{R}=\{R_0\}$.

We will show that in both Case 1 and Case 2, Item (B) is satisfied. In Case 2, we have $\#\mathcal{R}=1$ and Item (B) is satisfied. Suppose instead we are in Case 1, i.e.~$\delta'\in [\delta, \delta^{\sqrt\eta}]$; we will establish \eqref{conclusionB}. Let $\rho\in [\delta', 1]$, $T\in [1,1/\rho]$, let $R\in\mathcal{R}$, and let $R'\supset R$ be a $(\rho; k; T)$ rectangle. Suppose
\[
\#\{f\in F(R)\colon f\sim R'\} = \omega\mu,
\]
for some $\omega>0$. To show that Item (B) is satisfied, we need to prove that
\begin{equation}\label{neededBdOnOmega}
\omega \leq (\delta')^{-2\sqrt\eta}T^{-\eps}. 
\end{equation}
By the definition of $F(R)$ from \eqref{defnFR}, we have
\[
\int_R \sum_{\substack{f\in F(R)\\ f\sim R'}}\chi_{Y_3(f)} \geq \omega\lambda\mu.
\]
On the other hand, by \eqref{suppBd} the above function is supported on a set of size $\lessapprox \frac{\lambda\mu}{\nu}$. We conclude that there exists a point $x\in R$ with
\[
\#\{f\in F(R)\colon x \in Y_3(f),\ f\sim R'\} \gtrapprox \omega \nu,
\]
and thus
\begin{equation}\label{manyCurvesThruXInRPrime}
\#\{f\in F\colon x\in Y_1(f),\ f\sim R'\} \gtrapprox \omega \nu.
\end{equation}
Comparing \eqref{notTooManyKp1RectanglesThruX} and \eqref{manyCurvesThruXInRPrime}, we have 
\[
\omega \nu \lessapprox \delta^{-(3/2)\eta}T^{-\eps}\nu \lesssim (\delta^\prime)^{(3/2)\sqrt\eta}T^{-\eps}\nu,
\]
and thus we obtain \eqref{neededBdOnOmega}, provided $\delta_0=\delta_0(\eps,\eta)$ is selected sufficiently small---here we use the assumption that $\delta'\leq \delta^{\sqrt\eta}$ to dominate the implicit constant $(\log 1/\delta)^{O(1)}$ in \eqref{manyCurvesThruXInRPrime} by $(\delta')^{-\sqrt \eta/4}$. At this point, we have established Item (B).

\medskip

\noindent {\bf Step 3: Refining the shading.} Our next task is to establish Item (C). After dyadic pigeonholing, there exists a number $\nu_1\leq\nu$ so that if we define
\[
Y'(f) = \bigcup_{R\colon f\in F(R)} \Big\{x\in Y_3(f) \cap R \colon \nu_1 \leq \sum_{g\in F(R)}\chi_{Y_3(g)}(x)<2\nu_1\Big\},
\]
then
\begin{equation}\label{Y2VsYp}
\sum_{R\in\mathcal{R}}\int_R\Big(\sum_{f\in F(R)}\chi_{Y_3(f)}\Big)^{\frac{k+1}{k}}\lessapprox \sum_{R\in\mathcal{R}}\int_R\Big(\sum_{f\in F(R)}\chi_{Y'(f)}\Big)^{\frac{k+1}{k}}.
\end{equation}
We have $\nu_1\gtrapprox\nu$, and thus by \eqref{twoEndsForRectangles}, for each $x\in X$, each $\rho\in [\delta/\delta', 1]$, and each $(\rho\delta'; k)$ rectangle $R'$ containing $x$, we have
\begin{equation}\label{rectanglesTangentToPtInsideR}
\begin{split}
\#\{f\in F(R)\colon x\in &\phantom{.} Y'(f),\ f\sim R'\} \lesssim \rho^{\eta/2}\nu \\
& \lessapprox  \rho^{\eta/2}\nu_1\lessapprox  \rho^{\eta/2}\#\{f\in F(R)\colon x\in Y'(f)\}.
\end{split}
\end{equation}
Note that technically, we can only apply \eqref{twoEndsForRectangles} to establish \eqref{rectanglesTangentToPtInsideR} for $\rho\in[\delta/\delta', t(x)/\delta']$. However for $\rho\in(t(x)/\delta', 1]\subset [A^{-1},1]$, Inequality \eqref{rectanglesTangentToPtInsideR} is trivial, provided we select the implicit constant to be at least $A$. Inequality \eqref{rectanglesTangentToPtInsideR} also trivially holds for $x\not\in X$, since the LHS and RHS of \eqref{rectanglesTangentToPtInsideR} are both zero.

We will see how this establishes Item (C). Fix a rectangle $R=R(I)\in\mathcal{R}$, and let $\tilde F(R)$ and $\{ \tilde Y'(\tilde f) \colon \tilde f\in \tilde F(R)\}$ be the sets defined in Part (C) of the statement of Proposition \ref{rectangleDecompProp}. Under this rescaling, the inequality \eqref{rectanglesTangentToPtInsideR} becomes the following: For each $x\in [0,1]\times[-1,1]$, each $\rho\in [\delta/\delta', 1]$, and each $(\rho; k)$ rectangle $R'$ containing $x$, we have
\begin{equation}\label{rectanglesTangentToPtInsideRRescaled}
\#\{\tilde f\in \tilde F(R)\colon x\in \tilde Y'(f),\ \tilde f\sim R'\} \lessapprox \rho^{\eta/2}\#\{\tilde f\in \tilde F(R)\colon x\in \tilde Y'(\tilde f)\}.
\end{equation}

To show that Item (C) is satisfied, let $x\in [0,1]\times[-1,1]$, $\rho\in [\delta/\delta', 1]$, $T\in [1, 1/\rho]$, and let $R^\dag$ be a $(\rho; k-1; T)$ rectangle containing $x$. By Lemma \ref{longRectInsideKP1Rect}, the functions $\{\tilde f\in \tilde F(R)\colon x\in \tilde Y'(\tilde f),\ \tilde f\sim  R^\dag\}$ are all tangent to a curvilinear rectangle of dimensions $C\tau\times \tau^{1/k}$, where $C = O(1)$ and $\tau = \max(\rho, T^{-k})$. Thus by Lemma \ref{thickenedVsRegularRectangleLem}, at least an $\Omega(1)$ fraction of these rectangles are tangent to a common $(\tau;k)$ rectangle $R'$, which contains $x$, i.e.
\[
\#\{\tilde f\in \tilde F(R)\colon x\in \tilde Y'(f),\ \tilde f\sim  R^\dag\}\lesssim  \#\{\tilde f\in \tilde F(R)\colon x\in \tilde Y'(f),\ \tilde f\sim  R'\}.
\]
The size of the latter set is controlled by \eqref{rectanglesTangentToPtInsideRRescaled}. Thus we have

\begin{equation*} 
\begin{split}
\#\{\tilde f\in \tilde F(R)\colon x\in \tilde Y'(f),\ \tilde f\sim  R^\dag\}&  \lessapprox \tau^{\eta/2}\#\{\tilde f\in \tilde F(R)\colon x\in \tilde Y'(\tilde f)\}\\
& \lesssim \max(\rho^{\eta/2}, T^{-k\eta/2})\#\{\tilde f\in \tilde F(R)\colon x\in \tilde Y'(\tilde f)\}\\
& \leq T^{-\eta/2}\#\{\tilde f\in \tilde F(R)\colon x\in \tilde Y'(\tilde f)\},
\end{split}
\end{equation*}
where the second inequality used the fact that $T\leq \rho^{-1}$ and $k\geq 1$. This establishes Item (C).

Finally, by chaining inequalities \eqref{Y3VsOriginalY} (plus \eqref{chain3} if we are in Case 2 from Step 2)  and \eqref{Y2VsYp}, we verify that Item (A) is satisfied.
\end{proof}


Next, we will show how Proposition \ref{rectangleDecompProp} and Theorem \ref{RectBoundThm} can be combined to prove Proposition \ref{maximalFnBdNonconcentrated}, which is a variant of Theorem \ref{mainThm} where the non-concentration condition on $F$ is replaced by a (local) two-ends type non-concentration condition on the set of curves passing through each point. Before stating the result, we recall the following definition from \cite{KatzTao}.
\begin{defn}\label{defnKTSet}
Let $(M,d)$ be a metric space. Let $\alpha,\delta,C>0$. A set $E\subset M$ is called a $(\delta,\alpha;C)$-\emph{set} if for all $r\geq\delta$ and all metric balls $B$ of radius $r$, we have
\[
\mathcal{E}_{\delta}(E \cap B)\leq C(r/\delta)^{\alpha},
\] 
where $\mathcal{E}_{\delta}(X)$ denotes the $\delta$-covering number of the set $X$. In informal settings, we will sometimes abbreviate this to $(\delta,\alpha)$-set.
\end{defn}

Our proof below involves an anisotropic rescaling, which sends a $(\rho;k)$ rectangle to the unit square. Such a rescaling distorts $(\delta,\alpha)$-sets into slightly more complicated objects. The next definition describes a class of sets that is preserved (as a class) under this type of rescaling. 

\begin{defn}\label{defnOfStriped}
Let $\delta,\tau,\alpha>0$ and $C\geq 1$. Let $f\colon[0,1]\to\RR$. We say a set $Y(f)\subset f^\delta$ is a \emph{$\delta$--thick shading striped by a $(\tau,\alpha;C)$-set} if $Y(f)$ is contained in a set of the form $f^\delta\cap (E\times\RR)$, where $E\subset[0,1]$ is a $(\tau, \alpha;C)$-set.
\end{defn}

Note that in the special case $\tau=\delta$ and $f$ Lipschitz, a $\delta$--thick shadings striped by a $(\delta,\alpha;C)$-set is simply a $(\delta,\alpha; 2C)$-set, i.e.~Definition \ref{defnOfStriped} offers nothing new beyond what is already given in Definition \ref{defnKTSet} in the special case $\tau=\delta$. Our main task is to prove Theorem \ref{mainThm}$^\prime$ (stated in Section \ref{pfOfMainThmSec}). The hypotheses of Theorem \ref{mainThm}$^\prime$ require that each shading $Y(f)$ be a $(\delta,\alpha; C)$-set, i.e.~the statement of Theorem \ref{mainThm}$^\prime$ only uses the special case of Definition \ref{defnOfStriped} where $\tau=\delta$. The proof of Theorem \ref{mainThm}$^\prime$, however, will require the additional flexibility afforded by allowing $\tau$ and $\delta$ to differ. 

We exploit this flexibility as follows. Suppose $Y(f)\subset f^\delta$ is a $\delta$--thick shading striped by a $(\tau,\alpha;C)$-set. Suppose furthermore that $R=R(I)$ is a $(\delta';k)$ rectangle, and $f\sim R$. Let $\tilde f=f_R$ in the sense of Definition \ref{pullbackRect}, and let $\tilde Y(\tilde f)=\phi^R(Y(f)\cap R)$. Since this rescaling is anisotropic, it affects $\delta$ and $\tau$ differently---$\tilde Y(\tilde f)$ is a $\delta/\delta'$--thick shading striped by a $(\tau/\delta'^{1/k}, \alpha; C)$-set. This observation will play an important role in the proof below. 

\begin{prop}\label{maximalFnBdNonconcentrated}
Let $k,N\geq 1,$ $0\leq \alpha\leq 1$, and let $\eps>0$. Then there exists $\eta>0$ such that the following is true for all $\delta,\tau>0$. Let $F$ be a set of polynomials of degree at most $\delta^{-\eta}$, each of which has $C^{k}$ norm at most 1. Suppose that $\#F\leq\delta^{-N}$. For each $f\in F$, let $Y(f)$ be a $\delta$--thick shading striped by a $(\tau,\alpha;\delta^{-\eta})$-set. Suppose that for all $x\in [0,1]^2$, all $\rho\in[\delta,1]$, all $T\in [1, 1/\rho]$, and all $(\rho; k; T)$ rectangles $R$ containing $x$, we have
\begin{equation}\label{rectangleConcentrationBdNonconcentratedProp}
\#\{f\in F\colon x\in Y(f),\ f\sim R\} \leq \delta^{-\eta}T^{-\eps}\#\{f\in F\colon x\in Y(f)\}.
\end{equation}

Then
\begin{equation}\label{NonconcentratedPropLpBound}
\Big \Vert \sum_{f\in F}\chi_{Y(f)}\Big\Vert_{\frac{k+1}{k}}\lesssim_\eps  \delta^{-\eps}\delta^{\frac{k+\alpha}{k+1}}\tau^{\frac{k(1-\alpha)}{k+1}}
(\#F).
\end{equation}
\end{prop}

\begin{proof}
We prove the result by induction on $k$. 

\medskip

\noindent{\textbf{The base case}.}\ \ We begin with the base case $k=1$. By dyadic pigeonholing we can find a number $\mu$ and a set $X\subset[0,1]^2$ so that
\[
\Big \Vert \sum_{f\in F}\chi_{Y(f)}\Big\Vert_2^2\lesssim (\log 1/\delta) \mu^2|X|, \quad\textrm{and}\quad \mu\leq \sum_{f\in F}\chi_{Y(f)}(x)< 2\mu\quad\textrm{for all}\ x\in X.
\]
Thus in order to establish \eqref{NonconcentratedPropLpBound}, it suffices to show that
\begin{equation}\label{boundOnSizeOfX}
|X| \lesssim_\eps \delta^{-\eps/2}\delta^{1+\alpha}\tau^{1-\alpha}\mu^{-2}(\#F)^2.
\end{equation}

By \eqref{rectangleConcentrationBdNonconcentratedProp} with $\rho=2\delta$ and $T=2^{1/\eps}\delta^{-\eta/\eps}$, we have that for each $x\in X$, there are $\gtrsim\mu^2$ pairs $f,g\in F$ with the following two properties:
\begin{itemize}
\item[(i)] $x\in Y(f)\cap Y(g)$.
\item[(ii)] The connected component of $f^\delta\cap g^\delta$ containing $x$ projects to an interval of length at most $2^{1/\eps}\delta^{1-\eta/\eps}$ on the $x_1$-axis.
\end{itemize} 
For each pair $f,g\in F$, let 
\begin{equation*}
\begin{split}
W(f,g) = \{x\in & f^\delta\cap g^\delta\colon  \textrm{the connected component of}\ f^\delta\cap g^\delta\ \textrm{containing}\ x \\
& \textrm{projects to an interval of length}\ \leq 2^{1/\eps}\delta^{1-\eta/\eps}\ \textrm{on the $x_1$-axis}\}.
\end{split}
\end{equation*}
Since $f$ and $g$ are polynomials of degree at most $\delta^{-\eta}$, $f^\delta\cap g^\delta$ is a union of $O(\delta^{-\eta})$ connected components, and thus the projection of $W(f,g)$ to the $x_1$-axis is a union of $O(\delta^{-\eta})$ intervals, each of length at most $2^{1/\eps}\delta^{1-\eta/\eps-\eta}\lesssim_\eps \delta^{1-2\eta/\eps}$. Denote this set by $I(f,g)\subset[0,1]$.

Since $Y(f)$ is a $\delta$--thick shading striped by a $(\tau,\alpha;\delta^{-\eta})$-set, we have
\begin{equation}\label{Y1FcapIfg}
\begin{split}
|Y(f) \cap (I(f,g)\times\RR)|& \leq \sum_{I\subset I(f,g)} |Y(f) \cap (I\times\RR)|\\
&\lesssim_\eps \delta^{-\eta} \left\{ \begin{array}{ll} 
(\delta \tau) \big(\delta^{-\eta} (\delta^{1-2\eta/\eps}/\tau)^\alpha\big),& \delta^{1-2\eta/\eps}\geq\tau\\
\delta^{2-2\eta/\eps},& \delta^{1-2\eta/\eps}\leq\tau
\end{array}
\right.\\
&\leq \delta^{1+\alpha-3\eta/\eps}\tau^{1-\alpha},
\end{split}
\end{equation}
where in the first inequality, the sum is taken over the connected components (i.e.~intervals) in $I(f,g)$. This sum has $O(\delta^{-\eta})$ terms.

Let $\mathcal{T}$ be the set of triples $(x, f,g)\in X\times F^2$, where the pair $f,g$ satisfies Items (i) and (ii). Then $|\mathcal{T}|\sim \mu^2 X$, where $|\cdot|$ denotes the product of two-dimensional Lebesgue measure on $X$ and counting measure on $F^2$.

For each $f,g\in F$, we have
\[
|\{x\in X\colon (x, f,g)\in \mathcal{T}\}| \leq  |Y(f) \cap (I(f,g)\times\RR)|\leq \delta^{1+\alpha-3\eta/\eps}\tau^{1-\alpha},
\]
where the final inequality used \eqref{Y1FcapIfg}. We conclude that 
\[
|\mathcal{T}|\lesssim_\eps  \delta^{1+\alpha-3\eta/\eps}\tau^{1-\alpha}(\#F)^2.
\]
On the other hand, we have $|X|\leq \mu^{-2}|\mathcal{T}|$, and thus
\begin{equation}\label{boundOnX}
|X|\lesssim_\eps  \delta^{1+\alpha-3\eta/\eps}\tau^{1-\alpha}\Big(\frac{\#F}{\mu}\Big)^2.
\end{equation}
If we select $\eta\leq \eps^2/6$, then \eqref{boundOnX} implies \eqref{boundOnSizeOfX}. This completes our proof of the base case.

\medskip

\noindent \textbf{The induction step.}\ 
Suppose that $k\geq 2$ and that the result has been established for $k-1$. Fix $0\leq\alpha\leq 1$ and $\eps>0$. Let $\eta>0$ be a quantity to be specified below, let $\delta,\tau>0$, and let $F$ and $Y(f)$ satisfy the hypotheses of Proposition \ref{maximalFnBdNonconcentrated} with this value of $\eta$. First, let $\delta_0>0$ be a small quantity to be chosen below, which depends on $k$ and $\eps$. We may suppose that $\delta\leq\delta_0$, since otherwise \eqref{NonconcentratedPropLpBound} is trivial, provided we choose the implicit constant sufficiently large.  

Let $\eta_1=\eta_1(k,\eps)$ be the output from Theorem \ref{RectBoundThm}, with $\eps/2$ in place of $\eps$. Let $\eta_2 = \eta_1^2/4$. We will select $\eta>0$ sufficiently small so that $\eta\leq\eta_2$. Thus the shadings $\{Y(f)\colon f\in F\}$ satisfy Hypothesis \eqref{inductionPropHypothesis} of Proposition \ref{rectangleDecompProp} with $\eps/2$ in place of $\eps$ (this is because decreasing $\eps$ can only make \eqref{inductionPropHypothesis} easier to satisfy) and $\eta_2$ in place of $\eta$. Applying Proposition \ref{rectangleDecompProp}, we get a sub-shading $Y'(f)\subset Y(f)$; a scale $\delta'\in [\delta,1]$; a set $\mathcal{R}$ of $(\delta', k)$ rectangles; sets $F(R),$ $R\in\mathcal{R}$; and a multiplicity $\mu\leq\#F$.

By Proposition \ref{rectangleDecompProp} Item (B), either $\#\mathcal{R}=1$, or we can apply Theorem \ref{RectBoundThm} (recall that we selected $\eta_1$ sufficiently small to ensure that Theorem \ref{RectBoundThm} can be applied, and $\eta_2 = \eta_1^2/4$; c.f.~\eqref{conclusionB}) to conclude that
\begin{equation}\label{rectangleBdInPropMaximalFnBdNonconcentrated}
\# \mathcal{R} \leq \delta^{-\eps/2}  \Big(\frac{\# F}{\mu}\Big)^{\frac{k+1}{k}}.
\end{equation}
Technically, we can only apply Theorem \ref{RectBoundThm} if $\delta$ is sufficiently small in terms of $k$ and $\eps$. However if this is not the case then \eqref{rectangleBdInPropMaximalFnBdNonconcentrated} follows from the crude bound $\# \mathcal{R}\lesssim (\delta')^{-k-1}\lesssim_{\eps}1$.

We next explore the consequences of Item (C) from Proposition \ref{rectangleDecompProp}. We first consider the case where $\delta'\leq\delta^{1-\eps/2}$. By Item (A) from Proposition \ref{rectangleDecompProp}, we have 
\begin{equation}
\begin{split}
\Big \Vert \sum_{f\in F}\chi_{Y(f)}\Big\Vert_{\frac{k+1}{k}}^{\frac{k+1}{k}} 
& \leq \delta^{-O(\sqrt \eta_2)}\sum_{R\in\mathcal{R}} \int_{R}\Big(\sum_{f\in F(R)}\chi_{Y'(f)}\Big)^{\frac{k+1}{k}}\\
& \lesssim \left\{ \begin{array}{ll} \delta^{-O(\eta_1)}(\#\mathcal{R}) \mu^{\frac{k+1}{k}} (\delta\tau)\Big( \delta^{-\eta}(\delta^{1/k}/\tau)^{\alpha}\Big), & \tau\leq\delta^{1/k} \\ \delta^{-O(\eta_1)}(\#\mathcal{R}) \mu^{\frac{k+1}{k}} \delta^{\frac{k+1}{k}}, & \tau>\delta^{1/k} \end{array}\right. \\
& \lesssim\delta^{-\eps/2 -O(\eta_1)}\delta^{\frac{k+\alpha}{k}}\tau^{1-\alpha}  (\#F)^{\frac{k+1}{k}},
\end{split}
\end{equation}
and we have established \eqref{NonconcentratedPropLpBound} and completed the proof, provided we select $\eta_1$ sufficiently small depending on $\eps$ and $k$, and provided we select the implicit constant in \eqref{NonconcentratedPropLpBound} sufficiently large depending on $\eps$ and $k$.

For the remainder of the proof we consider the case where $\delta'>\delta^{1-\eps/2},$ so in particular the implicit constant $O_{\eps}(\log 1/\delta)^{O_\eps(1)}$ from \eqref{inductionPropHypothesisAfterIteration} is bounded by $O_{\eps}(\log \frac{1}{\delta/\delta'})^{O_\eps(1)}$. For each $R\in\mathcal{R}$, let $\tilde F(R)$ and $\tilde Y'(f)$ be as defined in Item (C) of Proposition \ref{rectangleDecompProp}; these sets satisfy \eqref{inductionPropHypothesisAfterIteration}.

If $\delta_0$ (and thus $\delta/\delta'$) is sufficiently small, then $\tilde F(R)$ and $\tilde Y'(f)$ will satisfy the induction hypothesis \eqref{rectangleConcentrationBdNonconcentratedProp} with the parameters changed as follows:
\begin{itemize}
\item $k$ is replaced by $k-1$.
\item $\delta$ is replaced by $\tilde\delta = \delta/\delta'$.
\item $\tau$ is replaced by $\tilde\tau = \tau/(\delta')^{1/k}$
\item The functions $f\in F$ are polynomials of degree $\delta^{-\eta} \leq \tilde\delta^{-2\eta/\eps}$, each of which have $C^k$ norm at most 1 (the latter property follows from Lemma \ref{CkNormRectangleRescaling}).
\item The shadings $\tilde Y'(f)$ are $\tilde\delta$--thick shadings striped by a $(\tilde\tau,\alpha;\tilde\delta^{-2\eta/\eps})$-set.
\item The shadings $\tilde Y'(f)$ satisfy \eqref{rectangleConcentrationBdNonconcentratedProp}, with $T^{-\eta_2/2}$ in place of $T^{-\eps}$, and $O_{\eps}(\log 1/\tilde\delta)^{O_\eps(1)}$ in place of $\delta^{-\eta}$.
\end{itemize}

It is now time to apply the induction hypothesis: we apply Proposition \ref{maximalFnBdNonconcentrated} with $k-1$ in place of $k$; $\alpha$ unchanged; and $\eta_2/2$ in place of $\eps$. Let $\eta_3$ be the output from this proposition. If $\eta>0$ is selected sufficiently small, then $2\eta/\eps\leq\eta_3$. This means that the functions $f\in F$ are polynomials of degree at most $\tilde\delta^{-\eta_3}$, and the shadings $\tilde Y'(f)$ are striped by $(\tilde\tau,\alpha;\tilde\delta^{-\eta_3})$-sets. If $\delta_0$ and thus $\tilde\delta$ are sufficiently small, then the quantity $O_{\eps}(\log 1/\tilde\delta)^{O_\eps(1)}$ from the final item above is at most $\tilde\delta^{-\eta_3}$. Thus we can use the induction hypotheses to conclude that 
\begin{equation}\label{Lkm1Bd}
\Big \Vert \sum_{\tilde f\in \tilde F(R) }\chi_{\tilde Y'(\tilde f)}\Big\Vert_{\frac{k}{k-1}}\lesssim \delta^{-\eta_2/2}\tilde\delta^{\frac{k-1+\alpha}{k}} \tilde\tau^{\frac{(1-\alpha)(k-1)}{k}} (\#F(R)).
\end{equation}
We also have the $L^1$ bound
\begin{equation}\label{L1Bd}
\Big \Vert \sum_{\tilde f\in \tilde F(R) }\chi_{\tilde Y'(\tilde f)}\Big\Vert_{1} \leq \delta^{-\eta}\tilde\delta\tilde\tau^{1-\alpha}(\#F(R)).
\end{equation}
Interpolating \eqref{Lkm1Bd} and \eqref{L1Bd} using H\"older and recalling the definition of $\tilde\delta$ and $\tilde\tau$ (and the fact that $\eta\leq\eta_2$), we have
\begin{equation*}
\begin{split}
\Big \Vert \sum_{\tilde f\in \tilde F(R) }\chi_{\tilde Y'(\tilde f)}\Big\Vert_{\frac{k+1}{k}}^{\frac{k+1}{k}} 
& \leq \Big \Vert \sum_{\tilde f\in \tilde F(R) }\chi_{\tilde Y'(\tilde f)}\Big\Vert_{1}^{\frac{1}{k}} \Big \Vert \sum_{\tilde f\in \tilde F(R) }\chi_{\tilde Y'(\tilde f)}\Big\Vert_{\frac{k}{k-1}}\\
& \leq \delta^{-\eta_2}(\delta')^{-\frac{k+1}{k}} \delta^{\frac{k+\alpha}{k}}\tau^{1-\alpha}(\#F).
\end{split}
\end{equation*}

Undoing the scaling that maps $R$ to the unit square (this scaling distorts volumes by a factor of $(\delta')^{\frac{k+1}{k}}$), and recalling that $\#F(R)=\mu$ for each $R\in\mathcal{R}$, we conclude that
\begin{equation}\label{inductionHypothesesRectBd}
\int_R \Big( \sum_{f\in  F(R) }\chi_{Y'(f)}  \Big)^{\frac{k+1}{k}} \lesssim 
 \delta^{-\eta_2}\delta^{\frac{k+\alpha}{k}}\tau^{1-\alpha} \mu^{\frac{k+1}{k}}.
\end{equation}
Combining Item (A) from Proposition \ref{rectangleDecompProp} and \eqref{inductionHypothesesRectBd}, we conclude that 
\begin{equation}\label{controlOFLkNormInductionClose}
\begin{split}
\Big \Vert \sum_{f\in F}\chi_{Y(f)}\Big\Vert_{\frac{k+1}{k}}^{\frac{k+1}{k}} & \lesssim_\eps \delta^{-O(\sqrt\eta_2)}\sum_{R\in\mathcal{R}} \int_{R}\Big(\sum_{f\in F(R)}\chi_{Y'(f)}\Big)^{\frac{k+1}{k}}\\
& \lesssim  \delta^{-O(\eta_1)-\eta_2}(\#\mathcal{R})\delta^{\frac{k+\alpha}{k}}\tau^{1-\alpha} \mu^{\frac{k+1}{k}}\\
& \leq \delta^{-O(\eta_1)-\eta_2-\eps/2}\delta^{\frac{k+\alpha}{k}}\tau^{1-\alpha} (\# F)^{\frac{k+1}{k}}.
\end{split}
\end{equation} 
If we select $\eta_1$ and $\eta_2$ sufficiently small (depending on $\eps$ and $k$), then the term $\delta^{-O(\eta_1)-\eta_2-\eps/2}$ has size at most $\delta^{-\eps}$. This establishes \eqref{NonconcentratedPropLpBound} and closes the induction. 
\end{proof}


\section{Proof of Theorem \ref{mainThm}}\label{pfOfMainThmSec}
In this section, we will prove the following slightly more technical generalization of Theorem \ref{mainThm}.
\begin{mainThmTechnical}
Let $k\geq 1$, let $I$ be a compact interval, and let $\mathcal{F}\subset C^\infty(I)$ be uniformly smooth and forbid $k$--th order tangency. Let  $0<\beta\leq\alpha\leq 1$, and let $\eps>0$. 
Then there exists $\eta,\delta_0>0$ so that the following is true for all $\delta\in (0,\delta_0]$. Let $F\subset\mathcal{F}$ be a $(\delta, \beta; \delta^{-\eta})$-set (here $\mathcal{F}$ is given the usual metric on $C^k(I)$). For each $f\in F$, let $Y(f)\subset f^\delta$ be a $(\delta, \alpha;\delta^{-\eta})$-set (here we use the usual Euclidean metric on $\RR^2$). Then 
\begin{equation}\label{LpBoundTech}
\Big\Vert \sum_{f\in F} \chi_{Y(f)} \Big\Vert_{\frac{k+1}{k}} \leq \delta^{-\eps}\big( \delta^{2-\alpha}\#F\big)^{\frac{k}{k+1}}.
\end{equation}
If instead $0\leq\alpha\leq 1$ and $\beta>\alpha$, then in place of \eqref{LpBoundTech} we have
\begin{equation}
\Big\Vert \sum_{f\in F} \chi_{Y(f)} \Big\Vert_{\frac{k+1}{k}} \leq \delta^{-\eps}\big( \delta^{2-\alpha-\frac{\beta-\alpha}{k}}\#F\big)^{\frac{k}{k+1}}.
\end{equation}
\end{mainThmTechnical}
\noindent Theorem \ref{mainThm} is the special case where $\alpha=\beta=1$ and $Y(f)=f^\delta$ for each $f\in F$.

Before proving Theorem \ref{mainThm}$'$, let us examine how it differs from Proposition \ref{maximalFnBdNonconcentrated}. First, the functions in Proposition \ref{maximalFnBdNonconcentrated} are polynomials, while those in Theorem \ref{mainThm}$'$ are uniformly smooth; moving between these two conditions will not pose any difficulties. More care, however, is needed to move between the different non-concentration hypotheses imposed by Theorem \ref{mainThm}$'$ versus Proposition \ref{maximalFnBdNonconcentrated}. In brief, if $F$ is a family of curves that violates the non-concentration hypothesis \eqref{rectangleConcentrationBdNonconcentratedProp} from Proposition \ref{maximalFnBdNonconcentrated}, then for a typical point $x\in [0,1]^2$, the curves whose $\delta$ neighborhoods contain $x$ will be concentrated inside a small ball (in the metric space $C^k(I)$) in $\mathcal{F}$. Thus the curves in $F$ can be partitioned into non-interacting pieces, each of which is localized to a small ball in $\mathcal{F}$. Since $F$ is a $(\delta, \beta; \delta^{-\eta})$-set, each of these pieces only contains a small fraction of the total collection of curves.  Each of these pieces can then be rescaled to create an arrangement of curves that satisfies the non-concentration hypothesis \eqref{rectangleConcentrationBdNonconcentratedProp}. We now turn to the details.

\begin{proof}[Proof of Theorem \ref{mainThm}\,$'$]
$\phantom{1}$\\
\noindent\textbf{Step 1: Polynomial approximation.}
First, after a harmless rescaling we may suppose that $I=[0,1]$ and $\sup_{f\in\mathcal{F}}\Vert f \Vert_{C^{k+1}}\leq 1/2$. Let $\eta>0$ be a quantity to be chosen below. By Jackson's approximation theorem (see e.g.~\cite{BBL}), for each $f\in F$ there exists a polynomial $P_f$ of degree at most $K\delta^{-\eta/2}$, so that $\Vert f-P_f\Vert_{C^{k+1}}\leq \delta/4$. The quantity $K$ depends only on $\sup_{f\in \mathcal{F}}\Vert f^{(i)} \Vert_{C^m}$ for some $m\lesssim 1/\eta$. Crucially, $K$ is independent of $\delta$ and the specific choice of $F\subset\mathcal{F}$. In particular, if $\delta_0$ and hence $\delta$ is sufficiently small depending on $\eta$ and $\mathcal{F}$, then the degree of each polynomial $P_f$ is at most $\delta^{-\eta}$. Define $F_1=\{P_f\colon f\in F\}$. For each $P_f\in F_1$, define the shading $Y(P_f) = P_f^{2\delta}\cap N_{\delta}(Y(f))$. Abusing notation slightly, we replace $\delta$ by $2\delta$, so $Y(P_f)\subset P_f^\delta$. It suffices to prove Theorem \ref{mainThm}$'$ with $F_1$ in place of $F$, i.e.~we must show that
\begin{equation}\label{polyApproxBd}
\Big\Vert \sum_{f\in F_1} \chi_{Y(f)} \Big\Vert_{\frac{k+1}{k}} \leq  \delta^{-\eps}\big( \delta^{2-\alpha-\max\big(0, \frac{\beta-\alpha}{k}\big)}\#F\big)^{\frac{k}{k+1}}.
\end{equation}
Note that since the set $F$ is a $(\delta,\beta;\delta^{-\eta})$-set, we may suppose after a (harmless) refinement by a factor of $O(\delta^{\eta})$ that the points in $F$ are $2\delta$-separated. Hence the set $F_1$ is $\delta$-separated, and is a $(\delta, \beta; 2\delta^{-\eta})$-set. The set $F_1$ also satisfies the ``forbidding $k$--th order tangencies'' condition \eqref{cinematicFunctionCondition}, with $c$ replaced by $c/2$.

\medskip

\noindent\textbf{Step 2: A two-ends reduction.} Let $\eps_1>0$ be a small quantity to be specified below. For each $x\in\RR^2$, let $t(x)$ be the infimum of all $s\geq\delta$ for which there exists a closed ball $B\subset C^k$ of radius $s$ satisfying
\begin{equation}\label{sufficientlyManyFInBall}
\#\{f\in F_1 \cap B\colon x\in Y(f)\}\geq s^{\eps_1}\#\{f\in F_1\colon x\in Y(f)\}.
\end{equation}
The infimum $t(x)$ exists since the set of admissible $s$ contains $s=1$ and thus is non-empty; this is because our reductions from Step 1 ensure that $F_1$ is contained in the unit ball of $C^k$. For each $x\in\RR^2$, let $B(x)$ be a closed ball of radius $t(x)$ that satisfies \eqref{sufficientlyManyFInBall}. Then for each $x\in\RR^2$ we have
\begin{equation}\label{mostMassPreservedInTwoEnds}
\#\{f\in F_1\colon x\in Y(f),\ f\in B(x)\} \geq \delta^{\eps_1}\#\{f\in F_1\colon x\in Y(f)\}.
\end{equation}
For every $r\in[\delta,t]$ and all closed balls $B'\subset C^k$ of radius $r$, we have
\begin{equation}\label{ballConditionInsideBx}
\begin{split}
\#\{f\in F_1 \cap B'\colon x\in Y(f)\} & < r^{\eps_1}\#\{f\in F_1\colon x\in Y(f)\}\\
& \leq (r/t(x))^{\eps_1}\#\{f\in F_1\cap B(x)\colon x\in Y(f)\}.
\end{split}
\end{equation}
After dyadic pigeonholing, we can select a radius $t\in[\delta,1]$ so that
\[
\int_{\substack{x\in[0,1]^2 \\ t\leq t(x)< 2t}} \Big(\sum_{f\in F_1}\chi_{Y(f)}\Big)^{\frac{k+1}{k}}\gtrapprox \int_{[0,1]^2} \Big(\sum_{f\in F_1}\chi_{Y(f)}\Big)^{\frac{k+1}{k}}.
\]
For $f\in F_1$, define
\[
Y_1(f) = \{ x\in Y(f)\colon  t(x)\in[t, 2t),\ f\in B(x)\}.
\]
Then for each $x\in \bigcup_{f\in F_1}Y_1(f)$, there is a ball (abusing notation, we will denote this ball by $B(x)\subset C^k$) of radius $2t$ that contains every $f\in F_1$ with $x\in Y_1(f)$. 

By \eqref{mostMassPreservedInTwoEnds} we have 
\begin{equation}\label{LkBdControlledBySumOverShading}
        \Big\Vert \sum_{f\in F_1} \chi_{Y(f)} \Big\Vert_{\frac{k+1}{k}}  \lessapprox \delta^{-\eps_1}\Big\Vert \sum_{f\in F_1} \chi_{Y_1(f)} \Big\Vert_{\frac{k+1}{k}}.
 \end{equation}
 Finally, by \eqref{ballConditionInsideBx} we conclude that for each $x\in \RR^2$, each $r\in[\delta,t]$, and each ball $B\subset C^k$ of radius $r$, we have
 \begin{equation}\label{twoEndsBallStep2}
       \#\{f\in F_1 \cap B\colon x\in Y_1(f) \}\leq (2r/t)^{\eps_1}\#\{f\in F_1\colon x\in Y_1(f)\}.
 \end{equation}

Let $\mathcal{B}_0$ be a maximal set of pairwise non-intersecting balls in $C^k$ of radius $2t$, each of which intersects the unit ball in $C^k$. For each $B\in\mathcal{B}$, let $8B$ denote the ball with the same center and radius $8t$; denote this new set of balls by $\mathcal{B}_1$. Then for every $x\in \bigcup_{f\in F_1} Y_1(f)$, the ball $B(x)$ is contained in at least one of the balls from $\mathcal{B}_1$, and hence we have the pointwise bound
\begin{equation}\label{sumOverF1VsSumOverBalls}
\int \Big(\sum_{f\in F_1} \chi_{Y_1(f)} \Big)^{\frac{k+1}{k}}\leq \sum_{B\in\mathcal{B}_1} \int \Big(\sum_{f\in F_1\cap B} \chi_{Y_1(f)} \Big)^{\frac{k+1}{k}}.
\end{equation}

We claim that each $f\in\mathcal{F}$ is contained in $O(c^{-O(1)})$ balls from $\mathcal{B}$, where $c>0$ is the quantity from \eqref{cinematicFunctionCondition} associated to the family $\mathcal{F}$. From this it follows that
\[
\sum_{B\in\mathcal{B}_1}\#(F_1\cap B)\lesssim c^{-O(1)}\#F.
\]
To verify the above claim, suppose that $f\in F$ is contained in $\ell$ distinct balls with centers $g_1,\ldots,g_\ell$. Since the points $g_1,\ldots,g_\ell$ are $2t$-separated in $C^k$, by \eqref{cinematicFunctionCondition} we have that the vectors $v_j = \big(g_j(0), g_j'(0),\ldots, g_j^{(k)}(0)\big)$, $j=1,\ldots,\ell$ are at least $c t$-separated in $\RR^{k+1}$ with the $L^1$ metric. But since $\Vert f - g_j\Vert_{C^k}\leq 8t$ for each index $j$, by the triangle inequality the vectors $\{v_j\}_{j=1}^\ell$ are contained in a ball of radius $16t$. We conclude that $\ell\lesssim c^{-O(1)}$, as desired. 

Finally, we claim that either $t\geq\delta^{1-\frac{\eps}{2+2\beta}}$, or else \eqref{polyApproxBd} holds and we are done. Indeed, if $t<\delta^{1-\frac{\eps}{2+2\beta}}$ then for each $B\in\mathcal{B}_1$ we have
\[
\Big\Vert \sum_{f\in F_1\cap B} \chi_{Y_1(f)}\Big\Vert_\infty \leq \#(F_1\cap B)\lesssim \delta^{-\frac{\beta\eps}{2+2\beta}} \leq\delta^{-\eps/2},
\]
and thus
\begin{equation*}
\begin{split}
\Big\Vert \sum_{f\in F_1} \chi_{Y(f)} \Big\Vert_{\frac{k+1}{k}}^{\frac{k+1}{k}}
&\lessapprox \delta^{-\eps_1} \sum_{B\in\mathcal{B}_1} \int \Big(\sum_{f\in F_1\cap B} \chi_{Y_1(f)} \Big)^{\frac{k+1}{k}}\\
&\leq \delta^{-\eps/(2k)-\eps_1}\sum_{B\in\mathcal{B}_1}\int \Big(\sum_{f\in F_1\cap B} \chi_{Y_1(f)} \Big)\\
& \lesssim  \delta^{-\eps/(2k)-\eps_1}\sum_{f\in F_1}|Y_1(f)|\\
& \lesssim \delta^{-\eps/(2k)-\eps_1-\eta}\delta^{2-\alpha}(\#F),
\end{split}
\end{equation*}
and thus \eqref{polyApproxBd} holds, provided we select $\eps_1,\eta$ and $\delta_0$ sufficiently small. Henceforth we shall assume that $t\geq\delta^{1-\frac{\eps}{2+2\beta}}$. 

\medskip

\noindent\textbf{Step 3: Rescaling and Applying Proposition \ref{maximalFnBdNonconcentrated}.} 
For each $B\in\mathcal{B}_1$, with center $g_B$ and each $f\in F_1\cap B$, define $\tilde f_B(x) =(4t)^{-1}\big(f(x) - g_B(x)\big)$. Then $\Vert \tilde f_B\Vert_{C^{k}}\leq 1$ for each $f\in F_1\cap B$. Define $\tilde\delta = \delta/(2t)$; since $t\geq\delta^{1-\frac{\eps}{2+2\beta}}$ we have that $\tilde\delta\leq \delta^{\frac{\eps}{2+2\beta}}$. Let $\tilde Y_1(\tilde f_B)$ be the image of $Y_1(f)$ under the map $\phi_B(x,y)= \big(x, (4t)^{-1}(y - g_B(x))\big)$. Then $\tilde Y_1(\tilde f_B)\subset \tilde f_B^{\tilde\delta}$. The shading $\tilde Y_1(\tilde f_B)$ now satisfies the hypotheses of Proposition \ref{maximalFnBdNonconcentrated}, with $\tilde\delta$ in place of $\delta$ and $\tau = \delta$ (the bound $\tilde\delta\leq \delta^{\frac{\eps}{2+2\beta}}$ allows us to bound the error term $\delta^{-\eta}$ by $\tilde\delta^{-(2+2\beta)\eta/\eps}$). Define $\tilde F_B = \{\tilde f_B\colon f\in F_1\cap B\}$. 

The non-concentration estimate \eqref{twoEndsBallStep2} now has the following consequence. For each $r\in[\tilde\delta, 1]$ and each ball $B'\subset C^k$ of radius $r$, we have
    \begin{equation}\label{twoEndsBallRescaled}
        \#\{\tilde f_B\in \tilde F_B \cap B'\colon x\in \tilde Y_1(\tilde f_B) \}\leq 8 r^{\eps_1}\#\{\tilde f_B\in \tilde F_B\colon x\in \tilde Y_1(\tilde f_B)\}.
    \end{equation}
 Indeed, the inequality holds with constant $4$ in place of $8$ for $r\in[2\tilde\delta,1]$, and for $r\in[\tilde\delta,2\tilde\delta]$ we can replace $r$ by $2\tilde\delta.$

The consequence of \eqref{twoEndsBallRescaled} is the following: for each $x\in \bigcup_{\tilde f_B \in \tilde F_B}\tilde Y_1(\tilde f_B)$, each $\rho\geq\tilde\delta$, each $T\in [1,1/\rho]$, and each $(k; \rho; T)$ rectangle $R$ containing $x$, we have
\begin{equation}\label{rectEstimate}
\# \{\tilde f_B \in \tilde F_B \colon x\in \tilde Y_1(\tilde f_B),\ \tilde f_B\sim R\}\lesssim T^{-\eps_1}\# \{\tilde f_B \in \tilde F_B \colon x\in \tilde Y_1(\tilde f_B)\}.
\end{equation}
To verify \eqref{rectEstimate}, we argue as follows. Let $R = R^h(I)$, where $I$ is an interval of length $(T\rho)^{1/k}$. By Lemma \ref{smallImpliesSmallCkOnJ}, each function $\tilde f_B$ on the LHS of \eqref{rectEstimate} satisfies
\[
\Vert \tilde f_B - h\Vert_{C^k} \lesssim \sum_{i=0}^k \rho|I|^{-i} = \sum_{i=0}^k T^{-i/k}\rho^{1-i/k}\lesssim T^{-1},
\]
where the final inequality follows from the fact that $T\rho\leq 1$. 

Thus the set of functions on the LHS of \eqref{rectEstimate} is contained in a ball (in the metric space $C^k$)  of diameter $O(T^{-1})$. Comparing with \eqref{twoEndsBallRescaled}, we obtain \eqref{rectEstimate}. 

For each $B\in\mathcal{B}_1$, define $F_B = F_1\cap B$. Applying Proposition \ref{maximalFnBdNonconcentrated} with $\eps_1$ in place of $\eps$; $\tilde\delta$ in place of $\delta$; and $\tau=\delta$ (note that $\delta^{-\eta}\leq\tilde\delta^{-(2+2\beta)\eta/\eps}$), we conclude that if $\eta>0$ is sufficiently small, then for each ball $B\in\mathcal{B}_1$ we have (provided $\delta$, and thus $\tilde\delta$ is sufficiently small)
\begin{equation*}
\begin{split}
\int\Big(  \sum_{\tilde f_B\in\tilde F_B}\chi_{\tilde Y_1(\tilde f_B)}\Big)^{\frac{k+1}{k}} 
& \lesssim_\eps \tilde\delta^{-\eps_1} \tilde\delta^{1+\alpha/k}\tau^{1-\alpha}(\#\tilde F_B)^{\frac{k+1}{k}} \\
& \leq \delta^{-\eps_1} t^{-1-\alpha/k} \delta^{2-\alpha+\alpha/k}(\#\tilde F_B)^{\frac{k+1}{k}}.
\end{split}
\end{equation*} 
Undoing the scaling $\phi_B$ (which distorted volumes by a factor of $4t$) and using the fact that $\#F_B\lesssim \delta^{-\eta}(t/\delta)^\beta$ (since $F$ is a $(\delta, \beta; \delta^{-\eta})$-set), we have
\begin{equation}\label{boundForOneBall}
\begin{split}
\int\Big(  \sum_{  f\in F_B}\chi_{ Y_1(f)}\Big)^{\frac{k+1}{k}} 
&\lesssim_\eps  \delta^{-\eps_1} t^{-\alpha/k} \delta^{2-\alpha+\alpha/k}(\# F_B)^{\frac{k+1}{k}}\\
&\lesssim  \delta^{-\eps_1-\eta} t^{\frac{\beta-\alpha}{k}} \delta^{2 - \alpha + \frac{\alpha-\beta}{k}}(\# F_B)\\
& = \delta^{-\eps_1-\eta} (\delta/ t)^{\frac{\alpha-\beta}{k}}\delta^{2 - \alpha}(\# F_B).
\end{split}
\end{equation}
Combining \eqref{LkBdControlledBySumOverShading}, \eqref{sumOverF1VsSumOverBalls}, and \eqref{boundForOneBall}, we conclude that
\begin{equation}\label{mainThmFinalProofIneq}
\begin{split}
\Big\Vert \sum_{f\in F_1} \chi_{Y(f)} \Big\Vert_{\frac{k+1}{k}}^{\frac{k+1}{k}}
 & \lessapprox_{\eps} \delta^{-2\eps_1-\eta} \sum_{B\in\mathcal{B}_1} (\delta/ t)^{\frac{\alpha-\beta}{k}}\delta^{2 - \alpha}(\# F_B)\\ 
&\lesssim \delta^{-2\eps_1-\eta}(\delta/ t)^{\frac{\alpha-\beta}{k}}\delta^{2 - \alpha}(\# F).
\end{split}
\end{equation}
If $\alpha\geq\beta$, then the worst case occurs when $t=\delta$. This is unsurprising, in light of the behavior of Arrangements 1, 2, and 3 from Section \ref{introProofSketch}. If instead $\beta>\alpha$, then the worst case occurs when $t=1$. Regardless, we obtain \eqref{polyApproxBd}, provided we select $\eta,\eps_1\leq\eps/4$, and choose $\delta_0>0$ sufficiently small so that the implicit constant $O_{\eps}(\log(1/\delta)^{O_\eps(1)})$ in inequality \eqref{mainThmFinalProofIneq} is at most $\delta^{\eps/4}$.
\end{proof}


\section{From Theorem \ref{mainThm}$'$ to Theorem \ref{mainThmMaximalKakeya}}\label{proofOfThmMainThmMaximalKakeyaSection}

In this section we will prove Theorem \ref{mainThmMaximalKakeya}. We begin with the case $s=m-1$. It suffices to establish \eqref{maximalFnBdMDelta} for $p = s+1 = m$, since the case $p>s+1$ then follows by interpolation with the bound $\Vert M_\delta \Vert_{L^\infty\to L^\infty}\leq 1$.

 Let $h\colon \mathcal{C}\times I\to\RR$, $\Phi\colon \mathcal{C}\to\RR^{m-s}$, $\mathcal{C}_0\subset \mathcal{C}$, and $I_0\subset I$ be as in the statement of Theorem \ref{mainThmMaximalKakeya}. Since $\mathcal{C}_0$ and $I_0$ are compact and $h,\Phi$ are smooth, it suffices to consider the case where $\mathcal{C}=B(u_0,r)$ is a small neighborhood of a point $u_0$, and $I$ is a short interval. Since $h$ parameterizes an $m$-dimensional family of cinematic curves, if the neighborhood $\mathcal{C}$ and the interval $I$ are chosen sufficiently small, then there exists $c>0$ so that
\[
\inf_{t\in I}\sum_{j=0}^{m-1} |\partial_t^j h(u;t) - \partial_t^j h(u';t)|\geq c \sup_{t\in I}\sum_{j=0}^{m-1} |\partial_t^j h(u;t) - \partial_t^j h(u';t)|,\quad u,u'\in \mathcal{C},
\]
and thus the family $\mathcal{F} = \{h(\cdot, u)\colon u\in \mathcal{C}_0\}$ is uniformly smooth and forbids $(m-1)$--st order tangency.

The reduction from Theorem \ref{mainThm} to Theorem \ref{mainThmMaximalKakeya} now proceeds by a standard $L^p$ duality argument. We will briefly sketch the proof, and refer the reader to Lemma 10.4 from \cite{WolffBook} for further details. Let $\{v_i\}$ be a maximal $\delta$-separated subset of $\Phi(\mathcal{C}_0)$. Since $\Phi$ is a submersion, if the neighborhood of $u_0$ and the interval $I$ were selected sufficiently small, then there is a number $A$ so that for all $v,v'\in \Phi(\mathcal{C})$ and all $u\in\Phi^{-1}(v)$, there exists $u'\in\Phi^{-1}(v')$ so that
\begin{equation}\label{curvesComparableByA}
\sup_{t\in I}|h(u;t)-h(u',t)|\leq A|v-v'|.
\end{equation}
This implies that for all $v,v'\in \Phi(\mathcal{C})$ with $|v-v'|\leq\delta$, we have $M_\delta f(v)\leq A M_{A\delta} f(v')$. Define $\tilde\delta = A\delta$. We have
\[
\Vert M_\delta f\Vert_m \lesssim A\Big(\delta \sum_j |M_{\tilde\delta} f(v_j)|^m\Big)^{1/m}.
\]
By the duality of $\ell^m$ and $\ell^{m/(m-1)}$, there exists a sequence $\{y_j\}$ with $\delta\sum_j |y_j|^{m/(m-1)}=1$, so that
\[
\Big(\delta \sum_j |M_{\tilde\delta} f(v_j)|^m\Big)^{1/m}=\delta \sum_j y_j|M_{\tilde\delta} f(v_j)|,
\]
and thus 
\[
\Vert M_\delta f\Vert_m \lesssim  \tilde \delta \sum_j \Big(y_j \frac{1}{\tilde\delta}\int_{g_j^{\tilde\delta}}|f|\Big)=\int \Big(\sum_j y_j\chi_{g_j^{\tilde\delta}}\Big)|f|,
\]
where $g_j\in \mathcal{F}$ is a function that comes within a factor of $1/2$ of achieving the supremum $M_{\tilde\delta} f(v_j)$. We now use H\"older's inequality to bound
\[
\int \Big(\sum_j y_j\chi_{g_j^{\tilde\delta}}\Big)|f| \leq \Big\Vert \sum_j y_j\chi_{g_j^{\tilde\delta}} \Big\Vert_{m/(m-1)}\Vert f\Vert_m.
\] 

Thus in order to establish \eqref{maximalFnBdMDelta}, it suffices to show that
\begin{equation}\label{dualizedLpIneq}
\Vert \sum_j y_j\chi_{g_j^{\tilde\delta}} \Big\Vert_{m/(m-1)} \leq \delta^{-\eps}.
\end{equation}
We would like to apply Theorem \ref{mainThm}, but we must first deal with the weights $\{y_j\}$. Since we do not care about factors of $\log(1/\delta)$, this can be handled using dyadic pigeonholing. Let $G_0 = \{g_j\colon |y_j|\leq 1\}$. For $k\geq 1$, let $G_k = \{ g_j\colon 2^k\leq |y_j| < 2^{k+1}\}$. We have $\#G_0\lesssim \delta^{-1}$, and for $k\geq 1$ we have $\# G_k\lesssim \delta^{-1} 2^{-km/(m-1)}$. By the triangle inequality, we have

\begin{equation}
\begin{split}
\Big\Vert \sum_j y_j\chi_{g_j^{\tilde\delta}} \Big\Vert_{\frac{m}{m-1}}
& = \Big\Vert\sum_{j\colon |y_j|\leq 1} y_j\chi_{g_j^{\tilde\delta}}\ + \!\!\sum_{\substack{k\\ 1 \leq 2^k\leq \delta^{-1}}}\ \sum_{\substack{j\\ 2^k\leq |y_j|<2^{k+1}}} y_j\chi_{g_j^{\tilde\delta}} \Big\Vert_{\frac{m}{m-1}}\\
&\lesssim \sum_{k=0}^{O(\log 1/\delta)} 2^{k} \Big\Vert \sum_{g\in G_k} \chi_{g^{\tilde\delta}} \Big\Vert_{\frac{m}{m-1}}.
\end{split}
\end{equation}
We can now apply Theorem \ref{mainThm} (with $\eps/2$ in place of $\eps$) to bound each summand. Summing the resulting contributions, we obtain \eqref{dualizedLpIneq},  provided $\delta>0$ is selected sufficiently small.

\begin{rem}\label{dependenceOnDeltaRemark}
The conclusion \eqref{maximalFnBdMDelta} of Theorem \ref{maximalFnBdMDelta} holds for all $\delta>0$ sufficiently small. More precisely, the conclusion holds for all $\delta\in (0,\delta_0]$, where $\delta_0$ depends on the following quantities:
\begin{itemize}
\item $m$ and $\eps$ (recall that the argument above assumes that $s = m-1$).

\item The infimum of $|\det DF^h_t(u)|$ from Definition \ref{paramFamily}, for $(u,t)\in \mathcal{C}_0\times I_0$; this quantifies the property that $h$ parameterizes an $m$-dimensional family of cinematic curves.


\item $\sup|\nabla \Phi|$; in order for $F$ to be a $(\delta, \beta; \delta^{-\eta})$-set, we need this supremum to be at most $\delta^{-\eta}$. 

\item The quantity $A$ from \eqref{curvesComparableByA}.

\item The $C^m$ norm of $h$, where $m=m(\eps)$ is a large integer depending on $\eps$. More precisely, we can cover $\mathcal{C}_0\subset\mathcal{C}$ by a finite (independent of $\delta$) set of coordinate charts, and our choice of $\delta_0$ will depend on the maximum of the $C^m$ norm of $h$ in these coordinate charts. 
\end{itemize}
\end{rem}

Next, we consider the case $s<m-1$. The reduction from $s<m-1$ to $s=m-1$ is a ``slicing'' argument. Again, since $\mathcal{C}_0$ and $I_0$ are compact and $h,\Phi$ are smooth, it suffices to consider the case where $\mathcal{C}=B(u_0,r)$ is a small neighborhood of a point $u_0$, and $I$ is a short interval. In particular, we can suppose that there is a unit vector $e\in \RR^{m-s}$ so that for each $u\in \mathcal{C}$ and each $t\in I$, if we consider the manifold $V_{u;t}$ given by \eqref{manifoldVut},
then $\Phi(V_{u;t})$ is a codimension-1 manifold in $\RR^{m-s}$, and at each point $p\in \Phi(V_{u;t})$, the tangent plane $T_p \Phi(V_{u;t})$ has normal vector that makes angle at most $1/100$ with $e$. For example, in Example \ref{runningExampleMomentCurve} from the introduction, for each $(u,t)\in\mathcal{C}\times[0,1]$ and each $p\in \Phi(V_{u;t})$, the tangent plane $T_p \Phi(V_{u;t})$ is given by $\operatorname{span}\{e_{s+2},\ldots,e_m\}$, and thus is normal to $e_{s+1}$ (recall that in Example \ref{runningExampleMomentCurve}, $\Phi$ is the projection to the last $m-s$ coordinates, and thus we identify $\RR^{m-s}$ with $\operatorname{span}\{e_{s+1},\ldots,s_m\}$; in general, however, we will use coordinates $e_1,\ldots,e_{m-s}$ for $\RR^{m-s}$).

After a harmless rotation (and using the convention that the basis vectors of $\RR^{m-s}$ are $e_1,\ldots,e_{m-s}$), we may suppose that $e = e_1$ is the first standard basis vector. After further restricting $\mathcal{C}$ and translating, we may suppose that $\Phi(\mathcal{C})$ is the cube $Q=[0,r]^{m-s}$ for some small $r>0$ (recall that by hypothesis, $\Phi$ is an immersion). Writing $v = (v_1, \bar v)\in\RR\times \RR^{m-s-1}$ and $Q=[0,r]\times \bar Q$, we have

\begin{equation}\label{slicingM}
\begin{split}
\big\Vert M_\delta f\big\Vert_{L^{s+1}(Q)}& =\Big(\int_{\bar Q}\int_0^r (M_\delta f(v))^{s+1}dv\Big)^{\frac{1}{s+1}}\\
&\leq \Big(\sup_{\bar v\in\bar Q}\int_0^r(M_\delta f( v_1, \bar v)^{s+1}dv_{1}\Big)^{\frac{1}{s+1}}\\
&=\sup_{\bar v\in\bar Q}\big\Vert M_\delta f(\cdot, \bar v)\big\Vert_{L^{s+1}([0,r])}\\
&=\sup_{\bar v\in\bar Q}\big\Vert M^{\bar v}_\delta f\big\Vert_{L^{s+1}([0,r])},
\end{split}
\end{equation}
where 
\[
M^{\bar v}_\delta f(v_1)=\frac{1}{\delta}\sup_{u\in \Phi^{-1}(v_1,\bar v)}\Big|\int_{\gamma_u^\delta} f\Big|.
\]
The purpose of the above computation is that $M^{\bar v}_\delta$ is a maximal operator in the sense of Definition \ref{defnClassOfMaximalOperators}, with $s+1$ in place of $m$. To see this, define 
\[
\mathcal{C}^{\bar v} = \{u\in\mathcal{C}\colon \pi(\Phi(u)) = \bar v\},
\]
where $\pi\colon\RR^{m-s}\to\RR^{m-s-1}$ is the projection to the last $m-s-1$ coordinates. Since $\Phi$ is an immersion, this is a manifold of dimension $m - (m-s-1) = s+1$. Define $h^{\bar v}$ to be the restriction of $h$ to $\mathcal{C}^{\bar v}\times I$. For each $t\in I$, define
\begin{equation*}
\begin{array}{lll}
F^{h^{\bar v}}_t(u)  \colon& \mathcal{C}^{\bar v}\to\RR^{s+1},\quad & u\mapsto \big(h^{\bar v}(u;t),\ \partial_t h^{\bar v}(u;t),\ldots,\partial_t^{s} h^{\bar v}(u;t)\big),\\
\tilde F^{h}_t(u)  \colon& \mathcal{C}\to\RR^{s+1},\ & u\mapsto \big(h(u;t),\ \partial_t h(u;t),\ldots,\partial_t^{s} h(u;t)\big).
\end{array}
\end{equation*}
Note that $F^{h^{\bar v}}_t$ is the restriction of $\tilde F^{h}_t$ to $\mathcal{C}^{\bar v}$. 

We claim that for each $t\in I$, $F^{h^{\bar v}}_t$ is a local diffeomorphism, and thus $h^{\bar v}$ parameterizes an $(s+1)$-dimensional family of cinematic curves, in the sense of Definition \ref{paramFamily}. To verify this claim, it suffices to show that for each $t\in I$ and each $u\in\mathcal{C}^{\bar v}$, $DF^{h^{\bar v}}_t(u)$ has full rank. 

Fix a choice of $t\in I$ and $u\in\mathcal{C}^{\bar v}$. Since $\mathcal{C}^{\bar v}\subset\mathcal{C}$, we can identify $u$ with a point in $\mathcal{C}$, and we can identify $T_u \mathcal{C}^{\bar v}$ with an $(m-s-1)$-dimensional subspace of $T_u\mathcal{C}$; denote this subspace by $W$. Since $h$ parameterizes an $m$-dimensional family of cinematic curves, we have that $DF^h_u$ has full rank (i.e.~rank $m$), and thus $D\tilde F^h_u$ also has full rank (i.e.~rank $s+1$). Our goal is to show that the restriction of $D\tilde F^{h}_t$ to $W$ still has rank $s+1$. 

By construction, $T_u V_{u;t}$ (recall \eqref{manifoldVut} for a definition) is orthogonal to $W$. By hypothesis, we have that $\Phi$ is transverse to $h$, and thus the manifolds $\mathcal{C}^{\bar v}$ and $V_{u;t}$ intersect transversely at $u$. The restriction of $D\tilde F^{h}_t$ to $T_u V_{u;t}$ has rank 0---this is because $\tilde F^h_t$ is constant on $V_{t;u}$. We conclude that the restriction of $D\tilde F^{h}_t$ to $W$ has the same rank as $D\tilde F^h_t$, i.e.~the restriction of $D\tilde F^{h}_t$ to $W$ has rank $s+1$, as desired.

Define $\Phi^{\bar v}$ to be the restriction of $\Phi$ to $\mathcal{C}^{\bar v}$. Since $\Phi$ is transverse to $h$, in the sense of Definition \ref{defnQuantTranverse}, we have that $\Phi^{\bar v}$ is a submersion. We have
\[
M^{\bar v}_\delta f(v_1)=\frac{1}{\delta}\sup_{u\in (\Phi^{\bar v})^{-1}(v_1)}\Big|\int_{\gamma_u^\delta} f\Big|,
\]
and this is an $s$-parameter maximal function associated to an $(s+1)$-dimensional family of cinematic curves, in the sense of Definition \ref{defnClassOfMaximalOperators}. Each of the quantities discussed in Remark \ref{dependenceOnDeltaRemark} associated to the maximal function $M^{\bar v}_\delta$ is uniformly bounded in the choice of $\bar v$. 

 Thus we can apply Theorem \ref{mainThmMaximalKakeya} with $s+1$ in place of $m$ to conclude that there exists a choice of $\delta_0>0$ (which is uniform in our choice of $\bar v$---see Remark \ref{dependenceOnDeltaRemark}) so that $\Vert M^{\bar v}_\delta\Vert_{L^{s+1}\to L^{s+1}}\leq\delta^{-\eps}$ for all $\bar v\in\bar Q$ and all $\delta\in (0,\delta_0]$. This means that for $\delta\in (0,\delta_0]$, we have
\begin{equation}\label{unifBdOnSlices}
\sup_{\bar v\in\underline Q}\big\Vert M^{\bar v}_\delta f\big\Vert_{L^{s+1}([0,r])}\leq\delta^{-\eps}\Vert f\Vert_{s+1}.
\end{equation}
Combining \eqref{slicingM} and \eqref{unifBdOnSlices}, we obtain \eqref{maximalFnBdMDelta}.


\section{From Theorem \ref{mainThmMaximalKakeya} to Theorem \ref{LpMaximalFnBd} }\label{proofOfTheoremLpMaximalFnBdSec}
In this section we will prove Theorem \ref{LpMaximalFnBd}. The main new input is a local smoothing estimate by Chen, Guo, and Yang \cite{CGYv6}. As noted in the introduction, Chen, Guo, and Yang prove sharp $L^p\to L^p$ bounds for the axis-parallel elliptic maximal function by combining their local smoothing theorem with an estimate similar to \eqref{maximalFnBdMDelta}. We will follow a similar strategy. We begin by recalling the setup from \cite{CGYv2}.

\subsection{Local smoothing: The Chen-Guo-Yang framework}
Let $s\geq 2$, $w = (w_1,\ldots w_s)$. Let $\zeta(w;t)\colon \RR^{s}\times\RR \to\RR$ be smooth, let $\phi(w, t)$ be a smooth bump function supported near the origin. Define
\begin{equation}\label{GuoA}
A_{\zeta,\phi}f(x,y; w)=\int_{\RR}f\big(x-t, y - \zeta(w; t)\big)\phi(w, t)dt,
\end{equation}
and define 
\begin{equation}\label{GuoM}
G_{\zeta,\phi}f(x,y) = \sup_{w\in\RR^{s}}| A_{\zeta,\phi}f(x,y;w) |.
\end{equation}

Next, we define an analogue of Sogge's cinematic curvature condition from \cite{Sogge} in this setting.  Define the column vector
\[
T^\zeta(w;t) = \big(\partial_{t}\zeta(w; t),\ \partial^2_{t}\zeta(w;t),\ldots,\partial^{s+1}_{t}\zeta(w;t)\big)^T.
\]
\begin{defn}\label{sParamCinematicCondition}
We say that $G_{\zeta,\phi}$ satisfies the \emph{$s$ parameter curvature condition} at the origin if
\begin{equation}\label{GuoCinematic}
\det\big[\partial_{t}T^\zeta,\  \partial_{w_1}T^\zeta,\  \partial_{w_2}T^\zeta, \ldots, \partial_{w_{s}}T^\zeta\big]\Big|_{(w,t)=(0,0)}\neq 0.
\end{equation}
\end{defn}
By continuity, if \eqref{GuoCinematic} is satisfied, then the determinant continues to be nonzero for $(w,t)$ in a small neighborhood of the origin. The bump function $\phi$ is chosen so that this determinant is uniformly bounded away from $0$ on the support of $\phi$. 

Now we can state Proposition 3.2 from \cite{CGYv2}. In what follows, $P_kf$ denotes the Littlewood-Paley projection to the frequency annulus of magnitude $\sim 2^k$. 
\begin{prop}[Chen, Guo, Yang]\label{CGYProp}
Let $\zeta(w,t)\colon \RR^s\times\RR\to\RR$ satisfy the $s$ parameter curvature condition at the origin. Then there exists $p_s=p(s)<\infty$ so that for all $\eps>0$ and all smooth bump functions $\phi(w,t)$ whose support is contained in a sufficiently small neighborhood of the origin, there is a constant $C = C(\eps,\zeta, \phi)$ so that

\begin{equation}\label{LpBoundSmoothingAGamma}
\Vert A_{\zeta,\phi}(P_kf)\Vert_{L^p(\RR^2\times\RR^{s})}\leq C 2^{-(\frac{s+1}{p}+\eps)k}\Vert f \Vert_{L^p(\RR^2)}.
\end{equation}
\end{prop}
\begin{rem}
In \cite{CGYv2}, Proposition 3.2 is stated with the additional hypothesis that $\zeta(w,t)$ is a ``normal form'' at the origin (this is defined in Definition 3.1 from \cite{CGYv2}). However, the argument immediately following Proposition 4.2 from \cite{CGYv2} shows how an arbitrary  $\zeta(w,t)$ can be reduced to the case where $\zeta$ is a normal form. We also remark that the analogue of \eqref{LpBoundSmoothingAGamma} from \cite{CGYv2} has the expression $A_{\zeta,\phi}f$ rather than $A_{\zeta,\phi}(P_kf)$, but the latter is what is intended. 
\end{rem}

Note that
\[
\Vert G_{\zeta,\phi}(P_kf)(x,y))\Vert_{L^p_{xy}}=
\Vert A_{\zeta,\phi}(P_k f)(x,y;w))\Vert_{L^p_{xy}(L^\infty_{w})},
\]
where $L^p_{xy}$ denotes $L^p(\RR^2)$ in the variables $(x,y)$ and $L^\infty_{w}$ denotes $L^\infty(\RR^{s})$ in the variable $w$. Thus by Sobolev embedding, \eqref{LpBoundSmoothingAGamma} implies
\begin{equation}\label{conclusionSobolev}
\Vert G_{\zeta,\phi}(P_kf) \Vert_{L^p(\RR^2)}\leq C 2^{-(\frac{1}{p}+\eps)k}\Vert f \Vert_{L^p(\RR^2)},
\end{equation}
with $p=p(s)$ as above, and for a (possibly different) constant $C = C(\eps,\gamma,\chi)$. Inequality \eqref{conclusionSobolev} says that the sublinear operator $G_{\zeta,\phi}$ has high frequency decay, in the sense of Definition \ref{highFrequencyDecayCondition}.

\subsection{From local smoothing to maximal averages over curves}

Our next task is to relate the maximal operator $G_{\zeta,\phi}f$ from \eqref{GuoM} to the operator $M$ from Definition \ref{defnClassOfMaximalOperators}. By compactness, it suffices to consider the case where $\mathcal{C}$ is a small neighborhood of a point $u_0$, and $I$ is a small interval. Since we restrict to the case where $M$ is translation invariant, we may choose local coordinates of the form $u=(x,y,w_1,\ldots,w_s)$ so that the parameterization and projection functions $h\colon  \mathcal{C}\times I\to\RR$ and $\Phi\colon \mathcal{C} \to\RR^{2}$ can be expressed in the form $h(u; t) = \zeta(w_1,\ldots,w_s; t-x) + y$ and  $\Phi(u) = (x,y)$; we can choose these coordinates so that $u_0 = 0$ and $I$ is an interval centered at 0. Let $G = G_{\zeta,\phi}$, where $\phi$ is a bump function chosen so that Proposition \ref{CGYProp} holds. We further restrict $\mathcal{C}$ and $I$ so that $\phi$ is identically 1 on $\mathcal{C}\times I$. With these restrictions, for every non-negative function $f\colon\RR^2\to\RR$ we have
\begin{equation}\label{MfDominatedGuo}
Mf(x,y) = \sup_{w\colon (x,y,w)\in\mathcal{C}_0}\int_{\gamma_w}f\ \  \leq\ \ \sup_{w\in\RR^s}A_{\zeta,\phi}f(x,y; w)\leq G_{\zeta,\phi}f(x,y).
\end{equation}

Let us suppose for the moment that $\zeta(w,t)\colon \RR^s\times\RR\to\RR$ satisfies the $s$ parameter curvature condition at the origin. Theorem \ref{mainThmMaximalKakeya} says that for each $\eps>0$, there exists a constant $C_\eps$ so that
\begin{equation}\label{sharpLsBdMaximal}
\Vert G_{\zeta,\phi}P_kf \Vert_{L^{s+1}(\RR^2)}\leq C_\eps 2^{\eps k}\Vert f \Vert_{L^{s+1}(\RR^2)}.
\end{equation}
Indeed, we have
\begin{equation}
G_{\zeta,\phi}(P_kf)(x,y) \lesssim \sum_{j=0}^k M_{2^{-j}}f(x,y),
\end{equation}
where $M_{2^{-j}}$ is the maximal operator from \eqref{defnMaximalFnDelta} associated to $h$ and $\Phi$ at scale $\delta=2^{-j}$. The conclusion of Theorem \ref{mainThmMaximalKakeya} holds for all $\delta>0$ sufficiently small (depending on $\eps, \mathcal{C}, h,$ and $\Phi$), but this may be extended to all $\delta>0$ by selecting a sufficiently large constant $C_\eps$.

Let $p>s+1$. If we select $\eps>0$ sufficiently small depending on $p$ and the Lebesgue exponent $p(s)$ from \eqref{conclusionSobolev}, then by interpolating \eqref{conclusionSobolev} and \eqref{sharpLsBdMaximal} we conclude that there exist constants $\eta>0$ (small) and $C$ (large) so that
\[
\Vert G_{\zeta,\phi}P_kf \Vert_{L^{p}(\RR^2)}\leq C 2^{-\eta k}\Vert f \Vert_{L^p(\RR^2)},
\]
and hence there is a constant $C_p$ so that
\begin{equation}\label{LpBdChi}
\Vert G_{\zeta,\phi}f \Vert_{L^{p}(\RR^2)}\leq C_p \Vert f \Vert_{L^p(\RR^2)}.
\end{equation}
Since it suffices to prove Theorem \ref{LpMaximalFnBd} for non-negative functions, the theorem now follows from \eqref{MfDominatedGuo}.

It remains to verify that $\zeta(w,t)\colon \RR^s\times\RR\to\RR$ satisfies the $s$ parameter curvature condition at the origin. By hypothesis, $h$ parameterizes an $(s+2)$-dimensional family of cinematic curves, in the sense of Definition \ref{paramFamily}. To slightly simplify notation below, we write coordinates $u = (y, x, w_1,\ldots,w_s)$ rather than $(x,y, w_1,\ldots,w_s)$. We have

\begin{equation}\label{determinantH}
DF^h_0(0)=
\left( 
\begin{array}{ccccc}
1 & \partial_t h & \partial_{w_1}h  & \cdots & \partial_{w_s}h \\
0 & \partial_t\partial_t h & \partial_t \partial_{w_1}h  & \cdots & \partial_t \partial_{w_s}h \\
\vdots&\vdots&\vdots&\ddots & \vdots\\
0 & \partial_t^{s+1}\partial_t h & \partial_t^{s+1} \partial_{w_1}h& \cdots & \partial_t^{s+1} \partial_{w_s}h 
\end{array}
\right).
\end{equation}
The bottom-right $(s+1)\times (s+1)$ minor of the above matrix is precisely $T^\zeta(w;t)$, and hence these matrices have the same determinant. Since $h$ parameterizes an $(s+2)$-dimensional family of cinematic curves, this determinant is non-vanishing at $(u;t)=(0;0)$ (recall that $u_0 = 0$). We conclude that $\zeta(w,t)\colon \RR^s\times\RR\to\RR$ satisfies the $s$ parameter curvature condition at the origin.


\section{From Theorem \ref{mainThm}$'$ to Theorems \ref{KaufmanThm} and \ref{generalizedFurstenbergSetsThm}}\label{reductionMainThmsSection}
In this section we will briefly discuss the reduction from Theorem \ref{mainThm}$'$ to Theorems \ref{KaufmanThm} and \ref{generalizedFurstenbergSetsThm}. Reductions of this type are already present in the literature, so we will just provide a brief sketch and refer the reader to the appropriate sources for further details.

\subsection{Restricted Projections} 
The connection between exceptional set estimates for projections in restricted sets of directions, and maximal function estimates for curves was first explored by K\"{a}enm\"{a}ki, Orponen, and Venieri in \cite{KOV2021}. We will follow the framework from Section 2 of \cite{PYZ}. We will only briefly sketch the numerology of the problem, and refer the reader to \cite{PYZ} for details. 

Let $\gamma\colon[0,1]\to\RR^n$ and $E\subset\RR^n$ be as in the statement of Theorem \ref{KaufmanThm}. After rescaling and replacing $E$ by a subset, we may suppose that $E\subset[-1,1]^n$ and $\dim E\leq 1$. Suppose for contradiction that there exists some $0\leq q<\dim E$ so that the set
\[
S = \{t\in [0,1]\colon \dim(E\cdot\gamma(t))<q\}
\]
satisfies $\dim S>q$. After possibly replacing $S$ and $E$ by subsets, we may suppose that $q<\dim S = \dim E$. Let $\alpha= \dim S$. Let $\mathcal{F} = \{t\mapsto z\cdot \gamma(t)\colon z\in [-1,1]^n\}$. Since $\gamma$ is smooth, the set $\mathcal{F}$ is uniformly smooth, and the nondegeneracy condition \eqref{nonDegenCurveCondition} implies that $\mathcal{F}$ forbids $(n-1)$-st order tangency. Define $\mathcal{F}_E = \{t\mapsto z\cdot \gamma(t)\colon z\in E\}$.

Let $\eta,\delta_0>0$. In Section 2 of \cite{PYZ}, the authors explain how to extract a $(\delta, \alpha; \delta^{-\eta})$-set $F\subset\mathcal{F}_E$, for some $\delta\in(0,\delta_0]$, and how to construct a shading $Y(f)\subset f^\delta$ for each $f\in F$, where each set $Y(f)$ is a $(\delta,\alpha;\delta^{-\eta})$-set (in the metric space $\RR^2$), with the property that the union $\bigcup_{f\in F}Y(f)$ is contained in a $(\delta, \alpha; \delta^{-\eta})\times (\delta, q; \delta^{-\eta})$ quasi-product, i.e.~a set $X\subset\RR^2$ whose projection to the $x$-axis is a $(\delta, \alpha; \delta^{-\eta})$-set, and every fiber above this projection is a $(\delta, q; \delta^{-\eta})$-set. In particular, such a quasi-product has measure at most $\delta^{2-\alpha-q-2\eta}$. Since $\sum_{f\in F}|Y(f)|\gtrsim\delta^{2 - 2\alpha+2\eta},$ by H\"older's inequality we have
\begin{equation}\label{HolderLowerBd}
\Big\Vert \sum_{f\in F}\chi_{Y(f)} \Big\Vert_{\frac{n}{n-1}}^{\frac{n}{n-1}} \geq \delta^{-2\alpha + \frac{q-\alpha}{n-1}+O(\eta)}.
\end{equation}
On the other hand, by Theorem \ref{mainThm} with $k=n-1$, for each $\eps>0$ we have
\begin{equation}\label{consequenceOfMainThmMaximal}
\Big\Vert \sum_{f\in F}\chi_{Y(f)} \Big\Vert_{\frac{n}{n-1}}^{\frac{n}{n-1}} \leq \delta^{-2\alpha-\eps},
\end{equation}
provided $\delta$ is sufficiently small. Since $q<\alpha$, we obtain a contradiction provided $\eps,\eta$, and $\delta_0$ are chosen sufficiently small. We refer the reader to Section 2 of \cite{PYZ} for details.

\subsection{Furstenberg sets of curves}
In this section we will briefly discuss the proof of Theorem \ref{generalizedFurstenbergSetsThm}. In \cite{HSY}, H\'era, Shmerkin, and Yavicoli obtained new bounds for the dimension of $(\alpha, 2\alpha)$ Furstenberg sets. They did this by first introducing the notion of a discretized $(\alpha,\beta)$ Furstenberg set, and then showing that covering number bounds on the size of such discretized Furstenberg sets imply Hausdorff dimension bounds for the corresponding $(\alpha,\beta)$ Furstenberg sets. An identical strategy will work here. The corresponding notion of a discretized Furstenberg set of curves is as follows. 

\begin{defn}\label{discretizedFurstenbergSetOfCurves}
Let $\mathcal{F}\subset C^k$. For $\alpha,\beta,\delta, C>0$, we say a set $E\subset[0,1]^2$ is a discretized $(\delta, \alpha, \beta)$ Furstenberg set of curves (with error $C$) from the family $\mathcal{F}$, if $E=\bigcup_{f\in F}A_f$, where
\begin{itemize}
    \item The set $F\subset \mathcal{F}$ is a $(\delta, \beta; C)$-set (in the metric space $C^k$, with $\#F \geq C^{-1}\delta^{-\beta}$.
    \item For each $f\in F$, the set $A_f$ is a $(\delta,\alpha;C)$-set (in the metric space $\RR^2$), with $|A_f|\geq C^{-1}\delta^{2-\alpha}$, which is contained in $f^{2\delta}$.
\end{itemize}
\end{defn}

Definition \ref{discretizedFurstenbergSetOfCurves} is modeled after Definition 3.2 from \cite{HSY}. The definitions are very similar, with the following two differences: in \cite{HSY}, the authors considered lines in $\RR^n$ rather that a family $\mathcal{F}$ of curves, and the authors used the notation ``$\lessapprox$'' to suppress the role of the constant $C$. 

In Lemma 3.3 from \cite{HSY}, the authors prove the following: Let $\alpha,\beta,s\geq 0$. Suppose that for every $\eps>0$, there exists $\eta>0$ so that every $(\delta, \alpha,\beta)$ Furstenberg set of lines (with error $\delta^{-\eta})$ has measure at least $\delta^{2-s+\eps}$. Then every $(\alpha,\beta)$ Furstenberg set has Hausdorff dimension at least $s$. 

An identical proof yields the analogous result for Furstenberg sets of curves: Fix a family $\mathcal{F}\subset C^k$, and fix $\alpha,\beta,s\geq 0$. Suppose that for every $\eps>0$, there exists $\eta>0$ so that every $(\delta, \alpha,\beta)$ Furstenberg set of curves (with error $\delta^{-\eta})$ from the family $\mathcal{F}$ has measure at least $\delta^{2-s+\eps}$. Then every $(\alpha,\beta)$ Furstenberg set of curves from $\mathcal{F}$ has Hausdorff dimension at least $s$. Thus in order to prove Theorem \ref{generalizedFurstenbergSetsThm}, it suffices to obtain the corresponding bound on the volume of discretized $(\delta, \alpha, \beta)$ Furstenberg sets of curves from $\mathcal{F}$. To this end, we have the following result.

\begin{lem}\label{volumeBdDiscretizedFurstenbergCurves}
Let $\eps>0, k\geq 1,$ and let $0\leq\beta\leq \alpha\leq 1.$ Then there exists $\eta,\delta_0>0$ so that the following holds for all $\delta\in(0,\delta_0]$. Let $\mathcal{F}$ be a family of uniformly sooth curves that forbid $k$--th order tangency. Let $E\subset[0,1]^2$ be a discretized $(\delta, \alpha, \beta)$ Furstenberg set of curves (with error $\delta^{-\eta}$) from the family $\mathcal{F}$. Let $F\subset\mathcal{F}$ and $\{Y(f)\colon f\in F\}$ be as in Definition \ref{discretizedFurstenbergSetOfCurves}. Then 
\begin{equation}
|E|\geq \delta^{2-\alpha-\beta+\eps}.
\end{equation}
\end{lem}
Lemma \ref{volumeBdDiscretizedFurstenbergCurves} follows from Theorem \ref{mainThm} and H\"older's inequality; we have
\begin{align*}
\delta^{2-\alpha-\beta}=\Big\Vert \chi_E \sum_{f\in F}\chi_{Y(f)}\Big\Vert_1 
&\leq \Big\Vert \chi_E\Big\Vert_{k+1}\Big\Vert \sum_{f\in F}\chi_{Y(f)}\Big\Vert_{\frac{k+1}{k}}\\
&\leq |E|^{\frac{1}{k+1}}\delta^{-\frac{\eps}{k+1}}(\delta^{2-\alpha-\beta})^{\frac{k}{k+1}}.
\end{align*}

\appendix


\section{Examples}\label{examplesSection}

In this section, we will show that the maximal functions discussed in the introduction can be expressed in the framework described in Section \ref{setupSection}. The Kakeya maximal function is straightforward: select $\mathcal{C}_0=[0,1]^2,$ $I_0 = [0,1]$, $\mathcal{C}$ a neighborhood of $\mathcal{C}_0$, and $I$ a neighborhood of $I_0$. Let $h(m,b;t) = mt+b$, and let $\Phi(m,b)= m$. Then $F^h_t(m, b) = \big(mt+b, m)$, and $DF^h_t = \binom{t, 1}{1, 0}$, which is invertible. Since $s=m-1$, $\Phi$ is automatically transverse to $h$.

\medskip

For the Wolff and Bourgain circular maximal functions, we can use translation and rotation symmetry to reduce to the case where $r$ takes values in a  neighborhood of $1$ and the centers $(x,y)$ take values in a neighborhood of $(0,0)$. Finally, we may restrict the integral \eqref{WolffMaximalFn} (resp. \eqref{bourgainMaxmlFn}) to the upper arc of $C(x,y, r)$ above the interval $[-\rho, \rho]$, for $\rho>0$ a small (fixed) quantity. With these reductions, define $\mathcal{C}$ to be a neighborhood of $(0,0,1)$; $I$ a neighborhood of $0$; $h(x,y,r;t) = \sqrt{r^2-(t-x)^2}+y$. Then it suffices to verify that $DF^h_0$ has full rank at $(x,y,r) = (0,0,1)$; this is a straightforward computation. 

For the Wolff circular maximal function, define $\Phi(x,y,r)= r$. we have $m=s+1$, and hence $\Phi$ is automatically transverse to $h$. For the Bourgain circular maximal function we have $\Phi(x,y,r)=(x,y)$, and thus we must verify that $D\Phi$ restricted to the manifold
\begin{equation*}
\begin{split}
V_{(0,0,1;0)} & = \{(x',y',r')\in \mathcal{C}\colon h(x',y',r'; 0) = 1, \partial_t h(x',y',r'; t)|_{t=0} = 0\}\\
& = \{(x',y',r')\in \mathcal{C}\colon y' = 0\ r' = 1- x'\}
\end{split}
\end{equation*}
has rank 1 at $(0,0,1)$. But this is evidently the case, since we can write this manifold as $\{ (t, 0, 1-t) \}$ for $t$ in a neighborhood of 0.

\medskip

Finally, we discuss the Erdo\smash{\u{g}}an elliptic maximal function. Given an ellipse with semi-major axis $a$, semi-minor axis $b$, center $(x,y)$, and rotation angle $\theta$, define
\begin{equation}\label{AThruF}
\begin{array}{ll}
A= a^2\sin^2\theta + b^2\cos^2\theta & 
B=2(b^2-a^2)\sin\theta\cos\theta \\
C=a^2\cos^2\theta+b^2\sin^2\theta &
D=-2Ax-By \\ 
E=-Bx-2Cy &
F=Ax^2+Bxy+Cy^2-a^2b^2.
\end{array}
\end{equation}

Then the corresponding ellipse is the locus of points $(X,Y)$ satisfying
\[
AX^2+BXY+CY^2+DX+EY+F=0.
\]

In light of the above, define
\[
h(a,b,x,y,\theta,t) = \frac{-(Bt+E) + \sqrt{(Bt+E)^2 - 4C(At^2+Dt+F)}}{2C}.
\]
Again, after a translation, rotation, and anisotropic rescaling, it suffices to consider the case where $\mathcal{C}$ is a neighborhood of $(2,1,0,0,0)$, i.e.~the semi-major axis has length close to 2, the semi-minor axis has length close to 1, the center is close to $(0,0)$, and the rotation is close to 0. With $A,\ldots,F$ as given by \eqref{AThruF}, $h$ is a function from $\mathcal{C}$ to $\RR$. The graph of $t\mapsto h(a,b,x,y,\theta;t)$ is the (upper half) of the ellipse with major axis $a$, minor axis $b$, center $(x,y)$, and rotation $\theta$. A direct computation shows that $DF^h_0(1,1,0,0,0)$ has non-zero determinant.

Next, we have $\Phi(a,b,x,y,\theta) = (x,y)$. We must verify that $D\Phi$ restricted to the manifold
\begin{equation*}
\begin{split}
V_{(1,1,0,0,0;0)} & = \{(a', b', x',y',\theta')\in \mathcal{C}\colon h(a', b', x',y',\theta';t) = 1/4,\\
& \ \partial_t h(a', b', x',y',\theta';t)|_{t=0} = 0,\ \partial^2_t h(a', b', x',y',\theta';t)|_{t=0}=-\sqrt 2,\\
& \qquad\qquad\qquad \partial^3_t h(a', b', x',y',\theta';t)|_{t=0}=0\}
\end{split}
\end{equation*}
has rank 1 at $(1,1,0,0,0).$ But in a neighborhood of $(1,1,0,0,0)$, this manifold can be written as $(1+a_1(t), 1+a_2(t), 0, b(t), a_3(t))$, where $a_1,a_2,a_3$ are smooth and satisfy $a_i(0)=0$, and $\partial_t b(t)|_{t=0}\sim 1$. Since $\Phi(a,b,x,y,\theta)=(x,y)$, we conclude that  $D\Phi$ restricted to $V_{(1,1,0,0,0;0)} $ has rank 1, as desired.

\subsection{The range of $p$ in Theorem \ref{LpMaximalFnBd} is sharp}\label{rangeOfPIsSharp}
In this section we will give an example showing that the range of $p$ in Theorem \ref{LpMaximalFnBd} is sharp. Define 
\[
h(x,y,w_1,\ldots,w_s;t)= y + w_1(t-x)^2 + w_2(t-x)^3 + w_2(t-x)^4 + \ldots + w_s(t-x)^{s+1}.
\] 
It is straightforward to show that every polynomial (in $t$) of degree at most $s+1$ can be uniquely expressed as $h(x,y,w_1,\ldots,w_s;t)$ for an appropriate choice of $x,y,w_1,\ldots,w_s$. 

For $\rho>0$ small, define
\[
f(x,y) = (y+\rho)^{-1/(s+1)}\chi_{[0,1]^2}.
\]
Then $\Vert f \Vert_{s+1}\sim (\log 1/\rho)^{1/(s+1)}$. On the other hand, for $(x,y)$ in a neighborhood of $(0,1)$, we can select $w$ so that the graph of $t\mapsto h(x,y,w_1,\ldots,w_s;t)$ is tangent to the $x$-axis to order $s$ at the origin, and hence
\[
\int_{\gamma_u}f \sim \log(1/\rho),
\]
where $\gamma_u$ is the graph of $t\mapsto h(x,y,w_1,\ldots,w_s;t)$ over $[0,1]$. Letting $\rho\searrow 0$, we conclude that the operator $M$ from \eqref{defnMaximalFnNoDelta} cannot be bounded from $L^p\to L^p$ for $p=s+1$. To show that no $L^p\to L^p$ bound is possible for $p<s+1$ is straightforward: let $h$ be as above, and let $f$ be the characteristic function of a $1\times \rho$ rectangle.

\section{Geometric lemmas}\label{geometricLemmasSection}
In this section we will record several computations that explore some of the consequences of curve-rectangle tangency. Our main tool will be Taylor's theorem with remainder. In a typical argument in this section, we will approximate a function $f$ by its $k$--th order Taylor polynomial, which we denote by $f_k$. To show that the function $f$ cannot be small on a large set, we will need the analogous result for $f_k$. The following inequalities will be useful for this purpose.

\begin{thm}[Remez inequality \protect{\cite{Reme}}]\label{remezInequalityThm}
Let $I\subset\RR$ be a finite interval, let $E\subset I$ be measurable, and let $P$ be a polynomial of degree at most $D$. Then
\[
\sup_{x\in I}|P(x)|\leq \Big(\frac{4|I|}{|E|}\Big)^D\sup_{x\in E}|P(x)|.
\]
\end{thm}

\begin{thm}[P\'olya inequality \protect{\cite{Polya}}]\label{PolyaIneq}
Let $P$ be a polynomial of degree at most $D$, with leading coefficient $a\in\RR$. Then for $\lambda>0$, we have 
\[
|\{x\in\RR \colon |P(x)|\leq\lambda \}|\leq 4\Big(\frac{\lambda}{2|a|}\Big)^{1/D}.
\]
\end{thm}

The next inequality says that if $f$ is small on a long interval, then the derivatives of $f$ must also be small on this interval.

\begin{lem}\label{smallImpliesSmallCkOnJ}
Let $k\geq 1$, $\delta>0$, and let $f\in C^k$ with $\Vert f \Vert_{C^k}\leq 1$. Let $I\subset[0,1]$ be a closed interval of length at most $\delta^{1/k}$, and suppose that $|f(x)|\leq\delta$ for $x\in I$. 

Then 
\begin{equation}\label{smallImpliesSmallCkOnJConclusion}
\sup_{x\in I}|f^{(i)}(x)|\lesssim\delta|I|^{-i},\quad i=0,\ldots,k-1.
\end{equation}
\end{lem}
\begin{proof}
First, we may suppose that $\delta^{1/(k-1)}<|I|\leq\delta^{1/k}$, since otherwise \eqref{smallImpliesSmallCkOnJConclusion} for $i=k-1$ follows from the assumption $\Vert f \Vert_{C^k}\leq 1$, and we may replace $k$ by $k-1$. 

Let $K=2\cdot 8^{k^2}k$. We will prove that \eqref{smallImpliesSmallCkOnJConclusion} holds with implicit constant $K$. Suppose not; then there exists an index $0\leq i< k$ so that
\begin{equation}\label{largeCiNorm}
\sup_{x\in I}|f^{(i)}(x)| > K 8^{-ik}\delta|I|^{-i}.
\end{equation}
Let $j$ be the largest index for which \eqref{largeCiNorm} holds. Select $x_0\in I$ with $|f^{(j)}(x_0)|=\sup_{x\in I}|f^{(j)}(x)|$. Define 
\[
Q(x) = f(x_0) + \sum_{i=1}^j \frac{(x-x_0)^i}{i!} f^{(i)}(x_0).
\]
 By P\'olya's sub-level set inequality (Theorem \ref{PolyaIneq}) with $\lambda = K8^{-(k+1)j}\delta/j!$, we have
\begin{equation}
|\{x\in I\colon |Q(x)| \leq \lambda \}| \leq 4\Big( \frac{ \lambda }{2 |f^{(j)}(x_0)|/j! } \Big)^{1/j}\leq 4\Big( \frac{K8^{-(k+1)j}\delta/j!}{ 2\cdot K8^{-jk}\delta|I|^{-j}/j!  } \Big)^{1/j}\leq \frac{1}{2}|I|.
\end{equation}
In particular, there exists a point $x\in I$ with $|Q(x)|\geq K8^{-(k+1)j} \delta/j!$. On the other hand, by Taylor's theorem there is a point $x_1$ between $x_0$ and $x$ so that 
\[
f(x) = Q(x) + \frac{(x-x_0)^{j+1}}{(j+1)!}f^{(j+1)}(x_1),
\]
and hence
\begin{equation}
\begin{split}
|f(x)| & \geq |Q(x)| - \frac{|x-x_0|^{j+1}}{(j+1)!}|f^{(j+1)}(x_1)|\\
&\geq \frac{K8^{-(k+1)j} \delta}{j!}  - \frac{|I|^{j+1}}{(j+1)!}\big( K8^{-(j+1)k} \delta|I|^{-j-1} \big)\\
& \geq K8^{-jk}\delta\Big( \frac{1}{8^jj!} - \frac{1}{8^k(j+1)!}\big) \geq \frac{K}{8^{k^2}k}\delta  >\delta.
\end{split}
\end{equation}
This contradicts the assumption that $|f(x)|\leq\delta$ on $I$.   We conclude that \eqref{smallImpliesSmallCkOnJConclusion} holds.
\end{proof}

The next result says that if $f$ is tangent to a $(\delta; k-1; T)$ rectangle $R$, then there is a corresponding value of $\rho\geq\delta$ (which depends on $\delta, k,$ and $T$) so that $f$ is tangent to a $(\rho;k)$ rectangle associated to $R$. 

\begin{lem}\label{longRectInsideKP1Rect}
Let $k\geq 1$, $\delta>0$, $T\in[1,1/\delta]$, and let $f\in C^{k}$ with $\Vert f \Vert_{C^k}\leq 1$. Let $I = [a,a+(T\delta)^{\frac{1}{k-1}}] \subset[0,1]$. Suppose that $|f(x)|\leq\delta$ for $x\in I$. 

Let $\rho = \max(\delta, T^{-k})$. Then $|f(x)|\lesssim\rho$ for $x\in [a, a+\rho^{1/k}]$. 
\end{lem}
\begin{proof}
If $T\geq\delta^{-1/k}$ then $(T\delta)^{\frac{1}{k-1}}\geq\delta^{1/k}$ the conclusion is immediate. 

Suppose instead that $T<\delta^{-1/k}$, and thus $\rho=T^{-k}$. Since $\Vert f \Vert_{C^k}\leq 1$ and $|I|\leq\delta^{1/k}$, we can apply Lemma \ref{smallImpliesSmallCkOnJ} to conclude that
\begin{equation}\label{fMinusF0Bound}
\sup_{x\in I}|f^{(i)}(x)|\lesssim\delta|I|^{-i} = \delta^{1-\frac{i}{k-1}}\rho^{\frac{i}{k(k-1)}},\quad i=0,\ldots,k-1.
\end{equation}

Since $\Vert f \Vert_{C^k}\leq 1$, we can apply Taylor's theorem to conclude that for each $x\in [a, a+\rho^{1/k}]$, there is a point $x_1$ between $a$ and $x$ so that
\[
f(x) = \sum_{i=0}^{k-1} \frac{f^{(i)}(a)}{i!}(x-a)^i + \frac{f^{k}(x_1)}{(k)!}(x-a)^{k},
\]
and hence by \eqref{fMinusF0Bound} (and noting that $\rho\geq\delta)$,
\begin{equation*}
\begin{split}
|f(x)|
\lesssim  \sum_{i=0}^{k-1}   \delta^{1-\frac{i}{k-1}}\rho^{\frac{i}{k(k-1)}} \rho^{\frac{i}{k}} + \rho^{k/k}
= \sum_{i=0}^{k-1}   \delta(\rho/\delta)^{\frac{i}{k-1}} + \rho
\lesssim \rho.\qedhere
\end{split}
\end{equation*}
\end{proof}

The next result says that if a set of functions are all tangent to a common curvilinear rectangle of dimensions $A\delta\times \delta^{1/k}$, then a large fraction of these functions must be tangent to a common curvilinear rectangle of dimensions $\delta\times\delta^{1/k}$. The proof is an application of pigeonholing, and is omitted. 
\begin{lem}\label{thickenedVsRegularRectangleLem}
Let $k\geq 1, A\geq 1$. Then there exists $\eps>0$ so that the following holds. Let $F$ be a set of functions with $C^k$ norm at most 1. Suppose that there is an interval $I\subset[0,1]$ of length $\delta^{1/k}$ so that $\sup_{x\in I}|f(x)-g(x)|\leq A \delta$ for all $f,g\in F$. Then there is a set $F'\subset F$ of cardinality at least $\eps(\#F)$ so that $\sup_{x\in I}|f(x)-g(x)|\leq  \delta$ for all $f,g\in F'$. 
\end{lem}

The next result records a useful property of families of curves that forbid $k$--th order tangency: if two functions from this family are both tangent to a common $(\delta; k; T)$ rectangle for $T\geq 1$, then these functions must be close in $C^k$ norm. The proof is similar to that of Lemma \ref{smallImpliesSmallCkOnJ}, and is omitted. 
\begin{lem}\label{smallImpliesSmallCkOnJCinematic}
Let $k\geq 2$, $\delta>0$, $K\geq 1.$ Let $I$ be a compact interval and let $f\in C^k(I)$ with $\Vert f \Vert_{C^k(I)}\leq 1$. Let $J\subset I$ be a closed interval of length at most 1. Suppose that 
\begin{equation}\label{fxLessDelta}
\sup_{x\in J}|f(x)|\leq\delta,
\end{equation}
and 
\begin{equation}\label{CkNormControlledInf}
\Vert f \Vert_{C^k(I)}\leq K\inf_{x\in I}\sum_{j=0}^k|f^{(j)}(x)|.
\end{equation}
Then 
\begin{equation}\label{smallImpliesSmallCkOnJConclusion2}
\Vert f \Vert_{C^k(I)}\lesssim (K/|J|)^k \delta.
\end{equation}
\end{lem}


\bibliographystyle{plain}
\bibliography{references}

\end{document}